\documentclass[12pt]{amsart}
\usepackage{amssymb}
\usepackage{amscd}
\usepackage{amsthm} 
\usepackage{amssymb, graphicx,epic,eepic} 
\usepackage{color,graphics} 

\newtheorem{theorem}{Theorem}[section]
\newtheorem{proposition}[theorem]{\bf Proposition}
\newtheorem{lemma}[theorem]{\bf Lemma}
\newtheorem{corollary}[theorem]{\bf Corollary}

\theoremstyle{remark}
\newtheorem{definition}[theorem]{\bf Definition}

\newtheorem{remark}[theorem]{\bf Remark}

\title{Rank two prolongations of second-order PDE and geometric singular solutions}
\author{Takahiro Noda 
and Kazuhiro Shibuya}
\address{Takahiro Noda\\
Graduate School of Mathematics\\ 
Nagoya University
Chikusa-ku \\Nagoya 464-8602 Japan\\
m04031x@math.nagoya-u.ac.jp}
\address{Kazuhiro Shibuya\\
Graduate School of Science\\ 
Hiroshima University \\
Higashi-Hiroshima\\ 739-8521, 
Japan\\
shibuya@hiroshima-u.ac.jp
}

\begin{document}

\maketitle

\begin{abstract} 
In this present paper, we study geometric structures 
of rank two prolongations of implicit second-order 
partial differential equations (PDEs) for two independent and
one dependent variables and characterize the type of these PDEs by 
the topology of fibers of the rank two prolongations. 
Moreover, by using properties of these 
prolongations, we give explicit expressions of 
geometric singular solutions of second-order PDEs 
from the point of view of contact geometry of second order. \\


\end{abstract}
\section{Introduction}
Let us start by recalling the geometric construction of the 
2-jet bundle for two independent and one dependent variables, following \cite{Y1} and \cite{Y2}. 

First, let $M$ be a manifold of dimension $3$. We consider the space of 
$2$-dimensional contact elements to $M$, i.e., 
the Grassmann bundle $J(M,2)$ over 
$M$ consisting of $2$-dimensional subspaces of tangent spaces to $M$, namely, 
$J(M,2)$ is defined by
$$
J(M,2) = \bigcup_{x\in M} J_x,\qquad J_x = \text{Gr}(T_x(M),2),
$$
where $\text{Gr}(T_x(M),2)$ denotes the Grassmann manifold of $2$-dimensional subspaces 
in $T_x(M)$. Let $\pi: J(M,2)\to M$ be the bundle projection. 
The canonical system 
$C$ on $J(M,2)$ is, by definition, the differential system of codimension $1$ 
on $J(M,2)$ defined by
$$
C(u) = \pi_*^{-1}(u) = \{v\in T_u(J(M,2))\ |\ \pi_*(v)\in u \}\subset T_u(J(M,2)) 
\overset{\pi_{*}}\longrightarrow T_x(M),
$$
where $\pi(u) = x$ for $u\in J(M,2)$. 
The differential system $(J(M,2),C)$ is the (geometric) $1$-jet space, also called contact manifold of dimension $5$. 
In general, by a differential system $(R,D)$,  
we mean a distribution $D$ on a manifold $R$, that is,  
$D$ is a subbundle of the tangent bundle $TR$ of $R$. 

Next, we should start from a contact manifold $(J,C)$ of dimension $5$, 
which is locally a space of $1$-jet 
for two independent and one dependent variables. 
Then we can construct the geometric second-order jet space $(L(J),E)$ as follows: We consider the Lagrange-Grassmann bundle $L(J)$ over $J$ consisting of all 
$2$-dimensional integral elements of $(J,C)$, namely, 
$$
L(J) = \bigcup_{u\in J} L_u \subset J(J,2),
$$
where $L_u$ is the Grassmann manifold of all Lagrangian (or Legendrian) subspaces of the symplectic vector space $(C(u),d\varpi)$ for any $u \in J$. 
Here $\varpi$ is a local contact form on $J$. 
Namely, $v\in J(J,2)$ is an integral element if and only if $v\subset C(u)$ 
and $d\varpi|_v = 0$, where $u= \pi(v)$. Then the canonical system $E$ on $L(J)$ is defined by
$$
E(v) = \pi_*^{-1}(v) \subset T_v(L(J))\overset{\pi_{*}}\longrightarrow T_u(J),
$$
where $\pi(v) = u$ for $v\in L(J)$ and $\pi: L(J)\to J$ is the projection.
The geometric jet space of second order $(L(J),E)$ is locally a space of 2-jets for two independent and one dependent variables $(J^{2}(\mathbb R^2, \mathbb R), C^2)$. Here, the 2-jet space $(J^{2}(\mathbb R^2, \mathbb R), C^2)$ is defined as follows:
\begin{equation}
J^{2}(\mathbb R^2, \mathbb R):=\left\{(x,y,z,p,q,r,s,t)\right\}
\end{equation}
and 
$C^2:=\left\{\varpi_{0}=\varpi_{1}=\varpi_{2}=0\right\}$ is given by the 
following 1-forms:
\begin{align*}
\varpi_{0}:=dz-pdx-qdy,\quad
\varpi_{1}:=dp-rdx-sdy,\quad 
\varpi_{2}:=dq-sdx-tdy.
\end{align*}
In this paper, we identify $(L(J),E)$ with $(J^{2}(\mathbb R^2, \mathbb R), C^2)$ since we only consider the local geometry of jet spaces.

Now we consider single PDEs  
$F(x,y,z,p,q,r,s,t)=0$, 
where $F$ is a smooth function on 
$J^{2}(\mathbb R^2, \mathbb R)$. 
We set 
$R=\left\{F=0 \right\} \subset J^2(\mathbb R^2, \mathbb R),$ 
and restrict the canonical differential system $C^2$ to $R$. We denote 
it by $D(:=C^2 |_R)$. 
We consider a PDE $R=\{F=0\}$ 
with the condition 
$(F_r,F_s,F_t)\not=(0,0,0)$ 
which we will call the {\it regularity condition}. 
Thus, $R$ is a smooth hypersurface, and also the restriction
$\pi^{2}_{1}|_{R}:R \to J^1(\mathbb R^2, \mathbb R)$  of the natural  
projection  $\pi^{2}_{1}:J^2(\mathbb R^2, \mathbb R) \to J^1(\mathbb R^2, \mathbb R)$ 
is a submersion. 
Due to the regularity condition, restricted 1-forms $\varpi_{i}|_{R}$ ($i=0,1,2$) on $R$ 
are linearly independent. Therefore, we have the induced differential system 
$D=\left\{\varpi_{0}|_{R}=\varpi_{1}|_{R}=\varpi_{2}|_{R}=0\right\}$
on $R$. 
Then, $D$ is a vector bundle of rank 4 on $R$. 
For brevity, we denote each restricted generator 1-form $\varpi_{i}|_{R}$ 
of $D$ by $\varpi_{i}$ in the following. 
For such an equation $F=0$, we consider the discriminant 
$\Delta:=F_rF_t-{F_s}^2/4$.  
\begin{definition}
Let $R=\left\{F=0\right\}$ be a single second-order regular PDE. For the discriminant $\Delta$ of $F$,   
a point $w\in R$ is said to be {\it hyperbolic} or {\it elliptic} if 
$\Delta(w)<0$ or $\Delta(w)>0$, respectively. Moreover, a point 
$w\in R$ is said to be {\it parabolic} if $(F_r,F_s,F_t)_{w}\not=(0,0,0)$ and 
$\Delta(w)=0$. 
\end{definition}
For second-order regular PDEs, we are interested in geometric singular solutions. Here, the 
notion of geometric solutions, including singular solutions is 
defined as follows (see \cite{NS}).  
\begin{definition}\label{solution}
Let $(R, D)$ be a second-order regular PDE. 
For a 2-dimensional integral manifold $S$ of $R$,  
if the restriction $\pi^{2}_{1}|_{R}: R\to J^1$ of the natural projection $\pi^{2}_{1}:J^2\to J^1$ is an immersion 
on an open dense subset in $S$,  
then we call $S$  a {\it geometric solution} of $(R, D)$. If all points of 
a geometric solution $S$ are immersion points, 
then we call $S$ a {\it regular solution}. 
On the other hand, a geometric solution $S$ have a 
nonimmersion point, then we call $S$ a {\it singular solution}.  
\end{definition}
From the definition, images $\pi^{2}_{1}(S)$ of geometric solutions $S$ by the 
projection $\pi^{2}_{1}$ are Legendrian in $J^1(\mathbb R^2,\mathbb R)$, i.e., 
$\varpi_{0}|_{\pi^{2}_{1}(S)}=d\varpi_{0}|_{\pi^{2}_{1}(S)}=0$. 
We will investigate the method of the construction 
of these singular solutions. For this purpose, 
we define the notion of rank $n$ prolongations of differential systems, in general, as follows:
\begin{definition}\label{rank 2-prolongation}
Let $(R,D)$ be a differential system given by 
$D=\{\varpi_1 =\cdots =\varpi_s =0 \}$. 
An $n$-dimensional {\it integral element} of $D$ at $x\in R$ is an 
$n$-dimensional subspace $v$ of $T_x R$ such that 
\[
\varpi_i|_{v} = d\varpi_i|_{v}=0\qquad (i=1,\cdots ,s).
\]
Namely, $n$-dimensional integral elements are candidates for 
the tangent spaces at $x$ to $n$-dimensional integral manifolds of $D$.
It follows that the {\it rank $n$ prolongation 
of $(R,D)$} is defined by 
\begin{equation}\label{Carpro}
\Sigma(R):=\bigcup_{x\in R}\Sigma_{x}, 
\end{equation}
where 
$\Sigma_x=\left\{v \subset T_{x}R\ |\ v\ {\rm is\ an}\ n {\rm \textit{-}dimensional\ integral\ element\ of\ } (R,D)\ \right\}$. 
We define the canonical system 
$\hat D$ on $\Sigma(R)$ by 
\begin{align}\label{canonical system}
\hat D(u):&={p_{*}^{-1}}(u) 
          =\left\{v\in T_{u}(\Sigma(R))\ |\ p_{*}(v)\in u\right\}, \nonumber
\end{align}
where $u\in \Sigma(R)$ is a smooth point of $\Sigma(R)$ and $p:\Sigma(R)\to R$ is the projection.
\end{definition}
This space $\Sigma(R)$ is a subset of the Grassmann bundle over $R$
\begin{equation}\label{Grassmann}
J(D, n):=\bigcup_{x\in R} J_x
\end{equation}
where 
$J_x:=\left\{v \subset T_{x}R\ |\ v\ {\rm is\ an}\ n{\rm \textit{-}dimensional\ subspace\ of}\ D(x)\right\}.$  
In general, the rank $n$ prolongations $\Sigma(R)$ 
have singular points, that is, 
$\Sigma(R)$ is not a smooth manifold. 
This kind of prolongation is very useful to 
study geometric structures of equations $(R,D)$ or their solutions. 
In this paper, we only consider in the case of $n=2$. \par
Let us now proceed to the description of the various sections and explain the main results in the present paper. 
In section 2, we investigate the fiber topology of rank 2 prolongations 
$(\Sigma(R),\hat{D})$ of differential systems 
$(R,D)$ induced by hyperbolic, parabolic and elliptic equations.
One of the main results of this paper is that the type of equations defined by 
local structure is characterized by the topology of fibers of 
the prolongation $p : \Sigma(R) \to R$. Namely, 
we obtain that the topology of fibers of 
the prolongations $p : \Sigma(R) \to R$ of differential systems 
$(R,D)$ associated with hyperbolic, parabolic or elliptic equations 
is torus, pinched torus or sphere, respectively 
(Corollary \ref{all-topology}). 
In section 3, we study structures of the canonical systems $\hat{D}$ 
on the rank 2 prolongations $\Sigma(R)$ for hyperbolic, parabolic and 
elliptic equations $(R,D)$ as differential systems. More precisely, obtained results in this 
section clarify the structure of 
nilpotent graded Lie algebras (symbol algebras) of the canonical systems  
on the rank 2 prolongations for hyperbolic, parabolic and 
elliptic equations. Here, the symbol algebra is 
a fundamental invariant of differential systems under 
contact transformations (see section 3.2). 
In section 4, we research an approach to 
construct geometric singular solutions of  
hyperbolic, parabolic, elliptic equations defined by 
Definition \ref{solution}. Especially, we give the explicit integral representation of these singular solutions of model equations for each class of single equations.   
In section 5, we introduce hyperbolic, parabolic and elliptic rank 4 distributions which are generalizations of hyperbolic, parabolic and elliptic PDEs
and prove the topology of fibers of the prolongation of these rank 4 distributions is torus, pinched torus or sphere, 
respectively (Proposition \ref{tower}).
This result is a generalization of a part of Theorem 18 in \cite{KL}. We also prove that the procedure of prolongations of these distributions preserves their types, namely, the rank $2$ prolongation of hyperbolic, parabolic or elliptic rank 4 distributions
is also a rank 4 distribution of the type of hyperbolic, parabolic or elliptic, respectively (Theorem \ref{towerthm}). 
It follows that, by successive prolongations of these rank 4 distributions, 
we can define the notion of $k$-th rank $2$ prolongations 
as a generalization of 
$k$-th rank $1$ prolongations introduced 
previously in \cite{MZ}  or \cite{SY} (these are called \lq\lq Monster Goursat manifolds" in \cite{MZ}).\par
{\bf Acknowledgments.}  
We would like to express our special thanks to 
Professor Keizo Yamaguchi for many useful discussions on the subject.
Moreover, the first author is supported by Osaka 
City University Advanced Mathematical Institute. 
\section{Rank 2 prolongations of regular PDEs} 
In this section, we show that the type of equations 
is characterized by the topology of fibers of the rank 2 prolongations
of equations. 
For this purpose, we provide the rank 2 prolongations of hyperbolic, parabolic and elliptic PDEs by using inhomogeneous Grassmann coordinates. \par
\medskip\noindent     
{\bf Rank 2 prolongations of hyperbolic equations.}
Let $(R,D)$ be a locally hyperbolic equation. Then, 
there exists a local coframe 
$\left\{\varpi_0, \varpi_1, \varpi_2, \omega_1, \omega_2, \pi_{11}, \pi_{22}\right\}$  around $x\in R$ such that 
$D=\{\varpi_0=\varpi_1=\varpi_2=0 \}$ 
and the following structure equation holds:
\begin{align}\label{hypequ}
d\varpi_{0}&\equiv \omega_{1}\wedge\varpi_{1}+\omega_{2}\wedge\varpi_{2}
\quad \mod \ \varpi_{0}, \nonumber \\
d\varpi_{1}&\equiv \omega_{1}\wedge \pi_{11} \quad \mod\ \varpi_{0}, \varpi_{1}, \varpi_{2},\\
d\varpi_{2}&\equiv \omega_{2}\wedge \pi_{22} \quad \mod\ \varpi_{0}, 
\varpi_{1}, \varpi_{2}. \nonumber
\end{align}
In terms of this structure equation, we construct the rank 2 prolongation of $(R,D)$ 
by taking integral elements. 
\begin{theorem}\label{hyp-pro} 
Let $(R,D)$ be a locally hyperbolic equation. Then, the 
rank $2$ prolongation $\Sigma(R)$ is a smooth submanifold of $J(D,2)$, 
and it is a $T^2=S^1\times S^1$-bundle over $R$.  
\end{theorem}
\begin{proof}
First, we show that $\Sigma(R)$ is a submanifold of $J(D,2)$.  
Let $\pi:J(D,2)\to R$ be the projection and $U$ an open set in $R$. 
Then $\pi^{-1}(U)$ is covered by $6$ open sets in $J(D,2)$: 
\begin{equation}\label{Grassmann-cover}
\pi^{-1}(U)=U_{\omega_{1}\omega_{2}}\cup U_{\omega_{1}\pi_{11}}
\cup U_{\omega_{1}\pi_{22}}\cup U_{\omega_{2}\pi_{11}}
\cup U_{\omega_{2}\pi_{22}}\cup U_{\pi_{11}\pi_{22}}, 
\end{equation}
where 
\begin{align*}
&U_{\omega_{1}\omega_{2}}:=\left\{v \in \pi^{-1}(U) \ |\ \omega_{1}|_{v}\wedge \omega_{2}|_{v} \not=0\right\},\ 
U_{\omega_{1}\pi_{11}}:=\left\{v \in \pi^{-1}(U) \ |\ \omega_{1}|_{v}\wedge \pi_{11}|_{v} \not=0\right\},\\
&U_{\omega_{1}\pi_{22}}:=\left\{v \in \pi^{-1}(U) \ |\ \omega_{1}|_{v}\wedge \pi_{22}|_{v} \not=0\right\},\  
U_{\omega_{2}\pi_{11}}:=\left\{v \in \pi^{-1}(U) \ |\ \omega_{2}|_{v}\wedge \pi_{11}|_{v} \not=0\right\},\\
&U_{\omega_{2}\pi_{22}}:=\left\{v \in \pi^{-1}(U) \ |\ \omega_{2}|_{v}\wedge \pi_{22}|_{v} \not=0\right\},\  
U_{\pi_{11}\pi_{22}}:=\left\{v \in \pi^{-1}(U) \ |\ \pi_{11}|_{v}\wedge \pi_{22}|_{v} \not=0\right\}. 
\end{align*}
In the following, we explicitly describe the defining 
equation of $\Sigma(R)$ in terms of the inhomogeneous Grassmann coordinate of fibers  
in $U_{\omega_{1}\omega_{2}},...., U_{\pi_{11}\pi_{22}}$.\\
(I) On $U_{\omega_{1}\omega_{2}}$:\par
For $w\in U_{\omega_{1}\omega_{2}}$, 
$w$ is a $2$-dimensional subspace of $D(v)$, 
$p(w)=v$. Hence, by restricting $\pi_{11},\pi_{22}$ to $w$, we can introduce the 
inhomogeneous coordinate $p^{1}_{ij}$ of fibers of $J(D,2)$ around $w$ 
with  
$\pi_{11}|_{w}={p_{11}^{1}}(w)\omega_{1}|_{w}+{p_{12}^{1}}(w)\omega_{2}|_{w},\   
\pi_{22}|_{w}={p_{21}^{1}}(w)\omega_{1}|_{w}+{p_{22}^{1}}(w)\omega_{2}|_{w}.$
Moreover, $w$ satisfies 
$d\varpi_{1}|_{w}\equiv d\varpi_{2}|_{w}\equiv 0$: 
\begin{align*}
d\varpi_{1}|_{w}&\equiv \omega_{1}|_{w}\wedge 
                   ({p_{11}^{1}}(w)\omega_{1}|_{w}+{p_{12}^{1}}(w)\omega_{2}|_{w})
                \equiv {p_{12}^{1}}(w)\omega_{1}|_{w}\wedge\omega_{2}|_{w},\\
d\varpi_{2}|_{w}&\equiv \omega_{2}|_{w}\wedge 
                   ({p_{21}^{1}}(w)\omega_{1}|_{w}+{p_{22}^{1}}(w)\omega_{2}|_{w})
                \equiv -{p_{21}^{1}}(w)\omega_{1}|_{w}\wedge\omega_{2}|_{w}. 
\end{align*} 
Hence, we obtain the defining equations $f_1=f_2=0$ of $\Sigma(R)$ in  $U_{\omega_{1}\omega_{2}}$ of $J(D,2)$, where 
$f_1=p_{12}^{1},\ f_2=p_{21}^{1}$, that is,  
$\left\{f_1=f_2=0 \right\}\subset U_{\omega_{1}\omega_{2}}.$
Then $df_1, df_2$ are independent on $\left\{f_1=f_2=0\right\}$.\\ 
(II) On $U_{\omega_{1}\pi_{11}}$:\par
For $w\in U_{\omega_{1}\pi_{11}}$, 
by restricting $\omega_2, \pi_{22}$ to $w$, we 
introduce the inhomogeneous coordinate $p^{2}_{ij}$ 
of fibers of $J(D,2)$ around $w$ with  
$\omega_{2}|_{w}={p_{11}^{2}}(w)\omega_{1}|_{w}+{p_{12}^{2}}(w)\pi_{11}|_{w},\quad 
\pi_{22}|_{w}={p_{21}^{2}}(w)\omega_{1}|_{w}+{p_{22}^{2}}(w)\pi_{11}|_{w}.$ 
Moreover, $w$ satisfies $d\varpi_{1}|_{w}\equiv d\varpi_{2}|_{w}\equiv 0$.  
However, we have 
$d\varpi_{1}|_{w}\equiv \omega_{1}|_{w}\wedge \pi_{11}|_{w}\not \equiv 0.$ 
Thus, there does not exist integral element, that is, 
$U_{\omega_{1}\pi_{11}}\cap p^{-1}(U)=\emptyset$. \\
(III) On $U_{\omega_{1}\pi_{22}}$:\par 
For $w\in U_{\omega_{1}\pi_{22}}$, 
by restricting $\omega_{2}, \pi_{11}$ to $w$, we 
introduce the inhomogeneous coordinate $p^{3}_{ij}$ 
of fibers of $J(D,2)$ around $w$ with
$\omega_{2}|_{w}={p_{11}^{3}}(w)\omega_{1}|_{w}+{p_{12}^{3}}(w)\pi_{22}|_{w},\quad 
\pi_{11}|_{w}={p_{21}^{3}}(w)\omega_{1}|_{w}+{p_{22}^{3}}(w)\pi_{22}|_{w}.$  
Moreover, $w$ satisfies $d\varpi_{1}|_{w}\equiv d\varpi_{2}|_{w}\equiv 0$: 
\begin{align*}
d\varpi_{1}|_{w}&\equiv \omega_{1}|_{w}\wedge \pi_{11}|_{w} 
                 \equiv {p_{22}^{3}}(w)\omega_{1}|_{w}\wedge \pi_{22}|_{w},\\
d\varpi_{2}|_{w}&\equiv \omega_{2}|_{w} \wedge \pi_{22}|_{w}
                \equiv {p_{11}^{3}}(w) \omega_{1}|_{w} \wedge \pi_{22}|_{w}.                
\end{align*} 
Then the defining functions of $\Sigma(R)$ are independent in the same 
as (I).\\  
(IV) On $U_{\omega_{2}\pi_{11}}$:\par
For $w\in U_{\omega_{2}\pi_{11}}$, 
by restricting $\omega_{1}, \pi_{22}$ to $w$, we 
introduce the inhomogeneous coordinate $p^{4}_{ij}$ 
of fibers of $J(D,2)$ around $w$ with 
$\omega_{1}|_{w}={p_{11}^{4}}(w)\omega_{2}|_{w}+{p_{12}^{4}}(w)\pi_{11}|_{w},\  
\pi_{22}|_{w}={p_{21}^{4}}(w)\omega_{2}|_{w}+{p_{22}^{4}}(w)\pi_{11}|_{w}.$  
Moreover, $w$ satisfies $d\varpi_{1}|_{w}\equiv d\varpi_{2}|_{w}\equiv 0$: 
\begin{align*}
d\varpi_{1}|_{w}&\equiv \omega_{1}|_{w} \wedge \pi_{11}|_{w} 
                 \equiv {p_{11}^{4}}(w) \omega_{2}|_{w} \wedge \pi_{11}|_{w},\\
d\varpi_{2}|_{w}&\equiv \omega_{2}|_{w} \wedge \pi_{22}|_{w}
                \equiv {p_{22}^{4}}(w)\omega_{2}|_{w} \wedge \pi_{11}|_{w}. 
\end{align*}   
Then the defining functions of $\Sigma(R)$ are independent in the same 
way as in (I).\\
(V) On $U_{\omega_{2}\pi_{22}}$:\par 
For $w\in U_{\omega_{2}\pi_{22}}$, 
by restricting $\omega_{1}, \pi_{11}$ to $w$, we 
introduce the inhomogeneous coordinate $p^{5}_{ij}$ 
of fibers of $J(D,2)$ around $w$ with
$\omega_{1}|_{w}={p_{11}^{5}}(w)\omega_{2}|_{w}+{p_{12}^{5}}(w)\pi_{22}|_{w},\quad 
\pi_{11}|_{w}={p_{21}^{5}}(w)\omega_{2}|_{w}+{p_{22}^{5}}(w)\pi_{22}|_{w}.$  
Moreover, $w$ satisfies $d\varpi_{1}|_{w}\equiv d\varpi_{2}|_{w}\equiv 0$. 
However, we have 
$d\varpi_{2}|_{w}\equiv \omega_{2}|_{w}\wedge \pi_{22}|_{w}\not\equiv 0.$  
Thus, there does not exist integral element, that is,  
$U_{\omega_{2}\pi_{22}}\cap p^{-1}(U)=\emptyset$.\\
(VI) On $U_{\pi_{11}\pi_{22}}$:\\
For $w\in U_{\pi_{11}\pi_{22}}$, 
by restricting $\omega_{1}, \omega_{2}$ to $w$, we 
introduce the inhomogeneous coordinate $p^{6}_{ij}$ 
of fibers of $J(D,2)$ around $w$ with
$\omega_{1}|_{w}={p_{11}^{6}}(w)\pi_{11}|_{w}+{p_{12}^{6}}(w)\pi_{22}|_{w},\  
\omega_{2}|_{w}={p_{21}^{6}}(w)\pi_{11}|_{w}+{p_{22}^{6}}(w)\pi_{22}|_{w}.$  
Moreover, $w$ satisfies $d\varpi_{1}|_{w}\equiv d\varpi_{2}|_{w}\equiv 0$: 
\begin{align*}
d\varpi_{1}|_{w}&\equiv \omega_{1}|_{w} \wedge \pi_{11}|_{w}
                 \equiv {p_{12}^{6}}(w) \pi_{22}|_{w}\wedge \pi_{11}|_{w},\\
d\varpi_{2}|_{w}&\equiv \omega_{2}|_{w}\wedge \pi_{22}|_{w}
                \equiv {p_{21}^{6}}(w)\pi_{11}|_{w}\wedge \pi_{22}|_{w}. 
\end{align*} 
Then the defining functions of $\Sigma(R)$ are independent 
in the same way as in (I).\par 
Under these discussions, 
the rank 2 prolongation $\Sigma(R)$ is a smooth submanifold of $J(D,2)$.\par   
Next, we show that the topology of fibers of $\Sigma(R)$ is torus. 
In the above discussion, we have 
$p^{-1}(U)=P_{\omega_{1}\omega_{2}}\cup P_{\omega_{1}\pi_{22}}\cup 
          P_{\omega_{2}\pi_{11}}\cup P_{\pi_{11}\pi_{22}}$,  
where 
$P_{\omega_{1}\omega_{2}}:=p^{-1}(U) \cap U_{\omega_{1}\omega_{2}},\   
P_{\omega_{1}\pi_{22}}:=p^{-1}(U) \cap U_{\omega_{1}\pi_{22}},\  
P_{\omega_{2}\pi_{11}}:=p^{-1}(U) \cap U_{\omega_{2}\pi_{11}},$ and  
 $P_{\pi_{11}\pi_{22}}:=p^{-1}(U) \cap U_{\pi_{11}\pi_{22}}$.
From Definition \ref{rank 2-prolongation}, we have the canonical system 
$\hat D$ on each open set. To prove our assertion, we investigate  
the gluing of $(\Sigma(R),\hat D)$. 
For instance, we construct the transition functions on 
$U_{\omega_{1}\omega_{2}}\cap U_{\omega_{1}\pi_{22}}$ in the following. 
On $U_{\omega_{1}\omega_{2}}$, the canonical system 
$\hat D=\left\{\varpi_{0}=\varpi_{1}=\varpi_{2}
=\varpi_{\pi_{11}}=\varpi_{\pi_{22}}=0\right\}$ 
is given by 
$\varpi_{\pi_{11}}:=\pi_{11}-{p_{11}^{1}}\omega_{1},\quad  
\varpi_{\pi_{22}}:=\pi_{22}-{p_{22}^{1}}\omega_{2}.$ 
On the other hand, the canonical system 
$\hat D=\left\{\varpi_{0}=\varpi_{1}=\varpi_{2}=
\varpi_{\omega_{2}}=\varpi_{\pi_{11}}=0\right\}$ 
on $U_{\omega_{1}\pi_{22}}$ is given by 
$\varpi_{\omega_{2}}:=\omega_{2}-{p_{12}^{3}}\pi_{22},\quad  
\varpi_{\pi_{11}}:=\pi_{11}-{p_{21}^{3}}\omega_{1}.$ 
Then, the transition functions  
$\phi: U_{\omega_{1}\omega_{2}}\cap U_{\omega_{1}\pi_{22}}\to U_{\omega_{1}\omega_{2}}\cap U_{\omega_{1}\pi_{22}}$ 
is given by 
\begin{equation*}
\phi(v,{p_{11}^{1}},{p_{22}^{1}})=\left(v,{p_{12}^{3}}:=\frac{1}{{p_{22}^{1}}},{p_{21}^{3}}:={p_{11}^{1}}\right)
\ {\rm for}\ {p_{22}^{1}}\not=0, 
\end{equation*}
where $v$ is a local coordinate on $R$. 
We also have similar transition functions for the  
other intersection open sets $U_{\omega_{1}\omega_{2}}\cap U_{\omega_{2}\pi_{11}}$, 
$U_{\omega_{1}\pi_{22}}\cap U_{\pi_{11}\pi_{22}}$, $U_{\omega_{2}\pi_{11}}\cap U_{\pi_{11}\pi_{22}}$. 
Consequently, the topological structure of fibers is $T^2=S^1\times S^1$.  
\end{proof}
\begin{remark}
In fact, this result (i.e. $\Sigma(R)$ is a torus bundle) is known by Bryant, 
Griffiths and Hsu in \cite{BGH}. They obtained this result for the hyperbolic 
exterior differential system which is a generalization of distributions corresponding to hyperbolic equations (see Remark \ref{EDS}).
However, we will also consider parabolic and elliptic cases and our method is distinct one. 
We will use the structure of this covering in $\Sigma(R)$ when we will study singular solutions (see, section 5). 
Thus, we need to prove in the above way. 
\end{remark}
\noindent
{\bf Rank 2 prolongations of parabolic equations.} 
Let $(R,D)$ be a locally parabolic equation. Then, 
there exists a local coframe 
$\left\{\varpi_0, \varpi_1, \varpi_2, \omega_1, \omega_2, \pi_{12}, \pi_{22}\right\}$  around $x\in R$ such that 
$D=\{\varpi_0=\varpi_1=\varpi_2=0 \}$ 
and the following structure equation holds: 
\begin{align}\label{parequ}
d\varpi_{0}&\equiv \omega_{1}\wedge\varpi_{1}+\omega_{2}\wedge\varpi_{2}
\quad \mod \ \varpi_{0}, \nonumber \\
d\varpi_{1}&\equiv \quad \quad \quad\ \quad \omega_{2}\wedge \pi_{12} \quad \mod\ \varpi_{0}, \varpi_{1}, \varpi_{2},\\
d\varpi_{2}&\equiv \omega_{1}\wedge \pi_{12}+\omega_{2}\wedge \pi_{22} \ \mod\ \varpi_{0}, 
\varpi_{1}, \varpi_{2}. \nonumber
\end{align}
From this structure equation, we clarify the rank 2 prolongation 
$\Sigma(R)$. 
\begin{lemma}\label{par-pro} 
Let $(R,D)$ be a locally parabolic equation. Then, the 
rank $2$ prolongation $\Sigma(R)$ has singular points.   
\end{lemma}
\begin{proof}
Let $U$ be an open set in $R$, and $\pi:J(D,2)\to R$ the 
projection. 
Then $\pi^{-1}(U)$ is covered by $6$ open sets in $J(D,2)$: 
\begin{equation}
\pi^{-1}(U)=U_{\omega_{1}\omega_{2}}\cup U_{\omega_{1}\pi_{12}}
\cup U_{\omega_{1}\pi_{22}}\cup U_{\omega_{2}\pi_{12}}
\cup U_{\omega_{2}\pi_{22}}\cup U_{\pi_{12}\pi_{22}}, 
\end{equation}
where each open set is given in the same way as the 
hyperbolic case (\ref{Grassmann-cover}).  
Now we explicitly describe the defining 
equation of $\Sigma(R)$ on each open set.\\  
(I) On $U_{\omega_{1}\omega_{2}}$:\par
For $w\in U_{\omega_{1}\omega_{2}}$, 
$w$ is a $2$-dimensional subspace of $D(v)$, 
$p(w)=v$. Hence, 
by restricting $\pi_{12},\pi_{22}$ to $w$, we can introduce the 
inhomogeneous coordinate $p^{1}_{ij}$ of fibers of $J(D,2)$ 
around $w$ with  
$\pi_{12}|_{w}={p_{11}^{1}}(w)\omega_{1}|_{w}+{p_{12}^{1}}(w)\omega_{2}|_{w},\ 
\pi_{22}|_{w}={p_{21}^{1}}(w)\omega_{1}|_{w}+{p_{22}^{1}}(w)\omega_{2}|_{w}.$  
Moreover, $w$ satisfies $d\varpi_{1}|_{w}\equiv d\varpi_{2}|_{w}\equiv 0$: 
\begin{align*}
d\varpi_{1}|_{w}&\equiv \omega_{2}|_{w}\wedge \pi_{12}|_{w}
                 \equiv {p_{11}^{1}}(w)\omega_{2}|_{w}\wedge\omega_{1}|_{w},\\
d\varpi_{2}|_{w}&\equiv \omega_{1}|_{w}\wedge \pi_{12}|_{w}
                  +\omega_{2}|_{w}\wedge \pi_{22}|_{w}
         \equiv ({p_{12}^{1}}(w)-{p_{21}^{1}}(w))\omega_{1}|_{w}\wedge\omega_{2}|_{w}. 
\end{align*} 
Hence we obtain the defining equations $f_1=f_2=0$ of $\Sigma(R)$ in 
$U_{\omega_{1}\omega_{2}}$ of $J(D,2)$, where 
$f_1={p_{11}^{1}}, f_2=p_{12}^{1}-p_{21}^{1}$, that is,  
$\left\{f_1=f_2=0 \right\}\subset U_{\omega_{1}\omega_{2}}.$ 
Then $df_1, df_2$ are independent on $\left\{f_1=f_2=0\right\}$.\\ 
(II) On $U_{\omega_{1}\pi_{12}}$:\par 
For $w\in U_{\omega_{1}\pi_{12}}$, 
by restricting $\omega_{2}, \pi_{22}$ to $w$, we we introduce the 
inhomogeneous coordinate $p^{2}_{ij}$ of fibers of $J(D,2)$ 
around $w$ with   
$\omega_{2}|_{w}={p_{11}^{2}}(w)\omega_{1}|_{w}+{p_{12}^{2}}(w)\pi_{12}|_{w},\quad 
\pi_{22}|_{w}={p_{21}^{2}}(w)\omega_{1}|_{w}+{p_{22}^{2}}(w)\pi_{12}|_{w}.$  
Moreover, $w$ satisfies $d\varpi_{1}|_{w}\equiv d\varpi_{2}|_{w}\equiv 0$: 
\begin{align*}
d\varpi_{1}|_{w}&\equiv \omega_{2}|_{w}\wedge \pi_{12}|_{w} 
                 \equiv {p_{11}^{2}}(w)\omega_{1}|_{w}\wedge \pi_{12}|_{w},\\
d\varpi_{2}|_{w}&\equiv \omega_{1}|_{w}\wedge \pi_{12}|_{w}
                   +\omega_{2}|_{w}\wedge \pi_{22}|_{w}
                \equiv (1+{p_{11}^{2}}(w){p_{22}^{2}}(w)
                  -{p_{12}^{2}}(w){p_{21}^{2}}(w))\omega_{1}|_{w}\wedge\pi_{12}|_{w}.               
\end{align*} 
Then the defining functions of $\Sigma(R)$ are independent.\\
(III) On $U_{\omega_{1}\pi_{22}}$:\par 
For $w\in U_{\omega_{1}\pi_{22}}$, by 
restricting $\omega_{2}, \pi_{12}$ to $w$, 
we introduce the inhomogeneous coordinate $p^{3}_{ij}$ of fibers of $J(D,2)$ 
around $w$ with  
$\omega_{2}|_{w}={p_{11}^{3}}(w)\omega_{1}|_{w}+{p_{12}^{3}}(w)\pi_{22}|_{w},\ 
\pi_{12}|_{w}={p_{21}^{3}}(w)\omega_{1}|_{w}+{p_{22}^{3}}(w)\pi_{22}|_{w}.$ 
Moreover, $w$ satisfies $d\varpi_{1}|_{w}\equiv d\varpi_{2}|_{w}\equiv 0$: 
\begin{align*}
d\varpi_{1}|_{w}&\equiv \omega_{2}|_{w}\wedge \pi_{12}|_{w} 
              \equiv ({p_{11}^{3}}(w){p_{22}^{3}}(w)-{p_{12}^{3}}(w)
                {p_{21}^{3}}(w))\omega_{1}|_{w}\wedge\pi_{22}|_{w},\\
d\varpi_{2}|_{w}&\equiv \omega_{1}|_{w}\wedge \pi_{12}|_{w}
                    +\omega_{2}|_{w}\wedge \pi_{22}|_{w}
              \equiv ({p_{11}^{3}}(w)+{p_{22}^{3}}(w))\omega_{1}|_{w}\wedge\pi_{22}|_{w}. 
\end{align*} 
Therefore, we obtain the defining equations $f_1=f_2=0$ of $\Sigma(R)$ in 
$U_{\omega_{1}\pi_{22}}$ of $J(D,2)$, where 
$f_1={p_{11}^{3}}{p_{22}^{3}}-{p_{12}^{3}}{p_{21}^{3}},\ f_2={p_{11}^{3}}+{p_{22}^{3}}$, 
that is,   
$\left\{f_1=f_2=0 \right\}\subset U_{\omega_{1}\pi_{22}}.$  
Then, $df_1,df_2$ are linearly dependent on 
$S:=\left\{{p_{11}^{3}}={p_{22}^{3}}, {p_{12}^{3}}={p_{21}^{3}}=0\right\}$. 
Hence, $S \cap \Sigma(R)=\{{p_{11}^{3}}={p_{22}^{3}}={p_{12}^{3}}={p_{21}^{3}}= 0\}$ which is a point on each fiber is a singular subset in $\Sigma(R)$ .\\
(IV) On $U_{\omega_{2}\pi_{12}}$:\par 
For $w\in U_{\omega_{2}\pi_{12}}$, 
by restricting $\omega_{1}, \pi_{22}$ to $w$, we introduce the 
inhomogeneous coordinate $p^{4}_{ij}$ of fibers of $J(D,2)$ 
around $w$ with   
$\omega_{1}|_{w}={p_{11}^{4}}(w)\omega_{2}|_{w}+{p_{12}^{4}}(w)\pi_{12}|_{w},\ 
\pi_{22}|_{w}={p_{21}^{4}}(w)\omega_{2}|_{w}+{p_{22}^{4}}(w)\pi_{12}|_{w}$.  
Moreover, $w$ satisfies $d\varpi_{1}|_{w}\equiv d\varpi_{2}|_{w}\equiv 0$. 
However, we have
$d\varpi_{1}|_{w}\equiv \omega_{2}|_{w}\wedge \pi_{12}|_{w}\not \equiv 0$.  
Hence, there does not exist integral element, that is,    
$U_{\omega_{2}\pi_{12}}\cap p^{-1}(U)=\emptyset$.\\
(V) On $U_{\omega_{2}\pi_{22}}$:\par 
For $w\in U_{\omega_{2}\pi_{22}}$, 
by restricting $\omega_{1}, \pi_{12}$ to $w$, 
we can introduce the 
inhomogeneous coordinate $p^{5}_{ij}$ of fibers of $J(D,2)$ 
around $w$ with    
$\omega_{1}|_{w}={p_{11}^{5}}(w)\omega_{2}|_{w}+{p_{12}^{5}}(w)\pi_{22}|_{w},\ 
\pi_{12}|_{w}={p_{21}^{5}}(w)\omega_{2}|_{w}+{p_{22}^{5}}(w)\pi_{22}|_{w}$.  
Moreover, $w$ satisfies $d\varpi_{1}|_{w}\equiv d\varpi_{2}|_{w}\equiv 0$: 
\begin{align*}
d\varpi_{1}|_{w}&\equiv \omega_{2}|_{w}\wedge \pi_{12}|_{w} 
                 \equiv {p_{22}^{5}}(w)\omega_{2}|_{w}\wedge \pi_{22}|_{w},\\
d\varpi_{2}|_{w}&\equiv \omega_{1}|_{w}\wedge \pi_{12}|_{w}
                +\omega_{2}|_{w}\wedge \pi_{22}|_{w}
              \equiv (1+{p_{11}^{5}}(w){p_{22}^{5}}(w)
                  -{p_{12}^{5}}(w){p_{21}^{5}}(w))\omega_{2}|_{w}\wedge\pi_{22}|_{w}. 
\end{align*} 
Then the defining functions of $\Sigma(R)$ are independent 
in the same as (I).\\
(VI) On $U_{\pi_{12}\pi_{22}}$:\par 
For $w\in U_{\pi_{12}\pi_{22}}$, 
by restricting $\omega_{1}, \omega_{2}$ to $w$, 
we introduce the inhomogeneous coordinate $p^{6}_{ij}$ 
of fibers of $J(D,2)$ around $w$ with  
$\omega_{1}|_{w}={p_{11}^{6}}(w)\pi_{12}|_{w}+{p_{12}^{6}}(w)\pi_{22}|_{w},\ 
\omega_{2}|_{w}={p_{21}^{6}}(w)\pi_{12}|_{w}+{p_{22}^{6}}(w)\pi_{22}|_{w}$. 
Moreover, $w$ satisfies $d\varpi_{1}|_{w}\equiv d\varpi_{2}|_{w}\equiv 0$: 
\begin{align*}
d\varpi_{1}|_{w}&\equiv \omega_{2}|_{w}\wedge \pi_{12}|_{w} 
                \equiv {p_{22}^{6}}(w)\pi_{22}|_{w}\wedge\pi_{12}|_{w},\\
d\varpi_{2}|_{w}&\equiv \omega_{1}|_{w}\wedge \pi_{12}|_{w}
                    +\omega_{2}|_{w}\wedge \pi_{22}|_{w}
                \equiv ({p_{21}^{6}}(w)-{p_{12}^{6}}(w))\pi_{12}|_{w}\wedge\pi_{22}|_{w}. 
\end{align*} 
Then the defining functions of $\Sigma(R)$ are also independent.\par  
Summarizing these discussions, the rank 2 prolongations $\Sigma(R)$ for locally parabolic equations $R$ has singular points, that is, 
these are not smooth.  
\end{proof}
We set 
$P_{\omega_{1}\omega_{2}}:=p^{-1}(U)\cap U_{\omega_{1}\omega_{2}},\  
 P_{\omega_{1}\pi_{12}}:=p^{-1}(U)\cap U_{\omega_{1}\pi_{12}},\ 
 P_{\omega_{1}\pi_{22}}:=p^{-1}(U)\cap U_{\omega_{1}\pi_{22}},\ 
P_{\omega_{2}\pi_{22}}:=p^{-1}(U)\cap U_{\omega_{2}\pi_{22}}$, and  
$P_{\pi_{12}\pi_{22}}:=p^{-1}(U)\cap U_{\pi_{12}\pi_{22}}$.  
\begin{lemma}\label{par-lem}
We have 
$p^{-1}(U)=P_{\omega_{1}\omega_{2}}\cup P_{\omega_{1}\pi_{22}}\cup P_{\pi_{12}\pi_{22}} \label{parcover}$. %
\end{lemma}
\begin{proof}
From the discussion of the proof of the previous proposition, we have 
$p^{-1}(U)=P_{\omega_{1}\omega_{2}}\cup P_{\omega_{1}\pi_{12}}
\cup P_{\omega_{1}\pi_{22}}
\cup P_{\omega_{2}\pi_{22}}\cup P_{\pi_{12}\pi_{22}}.$ 
Hence, it is sufficient to prove $P_{\omega_{1}\pi_{12}},\ P_{\omega_{2}\pi_{22}}\subset P_{\omega_{1}\omega_{2}}$. 
For the open set $P_{\omega_{1}\pi_{12}}$, we prove this property. 
Let $w$ be any point in $P_{\omega_{1}\pi_{12}}\subset p^{-1}(U)$. 
Here, if $w \not\in P_{\omega_{1}\omega_{2}}$, 
then $\omega_{1}|_{w}\wedge \omega_{2}|_{w}=0$.  
Hence, by $\omega_{1}|_{w}\wedge \omega_{2}|_{w}={p_{12}^{2}}(w)\omega_{1}|_{w}\wedge \pi_{12}|_{w}$, 
we have the condition ${p_{12}^{2}}(w)=0$. 
However, $w$ is an integral element, and  
we have ${p_{12}^{2}}(w)\not=0$.   
Thus, we have $P_{\omega_{1}\pi_{12}}\subset P_{\omega_{1}\omega_{2}}$. 
For the open set $P_{\omega_{2}\pi_{22}}$, we also obtain the statement from the same argument.  
\end{proof}
\begin{theorem}\label{par-topology}
Let $(R,D)$ be a locally parabolic equation. Then, the 
rank $2$ prolongation $\Sigma(R)$ has singular points,  
and it has the structure of pinched torus fibration.  
\end{theorem}
\begin{figure}[htbp]
\begin{center}
\includegraphics[width=2.5cm]{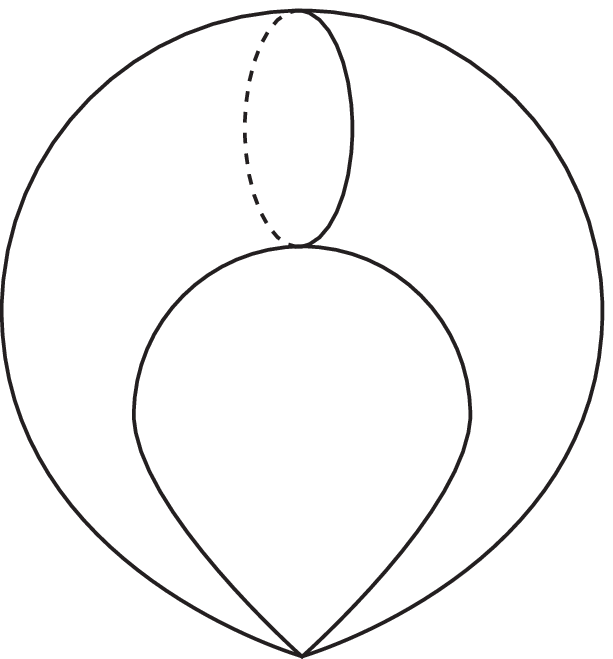}
\end{center}
\end{figure}
\begin{proof}
By the above lemma, note that the fiber $p^{-1}(w)$ at $w\in R$ 
decompose to the disjoint union  
$p^{-1}(w)=\mathbb R^2\cup \mathbb R\cup \left\{ \rm{a\ point}\right\}$ 
as a set. 
Moreover, by gluing on 
$p^{-1}(U)=P_{\omega_{1}\omega_{2}}\cup P_{\omega_{1}\pi_{22}}\cup P_{\pi_{12}\pi_{22}}$
in the proof of the previous proposition and lemma, we obtain the statement. 
\end{proof}
\noindent
{\bf Rank 2 prolongations of elliptic equations.} 
Let $(R,D)$ be a locally elliptic equation. Then,  
there exists a local coframe 
$\left\{\varpi_0, \varpi_1, \varpi_2, \omega_1, \omega_2, \pi_{11}, \pi_{12}\right\}$  around $x\in R$ such that  
$D=\{\varpi_0=\varpi_1=\varpi_2=0 \}$ and 
the following structure equation holds:   
\begin{align}\label{ellequ}
d\varpi_{0}&\equiv \omega_{1}\wedge \varpi_{1}+\omega_{2}\wedge \varpi_{2}
\quad \mod \ \varpi_{0}, \nonumber \\
d\varpi_{1}&\equiv \omega_{1}\wedge \pi_{11}+\omega_{2}\wedge \pi_{12} \quad \mod\ \varpi_{0}, \varpi_{1}, \varpi_{2},\\
d\varpi_{2}&\equiv \omega_{1}\wedge \pi_{12}-\omega_{2}\wedge \pi_{11} \ \mod\ \varpi_{0}, 
\varpi_{1}, \varpi_{2}. \nonumber
\end{align}
From this structure equation, we investigate the rank 2 prolongation 
$\Sigma(R)$. 
Let $U$ be an open set in $R$, and $\pi:J(D,2)\to R$ the projection. 
Then $\pi^{-1}(U)$ is covered by $6$ open sets in $J(D,2)$: 
\begin{equation}
\pi^{-1}(U)=U_{\omega_{1}\omega_{2}}\cup U_{\omega_{1}\pi_{11}}
\cup U_{\omega_{1}\pi_{12}}\cup U_{\omega_{2}\pi_{11}}
\cup U_{\omega_{2}\pi_{12}}\cup U_{\pi_{11}\pi_{12}}, 
\end{equation}
where each open set is also given in the same way 
as hyperbolic case (\ref{Grassmann-cover}). 
Now we explicitly describe the defining 
equation of $\Sigma(R)$ in terms of the inhomogeneous Grassmann coordinate of fibers in  
$U_{\omega_{1}\omega_{2}},...., U_{\pi_{11}\pi_{12}}$.\\
(I) On $U_{\omega_{1}\omega_{2}}$:\par 
For $w\in U_{\omega_{1}\omega_{2}}$, 
$w$ is a $2$-dimensional subspace of $D(v)$, 
$p(w)=v$. Hence, by restricting $\pi_{11},\pi_{12}$ to $w$, 
we introduce the inhomogeneous coordinate $p^{1}_{ij}$ of fibers of 
$J(D,2)$ around $w$ with 
$\pi_{11}|_{w}={p_{11}^{1}}(w)\omega_{1}|_{w}+{p_{12}^{1}}(w)\omega_{2}|_{w},\ 
\pi_{12}|_{w}={p_{21}^{1}}(w)\omega_{1}|_{w}+{p_{22}^{1}}(w)\omega_{2}|_{w}.$  
Moreover $w$ satisfies $d\varpi_{1}|_{w}\equiv d\varpi_{2}|_{w}\equiv 0$: 
\begin{align*}
d\varpi_{1}|_{w}&\equiv \omega_{1}|_{w}\wedge \pi_{11}|_{w}
                              +\omega_{2}|_{w}\wedge \pi_{12}|_{w} 
                 \equiv ({p_{12}^{1}}(w)-{p_{21}^{1}}(w)) 
                    \omega_{1}|  _{w}\wedge\omega_{2}|_{w},\\
d\varpi_{2}|_{w}&\equiv \omega_{1}|_{w}\wedge \pi_{12}|_{w}
                    -\omega_{2}|_{w}\wedge \pi_{11}|_{w}
        \equiv ({p_{11}^{1}}(w)+{p_{22}^{1}}(w))\omega_{1}|_{w}\wedge\omega_{2}|_{w}. 
\end{align*} 
Hence we obtain the defining equations $f_1=f_2=0$ of $\Sigma(R)$ in  $U_{\omega_{1}\omega_{2}}$ of $J(D,2)$, where 
$f_1={p_{12}^{1}}-{p_{21}^{1}}, f_2={p_{11}^{1}}+{p_{22}^{1}}$, that is,  
$\left\{f_1=f_2=0 \right\}\subset U_{\omega_{1}\omega_{2}}.$  
Then $df_1, df_2$ are independent on $\left\{f_1=f_2=0\right\}$.\\ 
(II) On $U_{\omega_{1}\pi_{11}}$:\par 
For $w\in U_{\omega_{1}\pi_{11}}$, 
by restricting $\omega_{2}, \pi_{12}$ to $w$, we  
introduce the inhomogeneous coordinate $p^{2}_{ij}$ of fibers of 
$J(D,2)$ around $w$ with 
$\omega_{2}|_{w}={p_{11}^{2}}(w)\omega_{1}|_{w}+{p_{12}^{2}}(w)\pi_{11}|_{w},\ 
\pi_{12}|_{w}={p_{21}^{2}}(w)\omega_{1}|_{w}+{p_{22}^{2}}(w)\pi_{11}|_{w}$.  
Moreover, $w$ satisfies $d\varpi_{1}|_{w}\equiv d\varpi_{2}|_{w}\equiv 0$: 
\begin{align*}
d\varpi_{1}|_{w}&\equiv \omega_{1}|_{w}\wedge \pi_{11}|_{w}
                       +\omega_{2}|_{w}\wedge \pi_{12}|_{w} 
                 \equiv (1+{p_{11}^{2}}(w){p_{22}^{2}}(w)-{p_{12}^{2}}(w){p_{21}^{2}}(w)) 
                  \omega_{1}|_{w}\wedge \pi_{11}|_{w},\\
d\varpi_{2}|_{w}&\equiv \omega_{1}|_{w}\wedge \pi_{12}|_{w}
                   -\omega_{2}|_{w}\wedge \pi_{11}|_{w} 
            \equiv (-{p_{11}^{2}}(w)+{p_{22}^{2}}(w))\omega_{1}|_{w}\wedge\pi_{11}|_{w}. 
\end{align*} 
Then the defining functions of $\Sigma(R)$ are independent 
in the same as (I).\\
(III) On $U_{\omega_{1}\pi_{12}}$:\par
For $w\in U_{\omega_{1}\pi_{12}}$, 
by restricting $\omega_{2}, \pi_{11}$ to $w$, we 
introduce the inhomogeneous coordinate $p^{3}_{ij}$ of fibers of 
$J(D,2)$ around $w$ with 
$\omega_{2}|_{w}={p_{11}^{3}}(w)\omega_{1}|_{w}+{p_{12}^{3}}(w)\pi_{12}|_{w},\ 
\pi_{11}|_{w}={p_{21}^{3}}(w)\omega_{1}|_{w}+{p_{22}^{3}}(w)\pi_{12}|_{w}$. 
Moreover, $w$ satisfies $d\varpi_{1}|_{w}\equiv d\varpi_{2}|_{w}\equiv 0$: 
\begin{align*}
d\varpi_{1}|_{w}&\equiv 
    \omega_{1}|_{w}\wedge \pi_{11}|_{w}+\omega_{2}|_{w}\wedge \pi_{12}|_{w} 
        \equiv ({p_{11}^{3}}(w)+{p_{22}^{3}}(w))\omega_{1}|_{w}\wedge \pi_{12}|_{w},\\
d\varpi_{2}|_{w}&\equiv 
    \omega_{1}|_{w}\wedge \pi_{12}|_{w}-\omega_{2}|_{w}\wedge \pi_{11}|_{w}
        \equiv (1-{p_{11}^{3}}(w){p_{22}^{3}}(w)+{p_{12}^{3}}(w){p_{21}^{3}}(w))
                \omega_{1}|_{w}\wedge\pi_{12}|_{w}. 
\end{align*} 
Then the defining functions of $\Sigma(R)$ are also independent.\\
(IV) On $U_{\omega_{2}\pi_{11}}$:\par 
For $w\in U_{\omega_{2}\pi_{11}}$, 
by restricting $\omega_{1}, \pi_{12}$ to $w$, we 
introduce the inhomogeneous coordinate $p^{4}_{ij}$ of fibers of 
$J(D,2)$ around $w$ with
$\omega_{1}|_{w}={p_{11}^{4}}(w)\omega_{2}|_{w}+{p_{12}^{4}}(w)\pi_{11}|_{w},\ 
\pi_{12}|_{w}={p_{21}^{4}}(w)\omega_{2}|_{w}+{p_{22}^{4}}(w)\pi_{11}|_{w}$. 
Moreover, $w$ satisfies $d\varpi_{1}|_{w}\equiv d\varpi_{2}|_{w}\equiv 0$: 
\begin{align*}
d\varpi_{1}|_{w}&\equiv 
  \omega_{1}|_{w}\wedge \pi_{11}|_{w}+\omega_{2}|_{w}\wedge \pi_{12}|_{w}
    \equiv ({p_{11}^{4}}(w)+{p_{22}^{4}}(w))\omega_{2}|_{w}\wedge \pi_{11}|_{w},\\
d\varpi_{2}|_{w}&\equiv 
  \omega_{1}|_{w}\wedge \pi_{12}|_{w}-\omega_{2}|_{w}\wedge \pi_{11}|_{w}
    \equiv ({p_{11}^{4}}(w){p_{22}^{4}}(w)-{p_{12}^{4}}(w){p_{21}^{4}}(w)-1)
                \omega_{2}|_{w}\wedge\pi_{11}|_{w}. 
\end{align*} 
Then the defining functions of $\Sigma(R)$ are also independent.\\
(V) On $U_{\omega_{2}\pi_{12}}$:\par 
For $w\in U_{\omega_{2}\pi_{12}}$, 
by restricting $\omega_{1}, \pi_{11}$ to $w$, we 
introduce the inhomogeneous coordinate $p^{5}_{ij}$ of fibers of 
$J(D,2)$ around $w$ with
$\omega_{1}|_{w}={p_{11}^{5}}(w)\omega_{2}|_{w}+{p_{12}^{5}}(w)\pi_{12}|_{w},\ 
\pi_{11}|_{w}={p_{21}^{5}}(w)\omega_{2}|_{w}+{p_{22}^{5}}(w)\pi_{12}|_{w}$. 
Moreover, $w$ satisfies $d\varpi_{1}|_{w}\equiv d\varpi_{2}|_{w}\equiv 0$: 
\begin{align*}
d\varpi_{1}|_{w}&\equiv 
   \omega_{1}|_{w}\wedge \pi_{11}|_{w}+\omega_{2}|_{w}\wedge \pi_{12}|_{w} 
   \equiv (1+{p_{11}^{5}}(w){p_{22}^{5}}(w)-{p_{12}^{5}}(w){p_{21}^{5}}(w)) 
                  \omega_{2}|_{w}\wedge \pi_{12}|_{w},\\
d\varpi_{2}|_{w}&\equiv 
   \omega_{1}|_{w}\wedge \pi_{12}|_{w}-\omega_{2}|_{w}\wedge \pi_{11}|_{w} 
   \equiv ({p_{11}^{5}}(w)-{p_{22}^{5}}(w))\omega_{2}|_{w}\wedge\pi_{12}|_{w}. 
\end{align*} 
Then the defining functions of $\Sigma(R)$ are also independent.\\
(VI) On $U_{\pi_{11}\pi_{12}}$:\par 
For $w\in U_{\pi_{11}\pi_{12}}$, 
by restricting $\omega_{1}, \omega_{2}$ to $w$, we 
introduce the inhomogeneous coordinate $p^{6}_{ij}$ of fibers of 
$J(D,2)$ around $w$ with
$\omega_{1}|_{w}={p_{11}^{6}}(w)\pi_{11}|_{w}+{p_{12}^{6}}(w)\pi_{12}|_{w},\ 
\omega_{2}|_{w}={p_{21}^{6}}(w)\pi_{11}|_{w}+{p_{22}^{6}}(w)\pi_{12}|_{w}$. 
Moreover, $w$ satisfies 
$d\varpi_{1}|_{w}\equiv d\varpi_{2}|_{w}\equiv 0$: 
\begin{align*}
d\varpi_{1}|_{w}&\equiv 
  \omega_{1}|_{w}\wedge \pi_{11}|_{w}+\omega_{2}|_{w}\wedge \pi_{12}|_{w} 
       \equiv (-{p_{12}^{6}}(w)+{p_{21}^{6}}(w))\pi_{11}|_{w}\wedge \pi_{12}|_{w},\\
d\varpi_{2}|_{w}&\equiv 
  \omega_{1}|_{w}\wedge \pi_{12}|_{w}-\omega_{2}|_{w}\wedge \pi_{11}|_{w} 
   \equiv ({p_{11}^{6}}(w)+{p_{22}^{6}}(w))\pi_{11}|_{w}\wedge \pi_{12}|_{w}. 
\end{align*} 
Then the defining functions of $\Sigma(R)$ are also independent.\par 
Summarizing these discussions, the rank 2 prolongation $\Sigma(R)$ of a locally elliptic equation $R$ is smooth, and it has the covering 
$p^{-1}(U)=P_{\omega_{1}\omega_{2}}\cup P_{\omega_{1}\pi_{11}}
\cup P_{\omega_{1}\pi_{12}}\cup P_{\omega_{2}\pi_{11}}
\cup P_{\omega_{2}\pi_{12}}\cup P_{\pi_{11}\pi_{12}},$ 
where
$P_{\omega_{1}\omega_{2}}:=p^{-1}(U)\cap U_{\omega_{1}\omega_{2}},\   
 P_{\omega_{1}\pi_{11}}:=p^{-1}(U)\cap U_{\omega_{1}\pi_{11}},\ 
 P_{\omega_{1}\pi_{12}}:=p^{-1}(U)\cap U_{\omega_{1}\pi_{12}},\ 
P_{\omega_{2}\pi_{11}}:=p^{-1}(U)\cap U_{\omega_{2}\pi_{11}},\ 
 P_{\omega_{2}\pi_{12}}:=p^{-1}(U)\cap U_{\omega_{2}\pi_{12}}$, and   
$P_{\pi_{11}\pi_{12}}:=p^{-1}(U)\cap U_{\pi_{11}\pi_{12}}$. 
However, this covering is not essential in the following sense.   
\begin{lemma}\label{ell-lem}
Let $(R,D)$ be a locally elliptic equation and $p:\Sigma(R)\to R$ be the 
rank $2$ prolongation. Then, for any open set $U\subset R$, we have
$p^{-1}(U)=P_{\omega_{1}\omega_{2}}\cup P_{\pi_{11}\pi_{12}}$.
\end{lemma} 
\begin{proof}
It is sufficient to prove $P_{\omega_{1}\pi_{11}},\ P_{\omega_{1}\pi_{12}},\ 
P_{\omega_{2}\pi_{11}},\ P_{\omega_{2}\pi_{12}}\subset P_{\omega_{1}\omega_{2}}$. 
For the open set $P_{\omega_{1}\pi_{11}}$, we prove this property. 
Let $w$ be a point in $P_{\omega_{1}\pi_{11}}\subset p^{-1}(U)$. 
Here, if $w \not\in P_{\omega_{1}\omega_{2}}$, 
then the condition $\omega_{1}|_{w}\wedge \omega_{2}|_{w}=0$ is satisfied.   
Hence, by $\omega_{1}|_{w}\wedge \omega_{2}|_{w}={p_{12}^{2}}(w)\omega_{1}|_{w}\wedge \pi_{11}|_{w}$, 
we have ${p_{12}^{2}}(w)=0$. However, $w$ is an integral element. In terms of $f_1=f_2=0$,   
we have $({p_{11}^{2}})^2=-1$. This is a contradiction.   
Thus, we have $P_{\omega_{1}\pi_{11}}\subset P_{\omega_{1}\omega_{2}}$. 
For other open sets, we also have the statement from the similar argument.  
\end{proof}
\begin{theorem}\label{ell-topology}
Let $(R,D)$ be a locally elliptic equation. Then, the 
rank $2$ prolongation $\Sigma(R)$ is a smooth submanifold of $J(D,2)$, 
and it is a $S^2$-bundle over $R$. 
\end{theorem}
\begin{proof}
By the above lemma, note that the fiber $p^{-1}(w)$ at $w\in R$ 
decompose to the disjoint union  
$p^{-1}(w)=\mathbb R^2\cup \left\{ \rm{a\ point}\right\}$ as a set. 
Moreover, we obtain the statement from the same argument to the parabolic case.  
\end{proof}
\noindent
{\bf Characterization of equations by the fiber topology.}
\ 
We obtain one of the main results by summarizing 
theorems of the previous part of this section. 
\begin{corollary}\label{all-topology}
Let $R=\{ F=0 \}$ be a second-order regular PDE and 
$\Sigma(R)$ be the its prolongation. Let $p:\Sigma(R) \to R$ be the natural projection. Then,  
\begin{enumerate}
\item [$(${\rm 1}$)$] $w\in R$ is hyperbolic   $\iff$ $p^{-1}(w)$ is a 
$2$-dimensional torus $T^2$.
\item [$(${\rm 2}$)$] $w\in R$ is parabolic   $\iff$ $p^{-1}(w)$ is a pinched $2$-dimensional torus.
\item [$(${\rm 3}$)$] $w\in R$ is elliptic   $\iff$ $p^{-1}(w)$ is a $2$-dimensional sphere $S^2$.
\end{enumerate}
\end{corollary}
\begin{proof}
Note that the fiber $p^{-1}(w)$ is defined by the structure 
equation of $D$ at $w$ as a subset in the fiber 
$J_w\cong Gr(2,4)$ of the fibration $\pi:J(D,2)\to R$. 
From this point of view, the topology of the fiber   
$p^{-1}(w)$ depends only on the pointwise structure equations 
(\ref{hypequ}), (\ref{parequ}) and (\ref{ellequ}).   
\end{proof} 
\section{Structures of the canonical systems on the rank 2 prolongations}
In this section, we study the geometric structures of the 
rank 2 prolongations $(\Sigma(R), \hat D)$ 
for each class of equations. We first recall Tanaka theory of weakly 
regular differential systems in this section.
For more details, we refer the reader to 
\cite{Tan1} and \cite{Y3}. \par
\medskip\noindent  
{\bf Derived system, Weak derived system.}
Let $D$ be a differential system on a manifold $R$. 
We denote by $\mathcal D=\Gamma(D)$ the sheaf of sections to $D$. 
The {\it derived system} $\partial D$ of a differential system $D$ is defined, 
in terms of sections, by 
$\partial \mathcal D:=\mathcal D+[\mathcal D,\mathcal D].$
In general, $\partial D$ is obtained as a 
subsheaf of the tangent sheaf of $R$.  
Moreover, {\it higher derived systems} $\partial^{k} D$ are 
defined successively by 
$\partial^{k} \mathcal D:=\partial (\partial^{k-1} \mathcal D),$ 
where we set $\partial^{0} D=D$ by convention. 
On the other hand, $k$-th {\it weak derived systems}
$\partial^{(k)} D$ of $D$ are defined inductively by 
$\partial^{(k)} \mathcal D:=\partial^{(k-1)}\mathcal D
                             +[\mathcal D,\partial^{(k-1)}\mathcal D].$
\begin{definition}
A differential system $D$ is called {\it regular} ({\it weakly regular}), if $\partial^{k} D$ 
(resp. $\partial^{(k)} D$) is a subbundle for each $k$.
\end{definition}
If $D$ is not weakly regular around $x\in R$, then $x$ is called 
{\it a singular point} in the sense of Tanaka theory.  
These derived systems are also interpreted by using 
annihilators as follows \cite{S}:  
Let $D=\{\varpi_1=\cdots =\varpi_s=0 \}$ be a differential system on a 
manifold $R$. We denote by $D^{\perp}$ the annihilator 
subbundle of $D$ in $T^* R$, namely, 
\begin{eqnarray*}
D^{\perp}(x) &=& \{\omega\in T_{x}^{*} R\ |\ \omega(X)=0\ \text{for any}\ X \in D(x)
 \}, \\
 &=& <\varpi_1, \cdots ,\varpi_s >.
\end{eqnarray*}
Then the annihilator $(\partial D)^{\perp}$ of 
the first derived system of $D$ 
is given by 
\[
(\partial D)^{\perp}=\{\varpi \in D^{\perp}\ |\ d\varpi \equiv 0\ 
(\bmod\ D^{\perp})\}.
\]
Moreover, the annihilator $(\partial^{(k+1)} D)^{\perp}$ of the $(k+1)$-th weak
derived system of $D$ is given by  
\begin{eqnarray*}
(\partial^{(k+1)} D)^{\perp} &=& \{\varpi \in (\partial^{(k)} D)^{\perp}\ |\ 
d\varpi \equiv 0\ (\bmod\ (\partial^{(k)} D)^{\perp}, \\
& & \hspace{3cm} (\partial^{(p)} D)^{\perp}\wedge (\partial^{(q)} D)^{\perp},\ 2 \le
p,q \le k-1)\}.
\end{eqnarray*}
We set $D^{-1}:=D,\ D^{-k}:=\partial ^{(k-1)} D$ ($k\geq 2$), for 
a weakly regular differential system $D$. Then we have 
(\cite[Proposition 1.1]{Tan1}):
\begin{enumerate} 
\item [(T1)] There exists a unique positive integer $\mu$ such that 
$$D^{-1}\subset D^{-2}\subset \cdot\cdot\cdot \subset D^{-k} \subset \cdot\cdot\cdot 
\subset D^{-(\mu-1)} \subset D^{-\mu}=D^{-(\mu+1)}=\cdot\cdot\cdot$$
\item[(T2)] $[\mathcal D^{p}, \mathcal D^{q}]\subset \mathcal D^{p+q}$\quad \quad  
for\ all\ $p,\ q<0$. 
\end{enumerate}
\noindent{\bf Symbol algebra of differential system.} 
Let $(R, D)$ be a weakly regular 
differential system such that 
$$TR=D^{-\mu}\supset D^{-(\mu-1)}\supset \cdot\cdot\cdot \supset
D^{-1}=:D.$$
For all $x\in R$, we put $\mathfrak g_{-1}(x):=D^{-1}(x)=D(x),\    
\mathfrak g_{p}(x):=D^{p}(x)/D^{p+1}(x),\ (p=-2,-3, \ldots, -\mu)$ and 
$$\mathfrak m(x):=\bigoplus_{p=-1}^{-\mu} \mathfrak g_{p}(x).$$
Then, dim $\mathfrak m(x)=$ dim $R$. 
We set $\mathfrak g_{p}(x)=\left\{0\right\}$ when $p\leq-\mu-1$. 
For $X\in \mathfrak g_{p}(x),\ Y\in \mathfrak g_{q}(x)$, the Lie bracket 
$[X, Y]\in\mathfrak g_{p+q}(x)$ is defined in the following way:\\ 
Let $\tilde X\in \mathcal D^{p},\ \tilde Y\in \mathcal D^{q}$ be extensions 
($\tilde X_{x}=X,\ \tilde Y_{x}=Y$). Then $[\tilde X, \tilde Y]\in \mathcal D^{p+q}$,  
and we set $[X,Y]:=[\tilde X, \tilde Y]_{x}\in \mathfrak g_{p+q}(x)$. 
It does not depend on the choice of the extensions  
because of the equation
\[
[f \tilde{X}, g \tilde{Y}] =fg[\tilde{X},\tilde{Y}] 
+f(\tilde{X} g) \tilde{Y} -g(\tilde{Y} f)\tilde{X}\quad 
(f,g \in C^{\infty}(R)).
\] 
The Lie algebra $\mathfrak m(x)$ is a nilpotent graded Lie algebra. 
we call $(\mathfrak m(x),\ [\ ,\ ])$ the 
{\it symbol algebra} of $(R, D)$ at $x$. 
Note that the symbol algebra $(\mathfrak m(x),\ [\ ,\ ])$ satisfies the generating conditions 
$[\mathfrak g_{p}, \mathfrak g_{-1}]=\mathfrak g_{p-1}\ \ (p<0).$ \par
Later, Morimoto \cite{M} 
introduced the notion of a filtered manifold as generalization 
of the weakly regular differential system.  
We define a {\it filtered manifold} $(R,F)$ by a pair of a manifold $R$ 
and a tangential filtration $F$. Here, a tangential filtration 
$F$ on $R$ is a sequence $\left\{F^{p}\right\}_{p<0}$ of subbundles of 
the tangent bundle $TR$ such that the following conditions are satisfied: 
\begin{enumerate} 
\item [(M1)] $TR=F^{k}=\cdot\cdot\cdot =F^{-\mu}\supset \cdot\cdot\cdot \supset 
F^{p}\supset F^{p+1}\supset \cdot\cdot\cdot \supset F^{0}=\left\{0\right\},$
\item[(M2)] $[\mathcal F^{p}, \mathcal F^{q}]\subset \mathcal F^{p+q}$\quad \quad  
for\ all\ $p,\ q<0$,  
\end{enumerate}
where $\mathcal F^{p}=\Gamma(F^{p})$ is the set of sections of $F^{p}$.\par 
Let $(R,F)$  be a filtered manifold, for $x\in R$, we set 
$\mathfrak f_{p}(x):=F^{p}(x)/F^{p+1}(x),$ 
and 
$$\mathfrak f(x):=\bigoplus_{p<0} \mathfrak f_{p}(x).$$
For $X \in \mathfrak f_{p}(x),\ Y \in \mathfrak f_{q}(x)$, 
Lie bracket $[X,Y]\in \mathfrak f_{p+q}(x)$ is defined by:\\
Let $\tilde X\in \mathcal F^{p},\ \tilde Y\in \mathcal F^{q}$ be extensions 
($\tilde X_{x}=X,\ \tilde Y_{x}=Y$). Then $[\tilde X, \tilde Y]\in \mathcal F^{p+q}$,  
and we set $[X,Y]:=[\tilde X, \tilde Y]_{x}\in \mathfrak f_{p+q}(x)$. 
It does not depend on the choice of the extensions. 
The Lie algebra 
$\mathfrak f(x)$ is also a nilpotent graded Lie algebra. 
We call $(\mathfrak f(x),\ [\ ,\ ])$ the {\it symbol algebra} of 
$(R, F)$ at $x$. 
In general $(\mathfrak f(x),\ [\ ,\ ])$ does not satisfy the generating conditions. \par
\medskip\noindent{\bf Structures of rank 2 prolongations for hyperbolic equations.} 
Let $(R,D)$ be a locally hyperbolic equation, and $(\Sigma(R),\hat D)$  
the rank 2 prolongation. 
We first explain the geometric meaning of the open covering 
$P_{\omega_{1}\omega_{2}}\cup P_{\omega_{1}\pi_{22}}\cup 
P_{\omega_{2}\pi_{11}}\cup P_{\pi_{11}\pi_{22}}$ in the proof of 
Theorem \ref{hyp-pro}. The set $\Sigma(R)$ has a geometric decomposition: 
\begin{equation}\label{decomposition1}
\Sigma(R)=\Sigma_{0}\cup \Sigma_{1}\cup \Sigma_{2}\quad \quad ({\rm disjoint\ union}). 
\end{equation} 
where $\Sigma_{i}=\left\{w \in \Sigma(R)\ |\ {\rm dim}\ (w\cap {\rm fiber})=i \right\}$, $i=0,1,2,$ 
and ``fiber" means that the fiber of $TR\supset D\to TJ^1$. 
Then, locally, we have 
$\Sigma_{0}|_{p^{-1}(U)}=P_{\omega_{1}\omega_{2}},\ 
\Sigma_{1}|_{p^{-1}(U)}=(P_{\omega_{1}\pi_{22}}\cup P_{\omega_{2}\pi_{11}})
\backslash P_{\omega_{1}\omega_{2}},\ 
\Sigma_{2}|_{p^{-1}(U)}=P_{\pi_{11}\pi_{22}}\backslash 
(P_{\omega_{1}\omega_{2}}\cup P_{\omega_{1}\pi_{22}}\cup P_{\omega_{2}\pi_{11}})$. 
The set $\Sigma_0$ is an open subset in $\Sigma(R)$, and is a $\mathbb{R}^2$--bundle over $R$. The set $\Sigma_1$ is a codimension 1 submanifold in $\Sigma(R)$, and is a $(\mathbb{R} \cup \mathbb{R})$-bundle over $R$. The set $\Sigma_2$ is a codimension 2 submanifold in $\Sigma(R)$, and is a section of $\Sigma(R) \to R$. 
\begin{proposition}\label{hyp-derived}
The differential system $\hat D$ on $\Sigma(R)$ is regular, but 
is not weakly regular. More precisely, we obtain that 
$\hat D\subset \partial \hat D\subset \partial^{2}\hat D \subset 
\partial^{3}\hat D=T\Sigma(R).$
Moreover, we have $\partial^{2}\hat D=\partial^{(2)}\hat D$, 
$\partial^{(3)}\hat D=T\Sigma(R)\ \ {\rm on}\ 
\Sigma_{0}\cup\Sigma_{1}$, and   
$\partial^{(3)}\hat D=\partial^{(2)}\hat D\ \ {\rm on}\ \Sigma_{2}$.  
\end{proposition}
\begin{proof}
On each component $\Sigma_{i}$ in the decomposition (\ref{decomposition1}), 
we calculate the structure equation of $\hat D$. First, we consider it 
on $\Sigma_{0}$.  
The canonical system $\hat D$ on $U_{\omega_{1}\omega_{2}}$ is given by 
$\hat D=\left\{\varpi_{0}=\varpi_{1}=\varpi_{2}=
\varpi_{\pi_{11}}=\varpi_{\pi_{22}}=0\right\},$ 
where $\varpi_{\pi_{11}}:=\pi_{11}-{p_{11}^{1}}\omega_{1},\   
\varpi_{\pi_{22}}:=\pi_{22}-{p_{22}^{1}}\omega_{2}.$ 
The structure equation of $\hat D$ on $\Sigma_0$ is given by 
\begin{align*}
d\varpi_{i}&\equiv 0 \hspace{0.5cm}(i=0,1,2)\hspace{0.5cm} \quad
\mod\ \varpi_{0},\ \varpi_{1},\ \varpi_{2},\ \varpi_{\pi_{11}},\ \varpi_{\pi_{22}},\\
d\varpi_{\pi_{11}}&\equiv \omega_{1}\wedge ({dp_{11}^{1}}+f\omega_{2})\quad  
\mod\ \varpi_{0},\ \varpi_{1},\ \varpi_{2},\ \varpi_{\pi_{11}},\ \varpi_{\pi_{22}},\\
d\varpi_{\pi_{22}}&\equiv \omega_{2}\wedge ({dp_{22}^{1}}+g\omega_{1})\quad 
\mod\ \varpi_{0},\ \varpi_{1},\ \varpi_{2},\ \varpi_{\pi_{11}},\ \varpi_{\pi_{22}}, 
\end{align*}
using by appropriate functions $f$ and $g$ since $\pi_{11}, \pi_{22}, \omega_1, \omega_2$ are 1-forms on the base manifold $R$. Hence we have 
$\partial \hat D=\left\{\varpi_{0}=\varpi_{1}=
\varpi_{2}=0\right\}={p_{*}^{-1}}(D)$. 
The structure equation of $\partial \hat D$ is written as 
\begin{align*}
d\varpi_{0}&\equiv 0 \hspace{3cm} \quad \mod\ \varpi_{0},\ \varpi_{1},\ \varpi_{2},\\
d\varpi_{1}&\equiv \omega_{1}\wedge \varpi_{\pi_{11}} \hspace{2cm}  
\mod\ \varpi_{0},\ \varpi_{1},\ \varpi_{2}, \varpi_{\pi_{11}}\wedge \varpi_{\pi_{22}},\\
d\varpi_{2}&\equiv \omega_{2}\wedge \varpi_{\pi_{22}} \hspace{2cm}  
\mod\ \varpi_{0},\ \varpi_{1},\ \varpi_{2}, \varpi_{\pi_{11}}\wedge \varpi_{\pi_{22}}. 
\end{align*}
Hence we have $\partial ^{2}\hat D=\partial ^{(2)}\hat D=\left\{\varpi_{0}=0\right\}.$  
The structure equation of $\partial^{2} \hat D$ is described by
\begin{align*}
d\varpi_{0}\equiv \omega_{1}\wedge \varpi_{1}+\omega_{2}\wedge \varpi_{2}
  \ \mod\ &  \varpi_{0},\ \varpi_{1}\wedge \varpi_{2},\ \varpi_{1}\wedge \varpi_{\pi_{11}},\      
 \varpi_{1}\wedge \varpi_{\pi_{22}},\\ & \varpi_{2}\wedge \varpi_{\pi_{11}}, \varpi_{2}\wedge \varpi_{\pi_{22}},  
\varpi_{\pi_{11}}\wedge \varpi_{\pi_{22}}.
\end{align*}
Therefore, we have $\partial ^{(3)}\hat D=T\Sigma(R).$ 
Next, we consider on $\Sigma_{1}$.   
It is sufficient to prove on $U_{\omega_{1}\pi_{22}}$ 
because the differential system $\hat D$ on $U_{\omega_{1}\pi_{22}}$  
is contact equivalent to the differential system $\hat D$ on $U_{\omega_{2}\pi_{11}}$. 
The canonical system $\hat D$ on $U_{\omega_{1}\pi_{22}}$ is given by 
$\hat D=\left\{\varpi_{0}=\varpi_{1}=\varpi_{2}=
\varpi_{\omega_{2}}=\varpi_{\pi_{11}}=0\right\},$ 
where $\varpi_{\omega_{2}}:=\omega_{2}-{p_{12}^{3}}\pi_{22},\   
\varpi_{\pi_{11}}:=\pi_{11}-{p_{21}^{3}}\omega_{1}.$
For a point $w\in U_{\omega_{1}\pi_{22}}$, $w\in \Sigma_{1}$ if and only if 
$p^{3}_{12}(w)=0$. Therefore, it is enough to consider at $w$ in the hypersurface 
$\left\{p^{3}_{12}=0\right\}\subset \Sigma(R)$. 
The structure equation at a point on $\left\{p^{3}_{12}=0\right\}$ is given by  
\begin{align*}
d\varpi_{i}&\equiv 0 \hspace{0.5cm}(i=0,1,2)\hspace{0.5cm} \quad
\mod\ \varpi_{0},\ \varpi_{1},\ \varpi_{2},\ \varpi_{\omega_{2}},\ \varpi_{\pi_{11}},\\
d\varpi_{\omega_{2}}&\equiv \pi_{22}\wedge ({dp_{12}^{3}}+f\omega_{1})\quad  
\mod\ \varpi_{0},\ \varpi_{1},\ \varpi_{2},\ \varpi_{\omega_{2}},\ \varpi_{\pi_{11}},\\
d\varpi_{\pi_{11}}&\equiv \omega_{1}\wedge ({dp_{21}^{3}}+g\pi_{22})\quad 
\mod\ \varpi_{0},\ \varpi_{1},\ \varpi_{2},\ \varpi_{\omega_{2}},\ \varpi_{\pi_{11}}, 
\end{align*}
where $f$ and $g$ are appropriate functions. Hence we have 
$\partial \hat D=\left\{\varpi_{0}=\varpi_{1}=
\varpi_{2}=0\right\}={p_{*}^{-1}}(D).$ 
The structure equation of $\partial \hat D$ at a point on $\left\{p^{3}_{12}=0\right\}$ is expressed as 
\begin{align*}
d\varpi_{0}&\equiv 0 \hspace{3cm} \quad \mod\ \varpi_{0},\ \varpi_{1},\ \varpi_{2},\\
d\varpi_{1}&\equiv \omega_{1}\wedge \varpi_{\pi_{11}} \hspace{2cm}  
\mod\ \varpi_{0},\ \varpi_{1},\ \varpi_{2}, \varpi_{\omega_{2}}\wedge \varpi_{\pi_{11}},\\
d\varpi_{2}&\equiv \varpi_{\omega_{2}}\wedge {\pi_{22}} \hspace{2cm}  
\mod\ \varpi_{0},\ \varpi_{1},\ \varpi_{2}, \varpi_{\omega_{2}}\wedge \varpi_{\pi_{11}}. 
\end{align*}
Hence we have $\partial ^{2}\hat D=\partial ^{(2)}\hat D=\left\{\varpi_{0}=0\right\}.$ 
The structure equation of $\partial^{2} \hat D$ at a point on $\left\{p^{3}_{12}=0\right\}$ is described by 
\begin{align*}
d\varpi_{0}&\equiv \omega_{1}\wedge \varpi_{1} + \omega_{2}\wedge \varpi_{2} \\
 &\equiv \omega_{1}\wedge \varpi_{1} + (\varpi_{\omega_{2}}-p_{12}^{3}\pi_{22})\wedge \varpi_{2} \\
&\equiv \omega_{1}\wedge \varpi_{1} + \varpi_{\omega_{2}}\wedge \varpi_{2} \\
&\equiv \omega_{1}\wedge \varpi_{1} 
\hspace{2cm} \mod\  \varpi_{0},\ \varpi_{1}\wedge \varpi_{2},\ 
\varpi_{1}\wedge \varpi_{\omega_{2}},\ \varpi_{1}\wedge \varpi_{\pi_{11}},\\ 
 & \hspace{5.5cm} \varpi_{2}\wedge \varpi_{\omega_{2}},\ \varpi_{2}\wedge \varpi_{\pi_{11}},\ 
\varpi_{\omega_{2}}\wedge \varpi_{\pi_{11}}.
\end{align*}
Thus, we have $\partial ^{(3)}\hat D=T\Sigma(R).$ 
Finally, we consider on $\Sigma_{2}$. 
The canonical system $\hat D$ on $U_{\pi_{11}\pi_{22}}$ is given by 
$\hat D=\left\{\varpi_{0}=\varpi_{1}=\varpi_{2}=
\varpi_{\omega_{1}}=\varpi_{\omega_{2}}=0\right\},$ 
where $\varpi_{\omega_{1}}:=\omega_{1}-{p_{11}^{6}}\pi_{11},\   
\varpi_{\omega_{2}}:=\omega_{2}-{p_{22}^{6}}\pi_{22}.$
For a point $w\in U_{\pi_{11}\pi_{22}}$, $w\in \Sigma_{2}$ if and only if 
$p^{6}_{11}(w)=p^{6}_{22}(w)=0$. Therefore, we calculate the structure equation 
of $\hat D$ at a point in codimension 2 submanifold 
$\left\{p^{6}_{11}=p^{6}_{22}=0\right\}\subset \Sigma(R)$. 
The structure equation is given by 
\begin{align*}
d\varpi_{i}&\equiv 0 \hspace{0.5cm}(i=0,1,2)\hspace{0.5cm} \quad
\mod\ \varpi_{0},\ \varpi_{1},\ \varpi_{2},\ \varpi_{\omega_{1}},\ \varpi_{\omega_{2}},\\
d\varpi_{\omega_{1}}&\equiv \pi_{11}\wedge ({dp_{11}^{6}}+f\pi_{22})\quad  
\mod\ \varpi_{0},\ \varpi_{1},\ \varpi_{2},\ \varpi_{\omega_{1}},\ \varpi_{\omega_{2}},\\
d\varpi_{\omega_{2}}&\equiv \pi_{22}\wedge ({dp_{22}^{6}}+g\pi_{11})\quad 
\mod\ \varpi_{0},\ \varpi_{1},\ \varpi_{2},\ \varpi_{\omega_{1}},\ \varpi_{\omega_{2}}.
\end{align*}
where $f$ and $g$ are appropriate functions. 
Hence we have 
$\partial \hat D=\left\{\varpi_{0}=\varpi_{1}=
\varpi_{2}=0\right\}={p_{*}^{-1}}(D).$  
The structure equation of $\partial \hat D$ at a point on $\left\{p^{6}_{11}=p^{6}_{22}=0\right\}$ is written as 
\begin{align*}
d\varpi_{0}&\equiv 0 \hspace{3cm} \quad \mod\ \varpi_{0},\ \varpi_{1},\ \varpi_{2},\\
d\varpi_{1}&\equiv \varpi_{\omega_{1}}\wedge {\pi_{11}} \hspace{2cm}  
\mod\ \varpi_{0},\ \varpi_{1},\ \varpi_{2}, \varpi_{\omega_{1}}\wedge \varpi_{\omega_{2}},\\
d\varpi_{2}&\equiv \varpi_{\omega_{2}}\wedge {\pi_{22}} \hspace{2cm}  
\mod\ \varpi_{0},\ \varpi_{1},\ \varpi_{2}, \varpi_{\omega_{1}}\wedge \varpi_{\omega_{2}}. 
\end{align*}
Hence, we have $\partial ^{2}\hat D=\partial ^{(2)}\hat D=\left\{\varpi_{0}=0\right\}.$
The structure equation of $\partial ^{2}\hat D$ at a point on $\left\{p^{6}_{11}=p^{6}_{22}=0\right\}$ is described by  
\begin{align*}
d\varpi_{0}\equiv 0\ 
\mod\ & \varpi_{0},\ \varpi_{1}\wedge \varpi_{2},\ 
\varpi_{1}\wedge \varpi_{\omega_{1}},\ \varpi_{1}\wedge \varpi_{\omega_{2}},\\  
 & \varpi_{2}\wedge \varpi_{\omega_{1}},\ \varpi_{2}\wedge \varpi_{\omega_{2}},\ 
\varpi_{\omega_{1}}\wedge \varpi_{\omega_{2}}.
\end{align*}
Therefore we obtain $\partial ^{(3)}\hat D=\partial ^{(2)}\hat D.$
\end{proof}
From the above proposition, $(\Sigma(R), \hat D)$ is locally weakly regular around 
$w\in \Sigma_{0}\cup \Sigma_{1}$. So we can define the symbol algebra at $w$ 
in the sense of Tanaka and the following 
holds:  
\begin{proposition}\label{hyp-symbol1}
For $w\in \Sigma_{0}$, the symbol algebra $\mathfrak m_{0}(w)$ 
is isomorphic to $\mathfrak m_{0}$, where  
$\mathfrak m_{0}=\mathfrak g_{-4}\oplus\mathfrak g_{-3}\oplus\mathfrak g_{-2}\oplus\mathfrak g_{-1},$  
whose bracket relations are given by 
$$[X_{p_{11}^{1}},\ X_{\omega_{1}}]=X_{\pi_{11}},\quad  
[X_{p_{22}^{1}},\ X_{\omega_{2}}]=X_{\pi_{22}},\quad    
[X_{\pi_{11}},\ X_{\omega_{1}}]=X_{1},$$ 
$$[X_{\pi_{22}},\ X_{\omega_{2}}]=X_{2},\quad 
[X_{1},\ X_{\omega_{1}}]=[X_{2},\ X_{\omega_{2}}]=X_{0},$$
and the other brackets are trivial.\\ Here $\left\{X_{0},\ X_{1},\ X_{2},\ X_{p_{11}^{1}},\ 
X_{p_{22}^{1}},\ X_{\omega_{1}},\ X_{\omega_{2}},\ 
X_{\pi_{11}},\ X_{\pi_{22}}\right\}$ is a basis of $\mathfrak m_{0}$ and 
\begin{align*}
\mathfrak g_{-1}=\left\{X_{\omega_{1}},\ X_{\omega_{2}},\ X_{p_{11}^{1}},\ X_{p_{22}^{1}}\right\},\  
\mathfrak g_{-2}=\left\{X_{\pi_{11}},\ X_{\pi_{22}}\right\},\ 
\mathfrak g_{-3}=\left\{X_{1},\ X_{2}\right\},\  
\mathfrak g_{-4}=\left\{X_{0}\right\}.
\end{align*} 
For $w\in \Sigma_{1}$, the symbol algebra $\mathfrak m_{1}(w)$ 
is isomorphic to $\mathfrak m_{1}$, where  
$\mathfrak m_{1}=\mathfrak g_{-4}\oplus\mathfrak g_{-3}\oplus\mathfrak g_{-2}\oplus\mathfrak g_{-1}$ 
whose bracket relations are given by
$$[X_{p_{12}^{3}},\ X_{\pi_{22}}]=X_{\omega_{2}},\quad 
[X_{p_{21}^{3}},\ X_{\omega_{1}}]=X_{\pi_{11}},\quad   
[X_{\pi_{11}},\ X_{\omega_{1}}]=X_{1},$$ 
$$[X_{\pi_{22}},\ X_{\omega_{2}}]=X_{2},\quad 
[X_{1},\ X_{\omega_{1}}]=X_{0},$$ 
and the other brackets are trivial.\\ Here  
$\left\{X_{0},\ X_{1},\ X_{2},\ X_{p_{12}^{3}},\  
X_{p_{21}^{3}},\ X_{\omega_{1}},\ X_{\omega_{2}},\ 
X_{\pi_{11}},\ X_{\pi_{22}}\right\}$ is a basis of $\mathfrak m_{1}$ and 
\begin{align*}
\mathfrak g_{-1}=\left\{X_{\omega_{1}},\ X_{\pi_{22}},\ X_{p_{12}^{3}},\ X_{p_{21}^{3}}\right\},\  
\mathfrak g_{-2}=\left\{X_{\omega_{2}},\ X_{\pi_{11}}\right\},\ 
\mathfrak g_{-3}=\left\{X_{1},\ X_{2}\right\},\ 
\mathfrak g_{-4}=\left\{X_{0}\right\}. 
\end{align*}
\end{proposition}
\begin{proof}
We first show that $\mathfrak m_{0}(w)\cong \mathfrak m_0$. 
On $U_{\omega_{1}\omega_{2}}$ in the proof of Proposition \ref{hyp-derived}, 
we set $\varpi_{p_{11}^{1}}:=d{p_{11}^{1}}+f\omega_{2}
,\ \varpi_{p_{22}^{1}}:=d{p_{22}^{1}}+g\omega_{1}$ and 
take a coframe:\\  
$\left\{\varpi_{0},\ \varpi_{1},\ \varpi_{2},\ \varpi_{\pi_{11}},\ \varpi_{\pi_{22}},\ 
\omega_{1},\ \omega_{2},\ \varpi_{p_{11}^{1}},\ \varpi_{p_{22}^{1}} \right\}$, 
then the structure equations are given by 
\begin{align*}
d\varpi_{i}&\equiv 0 \hspace{0.5cm}(i=0,1,2)\hspace{0.5cm} \quad
\mod\ \varpi_{0},\ \varpi_{1},\ \varpi_{2},\ \varpi_{\pi_{11}},\ \varpi_{\pi_{22}},\\
d\varpi_{\pi_{11}}&\equiv \omega_{1}\wedge \varpi_{p_{11}^{1}} \hspace{2cm}  
\mod\ \varpi_{0},\ \varpi_{1},\ \varpi_{2},\ \varpi_{\pi_{11}},\ \varpi_{\pi_{22}},\\
d\varpi_{\pi_{22}}&\equiv \omega_{2}\wedge \varpi_{p_{22}^{1}} \hspace{2cm}  
\mod\ \varpi_{0},\ \varpi_{1},\ \varpi_{2},\ \varpi_{\pi_{11}},\ \varpi_{\pi_{22}}, 
\end{align*} 
\begin{align*}
d\varpi_{0}&\equiv 0 \hspace{4.5cm}  \mod\ \varpi_{0},\ \varpi_{1},\ \varpi_{2},\\
d\varpi_{1}&\equiv \omega_{1}\wedge \varpi_{\pi_{11}} \hspace{3cm}  
\mod\ \varpi_{0},\ \varpi_{1},\ \varpi_{2}, \varpi_{\pi_{11}}\wedge \varpi_{\pi_{22}},\\
d\varpi_{2}&\equiv \omega_{2}\wedge \varpi_{\pi_{22}} \hspace{3cm}  
\mod\ \varpi_{0},\ \varpi_{1},\ \varpi_{2}, \varpi_{\pi_{11}}\wedge \varpi_{\pi_{22}}. 
\end{align*}
\begin{align*}
d\varpi_{0}\equiv \omega_{1}\wedge \varpi_{1}+\omega_{2}\wedge \varpi_{2}
  \ \mod\ &  \varpi_{0},\ \varpi_{1}\wedge \varpi_{2},\ \varpi_{1}\wedge \varpi_{\pi_{11}},\      
 \varpi_{1}\wedge \varpi_{\pi_{22}},\\ & \varpi_{2}\wedge \varpi_{\pi_{11}}, \varpi_{2}\wedge \varpi_{\pi_{22}},  
\varpi_{\pi_{11}}\wedge \varpi_{\pi_{22}}.
\end{align*}
We take the dual frame 
$\left\{X_{0},\ X_{1},\ X_{2},\ X_{\pi_{11}},\ X_{\pi_{22}},\ 
X_{\omega_{1}},\ X_{\omega_{2}},\ X_{p_{11}^{1}},\ X_{p_{22}^{1}} \right\}$, 
and set \\
$[X_{\omega_{1}}, X_{p_{11}^{1}}]=A_{11}X_{\pi_{11}}+A_{22}X_{\pi_{22}},\
 (A_{ii}\in \mathbb R).$ 
Then we have 
\begin{align*}
d\varpi_{\pi_{11}}(X_{\omega_{1}}, X_{p_{11}^{1}})&
=X_{\omega_{1}}\varpi_{\pi_{11}}(X_{p_{11}^{1}})-X_{p_{11}^{1}}\varpi_{\pi_{11}}(X_{\omega_{1}})
 -\varpi_{\pi_{11}}([X_{\omega_{1}}, X_{p_{11}^{1}}]),\\
 &=-\varpi_{\pi_{11}}([X_{\omega_{1}}, X_{p_{11}^{1}}])=-A_{11}. 
\end{align*}
On the other hand, we have  
\begin{align*}
d\varpi_{\pi_{11}}(X_{\omega_{1}}, X_{p_{11}^{1}})&
=\omega_{1}(X_{\omega_{1}})\varpi_{p_{11}^{1}}(X_{p_{11}^{1}})
-\varpi_{p_{11}^{1}}(X_{\omega_{1}})\omega_{1}(X_{p_{11}^{1}})=1.
\end{align*}
Therefore $A_{11}=-1$. From the same argument for $d\varpi_{\pi_{22}}$, we get $A_{22}=0$. 
Hence we have 
$[X_{\omega_{1}}, X_{p_{11}^{1}}]=-X_{\pi_{11}}.$ 
The other brackets are left to reader. 
Hence its dual frame satisfies the relation with respect to the algebra $\mathfrak m_{0}$.\par 
Next, we show that the isomorphism $\mathfrak m_{1}(w)\cong \mathfrak m_{1}$. 
On $U_{\omega_{1}\pi_{22}}$ in the proof of Proposition \ref{hyp-derived}, 
we set $\varpi_{p_{12}^{3}}:=d{p_{12}^{3}}+f\omega_{1}
,\ \varpi_{p_{21}^{3}}:=d{p_{21}^{3}}+g\pi_{22}$, and 
take a coframe  and  its dual frame
$\left\{\varpi_{0}, \varpi_{1}, \varpi_{2}, \varpi_{\omega_{2}}, 
\varpi_{\pi_{11}}, \omega_{1}, \pi_{22}, \varpi_{p_{12}^{3}}, 
\varpi_{p_{21}^{3}} \right\},$\   
$\left\{X_{0}, X_{1}, X_{2}, X_{\omega_{2}}, X_{\pi_{11}},    
X_{\omega_{1}}, X_{\pi_{22}}, X_{p_{12}^{3}}, X_{p_{21}^{3}} \right\}.$  
From the proof of Proposition \ref{hyp-derived}, the structure equations at a point on $\left\{p^{3}_{12}=0\right\}$ are  
\begin{align*}
d\varpi_{i}&\equiv 0 \hspace{0.5cm}(i=0,1,2)\hspace{0.9cm} \quad
\mod\ \varpi_{0},\ \varpi_{1},\ \varpi_{2},\ \varpi_{\omega_{2}},\ \varpi_{\pi_{11}},\\
d\varpi_{\omega_{2}}&\equiv \pi_{22}\wedge \varpi_{p_{12}^{3}} \hspace{2.2cm}  
\mod\ \varpi_{0},\ \varpi_{1},\ \varpi_{2},\ \varpi_{\omega_{2}},\ \varpi_{\pi_{11}},\\
d\varpi_{\pi_{11}}&\equiv \omega_{1}\wedge \varpi_{p_{21}^{3}} \hspace{2.3cm}
\mod\ \varpi_{0},\ \varpi_{1},\ \varpi_{2},\ \varpi_{\omega_{2}},\ \varpi_{\pi_{11}}, 
\end{align*}
\begin{align*}
d\varpi_{0}&\equiv 0 \hspace{4.5cm}  \mod\ \varpi_{0},\ \varpi_{1},\ \varpi_{2},\\
d\varpi_{1}&\equiv \omega_{1}\wedge \varpi_{\pi_{11}} \hspace{3cm}  
\mod\ \varpi_{0},\ \varpi_{1},\ \varpi_{2}, \varpi_{\omega_{2}}\wedge \varpi_{\pi_{11}},\\
d\varpi_{2}&\equiv \varpi_{\omega_{2}}\wedge {\pi_{22}} \hspace{3cm}  
\mod\ \varpi_{0},\ \varpi_{1},\ \varpi_{2}, \varpi_{\omega_{2}}\wedge \varpi_{\pi_{11}}. 
\end{align*}
\begin{align*}
d\varpi_{0}\equiv \omega_{1}\wedge \varpi_{1}\ 
\mod\ & \varpi_{0},\ \varpi_{1}\wedge \varpi_{2},\ 
\varpi_{1}\wedge \varpi_{\omega_{2}},\ \varpi_{1}\wedge \varpi_{\pi_{11}},\\ 
& \varpi_{2}\wedge \varpi_{\omega_{2}},\ \varpi_{2}\wedge \varpi_{\pi_{11}},\ 
\varpi_{\omega_{2}}\wedge \varpi_{\pi_{11}}.
\end{align*}
Thus we obtain the statement for $\mathfrak m_{1}$ from the same argument of the proof of 
$\mathfrak m_{0}$. 
\end{proof}
In the rest of this hyperbolic case, we calculate the symbol algebra at a point $w$ 
in $\Sigma_{2}$. From Proposition \ref{hyp-derived},  
$D$ is not weakly regular around $w\in \Sigma_{2}$. 
Hence, at the point $w$, we can not define the symbol algebra in the sense of Tanaka. 
However, by taking the following filtration $F$ on $\Sigma(R)$, 
we can define the symbol algebra $\mathfrak m_{2}(w)$ of $(\Sigma(R), F)$ at $w\in \Sigma_{2}$. 
We set 
$F^{-4}(w)=T_{w}(\Sigma(R)),\ F^{-3}(w)=\partial ^{(2)}D(w),\   
F^{-2}(w)=\partial D(w),\ F^{-1}(w)=D(w)$,
where $w\in \Sigma(R)$. 
Then, $\left\{F^{p}\right\}$ defines the filtration on $\Sigma(R)$. 
For $w\in \Sigma_{2}$, we set 
$\mathfrak g_{-1}(w):=F^{-1}(w)=D(w)$, $\mathfrak g_{-2}(w):=F^{-2}(w)/F^{-1}(w)$, 
$\mathfrak g_{-3}(w):=F^{-3}(w)/F^{-2}(w)$, $\mathfrak g_{-4}(w):=T_{w}(\Sigma(R))/F^{-3}(w)$, and  
$$\mathfrak m_{2}(w)=\mathfrak g_{-1}(w)\oplus \mathfrak g_{-2}(w)\oplus
\mathfrak g_{-3}(w)\oplus \mathfrak g_{-4}(w).$$ 
The way of the definition of the above symbol algebra in the sense of Morimoto 
coincides with the usual symbol algebra except for $[\mathfrak g_{-1}, \mathfrak g_{-3}]$.
\begin{proposition}\label{hyp-symbol2}
For $w\in \Sigma_{2}$, the symbol algebra $\mathfrak m_{2}(w)$ is 
isomorphic to $\mathfrak m_{2}$, where   
$\mathfrak m_{2}=\mathfrak g_{-4}\oplus\mathfrak g_{-3}\oplus\mathfrak g_{-2}\oplus\mathfrak g_{-1},$ 
whose bracket relations are given by 
$$[X_{p_{11}^{6}},\ X_{\pi_{11}}]=X_{\omega_{1}},\ 
[X_{p_{22}^{6}},\ X_{\pi_{22}}]=X_{\omega_{2}},\ 
[X_{\pi_{11}},\ X_{\omega_{1}}]=X_{1},\ 
[X_{\pi_{22}},\ X_{\omega_{2}}]=X_{2},$$
and the other brackets are trivial.\\
Here $\left\{X_{0},\ X_{1},\ X_{2},\ X_{p_{11}^{6}},\ 
X_{p_{22}^{6}},\ X_{\omega_{1}},\ X_{\omega_{2}},\ 
X_{\pi_{11}},\ X_{\pi_{22}}\right\}$ is a basis of $\mathfrak m_{2}$ and 
\begin{align*}
\mathfrak g_{-1}=\left\{X_{\pi_{11}},\ X_{\pi_{22}},\ X_{p_{11}^{6}},\ X_{p_{22}^{6}}\right\},\ 
\mathfrak g_{-2}=\left\{X_{\omega_{1}},\ X_{\omega_{2}}\right\},\ 
\mathfrak g_{-3}=\left\{X_{1},\ X_{2}\right\},\ 
\mathfrak g_{-4}=\left\{X_{0}\right\}.
\end{align*}
\end{proposition}
\begin{proof}
On $U_{\pi_{11}\pi_{22}}$ in the proof of Proposition \ref{hyp-derived}, 
we set $\varpi_{p_{11}^{6}}:=d{p_{11}^{6}}+f\pi_{22}
,\ \varpi_{p_{22}^{6}}:=d{p_{22}^{6}}+g\pi_{11}$ and 
take a coframe: \\
$\left\{\varpi_{0},\ \varpi_{1},\ \varpi_{2},\ \varpi_{\omega_{1}},\ \varpi_{\omega_{2}},\ 
\pi_{11},\ \pi_{22},\ \varpi_{\pi_{11}^{6}},\ \varpi_{\pi_{22}^{6}} \right\}$ and its  
dual frame:\\ 
$\left\{X_{0},\ X_{1},\ X_{2},\ X_{\omega_{1}},\ X_{\omega_{2}},\ 
X_{\pi_{11}},\ X_{\pi_{22}},\ X_{p_{11}^{6}},\ X_{p_{22}^{6}} \right\}$.    
From the proof of Proposition \ref{hyp-derived}, the structure equations at a point on 
$\left\{p^{6}_{11}=p^{6}_{22}=0\right\}$ are 
\begin{align*}
d\varpi_{i}&\equiv 0 \hspace{0.5cm}(i=0,1,2)\hspace{0.7cm} \quad
\mod\ \varpi_{0},\ \varpi_{1},\ \varpi_{2},\ \varpi_{\omega_{1}},\ \varpi_{\omega_{2}},\\
d\varpi_{\omega_{1}}&\equiv \pi_{11}\wedge \varpi_{p_{11}^{6}} \hspace{2cm} 
\mod\ \varpi_{0},\ \varpi_{1},\ \varpi_{2},\ \varpi_{\omega_{1}},\ \varpi_{\omega_{2}},\\
d\varpi_{\omega_{2}}&\equiv \pi_{22}\wedge \varpi_{p_{22}^{6}} \hspace{2cm}
\mod\ \varpi_{0},\ \varpi_{1},\ \varpi_{2},\ \varpi_{\omega_{1}},\ \varpi_{\omega_{2}}.
\end{align*}
\begin{align*}
d\varpi_{0}&\equiv 0 \hspace{4cm} \quad \mod\ \varpi_{0},\ \varpi_{1},\ \varpi_{2},\\
d\varpi_{1}&\equiv \varpi_{\omega_{1}}\wedge {\pi_{11}} \hspace{3cm}  
\mod\ \varpi_{0},\ \varpi_{1},\ \varpi_{2}, \varpi_{\omega_{1}}\wedge \varpi_{\omega_{2}},\\
d\varpi_{2}&\equiv \varpi_{\omega_{2}}\wedge {\pi_{22}} \hspace{3cm}  
\mod\ \varpi_{0},\ \varpi_{1},\ \varpi_{2}, \varpi_{\omega_{1}}\wedge \varpi_{\omega_{2}}. 
\end{align*}
\begin{align*}
d\varpi_{0}\equiv 0\ 
\mod\ & \varpi_{0},\ \varpi_{1}\wedge \varpi_{2},\ 
\varpi_{1}\wedge \varpi_{\omega_{1}},\ \varpi_{1}\wedge \varpi_{\omega_{2}},\\  
 & \varpi_{2}\wedge \varpi_{\omega_{1}},\ \varpi_{2}\wedge \varpi_{\omega_{2}},\ 
\varpi_{\omega_{1}}\wedge \varpi_{\omega_{2}}.
\end{align*}
Thus we have the assertion by the same argument 
in the proof of Proposition \ref{hyp-symbol1}. 
\end{proof}
\noindent{\bf Structures of rank 2 prolongations for parabolic equations.} 
Let $(R,D)$ be a locally parabolic equation, and $(\Sigma(R),\hat D)$
be the rank 2 prolongation. We use the geometric decomposition (\ref{decomposition1}) of  $\Sigma(R)$  which is similar to the hyperbolic case. 
From Lemma \ref{par-lem}, locally, we have 
$\Sigma_{0}|_{p^{-1}(U)}=P_{\omega_{1}\omega_{2}},\ 
\Sigma_{1}|_{p^{-1}(U)}=P_{\omega_{1}\pi_{22}}\backslash P_{\omega_{1}\omega_{2}}$, and 
$\Sigma_{2}|_{p^{-1}(U)}=P_{\pi_{12}\pi_{22}}\backslash 
(P_{\omega_{1}\omega_{2}}\cup P_{\omega_{1}\pi_{22}})$,  
where $p$ is the projection of the fibration $\Sigma(R) \to R$. 
The set $\Sigma_{0}$ is an open set in $\Sigma(R)$, and is a $\mathbb{R}^2$--bundle over $R$. 
The set $\Sigma_1$ is a submanifold in $J(D,2)$ and contains singular points of $\Sigma(R)$ in $J(D,2)$ and is a $\mathbb{R}$-bundle over $R$. The set $\Sigma_{2}$ is codimension 2 submanifold in $\Sigma(R)$, and is a section of $\Sigma(R) \to R$. 
\begin{remark}
From now on, we examine the geometric structures of $(\Sigma(R), \hat D)$ 
on a domain except for singular points in $\Sigma_{1}$.  
\end{remark}
\begin{proposition}\label{par-derived}
The differential system $\hat D$ on $\Sigma(R)$ is regular, but 
is not weakly regular. More precisely, we obtain that 
$\hat D\subset \partial \hat D\subset \partial^{2}\hat D \subset 
\partial^{3}\hat D=T\Sigma(R).$ 
Moreover, we have $\partial^{2}\hat D=\partial^{(2)}\hat D$,  
$\partial^{(3)}\hat D=T\Sigma(R)\ \ {\rm on}\ 
\Sigma_{0}\cup\Sigma_{1}$, and  
$\partial^{(3)}\hat D=\partial^{(2)}\hat D\ \ {\rm on}\ \Sigma_{2}$.  
\end{proposition}
\begin{proof}
On each component $\Sigma_{i}$ in the decomposition,  
we calculate the structure equation of $\hat D$. 
First, we consider it on $\Sigma_{0}$. 
The canonical system $\hat D$ on $U_{\omega_{1}\omega_{2}}$ is given by 
$\hat D =\left\{\varpi_{0}=\varpi_{1}=\varpi_{2}
 =\varpi_{\pi_{12}}=\varpi_{\pi_{22}}=0\right\},$ 
where 
$\varpi_{\pi_{12}}:=\pi_{12}-{p_{12}^{1}}\omega_{2},\quad  
\varpi_{\pi_{22}}:=\pi_{22}-{p_{12}^{1}}\omega_{1}-{p_{22}^{1}}\omega_{2}.$  
The structure equation of $\hat D$ on $\Sigma_0$ is written as 
\begin{align*}
d\varpi_{i}&\equiv 0 \hspace{0.5cm}(i=0,1,2)\hspace{0.5cm} \quad
\mod\ \varpi_{0},\ \varpi_{1},\ \varpi_{2},\ \varpi_{\pi_{12}},\ \varpi_{\pi_{22}},\\
d\varpi_{\pi_{12}}&\equiv \omega_{2}\wedge ({dp_{12}^{1}}+f\omega_{1})\quad  
\mod\ \varpi_{0},\ \varpi_{1},\ \varpi_{2},\ \varpi_{\pi_{12}},\ \varpi_{\pi_{22}},\\
d\varpi_{\pi_{22}}&\equiv g\omega_{1}\wedge \omega_{2}-{dp_{12}^{1}}\wedge \omega_{1}
                    -{dp_{22}^{1}}\wedge \omega_{2} \quad 
\mod\ \varpi_{0},\ \varpi_{1},\ \varpi_{2},\ \varpi_{\pi_{12}},\ \varpi_{\pi_{22}}.\\
 &\equiv -(d{p_{12}^{1}}+f\omega_{1})\wedge \omega_{1}
 -(d{p_{22}^{1}}-g\omega_{1})\wedge \omega_{2}.
\end{align*}
where $f$ and $g$ are appropriate functions. Hence we have 
$\partial \hat D=\left\{\varpi_{0}=\varpi_{1}
=\varpi_{2}=0\right\}={p_{*}^{-1}}(D).$ 
The structure equation of $\partial \hat D$ is expressed as 
\begin{align*}
d\varpi_{0}&\equiv 0 \hspace{5cm} \quad \mod\ \varpi_{0},\ \varpi_{1},\ \varpi_{2},\\
d\varpi_{1}&\equiv \omega_{2}\wedge \varpi_{\pi_{12}} \hspace{3.9cm}  
\mod\ \varpi_{0},\ \varpi_{1},\ \varpi_{2}, \varpi_{\pi_{12}}\wedge \varpi_{\pi_{22}},\\
d\varpi_{2}&\equiv \omega_{1}\wedge \varpi_{\pi_{12}}+\omega_{2}\wedge \varpi_{\pi_{22}} \hspace{1.7cm}  
\mod\ \varpi_{0},\ \varpi_{1},\ \varpi_{2}, \varpi_{\pi_{12}}\wedge \varpi_{\pi_{22}},
\end{align*}
Hence we have 
$\partial ^{2}\hat D=\partial ^{(2)}\hat D=\left\{\varpi_{0}=0\right\}.$ 
The structure equation of $\partial^{2} \hat D$ is described by 
\begin{align*}
d\varpi_{0}\equiv \omega_{1}\wedge \varpi_{1}+\omega_{2}\wedge \varpi_{2},\ 
 \mod\ & \varpi_{0},\ \varpi_{1}\wedge \varpi_{2},\ 
\varpi_{1}\wedge \varpi_{\pi_{12}},\ \varpi_{1}\wedge \varpi_{\pi_{22}},\\
& \varpi_{2}\wedge \varpi_{\pi_{12}},\ \varpi_{2}\wedge \varpi_{\pi_{22}}
,\ \varpi_{\pi_{12}}\wedge \varpi_{\pi_{22}}.
\end{align*}
Therefore, we obtain $\partial ^{(3)}\hat D=T\Sigma(R).$ 
Next, we consider on $\Sigma_{1}$. 
It is enough to work on $U_{\pi_{12}\pi_{22}}$ since $\Sigma_{1}\backslash 
\{\textrm{singular\ points}\}$ is covered by $U_{\pi_{12}\pi_{22}}$. 
The canonical system $\hat D$ on $U_{\pi_{12}\pi_{22}}$ is given by 
$\hat D=\left\{\varpi_{0}=\varpi_{1}=\varpi_{2}=
\varpi_{\omega_{1}}=\varpi_{\omega_{2}}=0\right\},$
where 
$\varpi_{\omega_{1}}:=\omega_{1}-{p_{11}^{6}}\pi_{12}-{p_{12}^{6}}\pi_{22},\quad  
\varpi_{\omega_{2}}:=\omega_{2}-{p_{12}^{6}}\pi_{12}.$ 
For $w \in U_{\pi_{12}\pi_{22}}$, $w\in \Sigma_{1}$ 
if and only if $p_{11}^{6}(w)\not=0, p_{12}^{6}(w)=0$. Because, 
$w\in \Sigma_{2}$ is given by the coordinate 
$p_{11}^{6}(w)=0, p_{12}^{6}(w)=0$, and $w\in \Sigma_1 \backslash \Sigma_0$ 
is given by $p_{12}^{6}(w)=0$. Therefore, we calculate the structure equation 
at $w$ in the hypersurface $\left\{p_{11}^{6}\not=0, p_{12}^{6}=0\right\}\subset \Sigma(R)$. 
The structure equation at a point on $\left\{p_{11}^{6}\not=0, p_{12}^{6}=0\right\}$ is 
\begin{align*}
d\varpi_{i}&\equiv 0 \hspace{0.5cm}(i=0,1,2)\hspace{0.5cm} \quad
\mod\ \varpi_{0},\ \varpi_{1},\ \varpi_{2},\ \varpi_{\omega_{1}},\ \varpi_{\omega _{2}},\\
d\varpi_{\omega_{1}}&\equiv \pi_{12}\wedge ({dp_{11}^{6}}+f\pi_{22})+\pi_{22}\wedge 
          ({dp_{12}^{6}}+g\pi_{22})\quad  
            \mod\ \varpi_{0},\ \varpi_{1},\ \varpi_{2},\ \varpi_{\omega_{1}},\ \varpi_{\omega_{2}},\\
d\varpi_{\omega_{2}}&\equiv \pi_{12}\wedge ({dp_{12}^{6}}+g\pi_{22})\quad 
\mod\ \varpi_{0},\ \varpi_{1},\ \varpi_{2},\ \varpi_{\omega_{1}},\ \varpi_{\omega_{2}}.
\end{align*}
where $f$ and $g$ are appropriate functions. Hence we have  
$\partial \hat D=\left\{\varpi_{0}=\varpi_{1}=
\varpi_{2}=0\right\}={p_{*}^{-1}}(D).$ 
The structure equation of $\partial \hat D$ is given by  
\begin{align*}
d\varpi_{0}&\equiv 0 \hspace{4.8cm} \ \mod\ \varpi_{0},\ \varpi_{1},\ \varpi_{2},\\
d\varpi_{1}&\equiv \varpi_{\omega_{2}}\wedge \pi_{12} \hspace{3.5cm}  
\mod\ \varpi_{0},\ \varpi_{1},\ \varpi_{2}, \varpi_{\omega_{1}}\wedge \varpi_{\omega_{2}},\\
d\varpi_{2}&\equiv \varpi_{\omega_{1}}\wedge \pi_{12}+\varpi_{\omega_{2}}\wedge \pi_{22} \hspace{1.3cm}  
\mod\ \varpi_{0},\ \varpi_{1},\ \varpi_{2}, \varpi_{\omega_{1}}\wedge \varpi_{\omega_{2}}. 
\end{align*}
Hence we have 
$\partial ^{2}\hat D=\partial ^{(2)}\hat D=\left\{\varpi_{0}=0\right\}$. 
The structure equation of $\partial^{2} \hat D$ is written as  
\begin{align*}
d\varpi_{0}\equiv p_{11}^{6}\pi_{12}\wedge \varpi_{1}\ 
\mod\ & \varpi_{0},\ \varpi_{1}\wedge \varpi_{2},\ 
\varpi_{1}\wedge \varpi_{\omega_{1}},\ \varpi_{1}\wedge \varpi_{\omega_{2}},\\  
 & \varpi_{2}\wedge \varpi_{\omega_{1}},\ \varpi_{2}\wedge \varpi_{\omega_{2}},\ 
\varpi_{\omega_{1}}\wedge \varpi_{\omega_{2}}.
\end{align*}
Here, if we set $\varpi_{0}^{\prime}:=\varpi_{0}/{p_{11}^{6}}$, 
then structure equation is rewritten in the form:
\begin{align*}
d\varpi_{0}^{\prime}\equiv \pi_{12}\wedge \varpi_{1}\ 
\mod\ & \varpi_{0}^{\prime},\ \varpi_{1}\wedge \varpi_{2},\ 
\varpi_{1}\wedge \varpi_{\omega_{1}},\ \varpi_{1}\wedge \varpi_{\omega_{2}},\\ 
 & \varpi_{2}\wedge \varpi_{\omega_{1}},\ \varpi_{2}\wedge \varpi_{\omega_{2}},\ 
\varpi_{\omega_{1}}\wedge \varpi_{\omega_{2}}.
\end{align*}
Hence we have $\partial ^{(3)}\hat D=T\Sigma(R).$ 
Finally, we consider on $\Sigma_2$. 
We use the coordinate on $U_{\pi_{12}\pi_{22}}$. 
For $w \in U_{\pi_{12}\pi_{22}}$, $w\in \Sigma_{2}$ 
if and only if $p_{11}^{6}(w)=p_{12}^{6}(w)=0$. 
Therefore, we calculate the structure equation 
at $w$ in the codimension 2 submanifold $\left\{p_{11}^{6}=p_{12}^{6}=0\right\}\subset \Sigma(R)$. 
The structure equation at a point on 
$\left\{p_{11}^{6}=p_{12}^{6}=0\right\}$ is described by 
\begin{align*}
d\varpi_{i}&\equiv 0 \hspace{0.5cm}(i=0,1,2)\hspace{0.5cm} 
\mod\ \varpi_{0},\ \varpi_{1},\ \varpi_{2},\ \varpi_{\omega_{1}},\ \varpi_{\omega _{2}},\\
d\varpi_{\omega_{1}}&\equiv \pi_{12}\wedge ({dp_{11}^{6}}+f\pi_{22})+\pi_{22}\wedge 
          ({dp_{12}^{6}}+g\pi_{22}) \quad  
            \mod\ \varpi_{0},\ \varpi_{1},\ \varpi_{2},\ \varpi_{\omega_{1}},\ \varpi_{\omega_{2}},\\
d\varpi_{\omega_{2}}&\equiv \pi_{12}\wedge ({dp_{12}^{6}}+g\pi_{22})\quad 
\mod\ \varpi_{0},\ \varpi_{1},\ \varpi_{2},\ \varpi_{\omega_{1}},\ \varpi_{\omega_{2}}.
\end{align*}
where $f$ and $g$ are appropriate functions. Hence we have 
$\partial \hat D=\left\{\varpi_{0}=\varpi_{1}
=\varpi_{2}=0\right\}={p_{*}^{-1}}(D).$ 
The structure equation of $\partial \hat D$ is given by  
\begin{align*}
d\varpi_{0}&\equiv 0 \hspace{4.8cm} \quad \mod\ \varpi_{0},\ \varpi_{1},\ \varpi_{2},\\
d\varpi_{1}&\equiv \varpi_{\omega_{2}}\wedge {\pi_{12}} \hspace{3.7cm}  
\mod\ \varpi_{0},\ \varpi_{1},\ \varpi_{2}, \varpi_{\omega_{1}}\wedge \varpi_{\omega_{2}},\\
d\varpi_{2}&\equiv \varpi_{\omega_{1}}\wedge {\pi_{12}}+\varpi_{\omega_{2}}\wedge {\pi_{22}} \hspace{1.4cm}  
\mod\ \varpi_{0},\ \varpi_{1},\ \varpi_{2}, \varpi_{\omega_{1}}\wedge \varpi_{\omega_{2}}. 
\end{align*}
Therefore, we get 
$\partial^{2}\hat D=\partial ^{(2)}\hat D=\left\{\varpi_{0}=0\right\}.$  
The structure equation of $\partial^{2} \hat D$ is expressed as 
\begin{align*}
d\varpi_{0}\equiv 0\ 
\mod\ & \varpi_{0},\ \varpi_{1}\wedge \varpi_{2},\ 
\varpi_{1}\wedge \varpi_{\omega_{1}},\ \varpi_{1}\wedge \varpi_{\omega_{2}},\\  
 & \varpi_{2}\wedge \varpi_{\omega_{1}},\ \varpi_{2}\wedge \varpi_{\omega_{2}},\ 
\varpi_{\omega_{1}}\wedge \varpi_{\omega_{2}}.
\end{align*}
Hence we have $\partial ^{(3)}\hat D=\partial ^{(2)}\hat D.$  
\end{proof} 
From the above proposition, $(\Sigma(R), \hat D)$ is locally 
weakly regular around $w \in \Sigma_{0}\cup \Sigma_{1}.$ So 
we define the symbol algebra at $w$. On the other hand, for a point 
$w$ on $\Sigma_{2}$, $(\Sigma(R), \hat D)$ is not weakly regular around $w$. 
However, by taking the filtration on $\Sigma(R)$ which is same to the 
hyperbolic case, we can define the symbol algebra at $w$. 
Each structure of symbol algebras is given in the following.  
\begin{proposition}\label{par-symbol}
For $w\in \Sigma_{0}$, the symbol algebra $\mathfrak m_{0}(w)$ 
is isomorphic to $\mathfrak m_{0}$, where  
$\mathfrak m_{0}=\mathfrak g_{-4}\oplus\mathfrak g_{-3}\oplus\mathfrak g_{-2}\oplus\mathfrak g_{-1},$ 
whose bracket relations are given by 
$$[X_{p_{12}^{1}},\ X_{\omega_{2}}]=X_{\pi_{12}},\quad
[X_{p_{12}^{1}},\ X_{\omega_{1}}]=[X_{p_{22}^{1}},\ X_{\omega_{2}}]=X_{\pi_{22}},\quad 
[X_{\pi_{12}},\ X_{\omega_{2}}]=X_{1},$$  
$$[X_{\pi_{12}},\ X_{\omega_{1}}]=[X_{\pi_{22}},\ X_{\omega_{2}}]=X_{2},\quad 
[X_{1},\ X_{\omega_{1}}]=[X_{2},\ X_{\omega_{2}}]=X_{0},$$
and the other brackets are trivial.\\ Here $\left\{X_{0},\ X_{1},\ X_{2},\ X_{p_{12}^{1}},\ 
X_{p_{22}^{1}},\ X_{\omega_{1}},\ X_{\omega_{2}},\ 
X_{\pi_{12}},\ X_{\pi_{22}}\right\}$ is a basis of $\mathfrak m_{0}$ and 
\begin{align*}
\mathfrak g_{-1}&=\left\{X_{\omega_{1}},\ X_{\omega_{2}},\ X_{p_{12}^{1}},\ X_{p_{22}^{1}}\right\},\ 
\mathfrak g_{-2}=\left\{X_{\pi_{12}},\ X_{\pi_{22}}\right\},\ 
\mathfrak g_{-3}=\left\{X_{1},\ X_{2}\right\},\ 
\mathfrak g_{-4}=\left\{X_{0}\right\}.
\end{align*} 
For $w\in \Sigma_{1}$, the symbol algebra $\mathfrak m_{1}(w)$ 
is isomorphic to $\mathfrak m_{1}$, where  
$\mathfrak m_{1}=\mathfrak g_{-4}\oplus\mathfrak g_{-3}\oplus\mathfrak g_{-2}\oplus\mathfrak g_{-1},$ 
whose bracket relations are given by
$$[X_{p_{11}^{6}},\ X_{\pi_{12}}]=[X_{p_{12}^{6}},\ X_{\pi_{22}}]=X_{\omega_{1}},\quad
[X_{p_{12}^{6}},\ X_{\pi_{12}}]=X_{\omega_{2}},$$ 
$$[X_{\pi_{12}},\ X_{\omega_{2}}]=X_{1},\quad   
[X_{\pi_{12}},\ X_{\omega_{1}}]=[X_{\pi_{22}},\ X_{\omega_{2}}]=X_{2},\  
[X_{1}, X_{\pi_{12}}]=X_{0},$$
and the other brackets are trivial.\\ Here  
$\left\{X_{0},\ X_{1},\ X_{2},\ X_{p_{11}^{6}},\ 
X_{p_{12}^{6}},\ X_{\omega_{1}},\ X_{\omega_{2}},\ 
X_{\pi_{12}},\ X_{\pi_{22}}\right\}$ is a basis of $\mathfrak m_{1}$ and 
\begin{align*}
\mathfrak g_{-1}=\left\{X_{\pi_{12}},\ X_{\pi_{22}},\ X_{p_{11}^{6}},\ X_{p_{12}^{6}}\right\},\ 
\mathfrak g_{-2}=\left\{X_{\omega_{1}},\ X_{\omega_{2}}\right\},\ 
\mathfrak g_{-3}=\left\{X_{1},\ X_{2}\right\},\ 
\mathfrak g_{-4}=\left\{X_{0}\right\}.
\end{align*}
For $w\in \Sigma_{2}$, the symbol algebra $\mathfrak m_{2}(w)$ 
is isomorphic to $\mathfrak m_{2}$, where  
$\mathfrak m_{2}=\mathfrak g_{-4}\oplus\mathfrak g_{-3}\oplus\mathfrak g_{-2}\oplus\mathfrak g_{-1},$ 
whose bracket relations are given by
$$[X_{p_{11}^{6}},\ X_{\pi_{12}}]=[X_{p_{12}^{6}},\ X_{\pi_{22}}]=X_{\omega_{1}},\quad
[X_{p_{12}^{6}},\ X_{\pi_{12}}]=X_{\omega_{2}},$$ 
$$[X_{\pi_{12}},\ X_{\omega_{2}}]=X_{1},\quad   
[X_{\pi_{12}},\ X_{\omega_{1}}]=[X_{\pi_{22}},\ X_{\omega_{2}}]=X_{2},$$
and the other brackets are trivial.\\ Here  
$\left\{X_{0},\ X_{1},\ X_{2},\ X_{p_{11}^{6}},\ 
X_{p_{12}^{6}},\ X_{\omega_{1}},\ X_{\omega_{2}},\ 
X_{\pi_{12}},\ X_{\pi_{22}}\right\}$ is a basis of $\mathfrak m_{2}$ and 
\begin{align*}
\mathfrak g_{-1}=\left\{X_{\pi_{12}},\ X_{\pi_{22}},\ X_{p_{11}^{6}},\ X_{p_{12}^{6}}\right\},\ 
\mathfrak g_{-2}=\left\{X_{\omega_{1}},\ X_{\omega_{2}}\right\},\ 
\mathfrak g_{-3}=\left\{X_{1},\ X_{2}\right\},\ 
\mathfrak g_{-4}=\left\{X_{0}\right\}. 
\end{align*}
\end{proposition} 
\begin{proof}
We first show that $\mathfrak m_{0}(w)\cong \mathfrak m_{0}$ for $w\in \Sigma_{0}$. 
On $U_{\omega_{1}\omega_{2}}$ in the proof of Proposition \ref{par-derived},  
we set 
$\varpi_{p_{12}^{1}}:=d{p_{12}^{1}}+f\omega_{1}
,\ \varpi_{p_{22}^{1}}:=d{p_{22}^{1}}-g\omega_{2},$
and take a coframe:\\
$\left\{\varpi_{0}, \varpi_{1}, \varpi_{2}, \varpi_{\pi_{12}}, \varpi_{\pi_{22}},  
\omega_{1}, \omega_{2}, \varpi_{p_{12}^{1}}, \varpi_{p_{22}^{1}} \right\}$, 
then the structure equations are given by 
\begin{align*}
d\varpi_{i}&\equiv 0 \hspace{0.5cm}(i=0,1,2)\hspace{2.5cm} \quad
\mod\ \varpi_{0},\ \varpi_{1},\ \varpi_{2},\ \varpi_{\pi_{12}},\ \varpi_{\pi_{22}},\\
d\varpi_{\pi_{12}}&\equiv \omega_{2}\wedge \varpi_{p_{12}^{1}} \hspace{3.9cm}  
\mod\ \varpi_{0},\ \varpi_{1},\ \varpi_{2},\ \varpi_{\pi_{12}},\ \varpi_{\pi_{22}},\\
d\varpi_{\pi_{22}}&\equiv - \varpi_{p_{12}^{1}} \wedge \omega_{1}
 -\varpi_{p_{22}^{1}} \wedge \omega_{2} \quad \quad \quad \  
\mod\ \varpi_{0},\ \varpi_{1},\ \varpi_{2},\ \varpi_{\pi_{12}},\ \varpi_{\pi_{22}}. 
\end{align*}
\begin{align*}
d\varpi_{0}&\equiv 0 \hspace{5cm} \quad \mod\ \varpi_{0},\ \varpi_{1},\ \varpi_{2},\\
d\varpi_{1}&\equiv \omega_{2}\wedge \varpi_{\pi_{12}} \hspace{4cm}  
\mod\ \varpi_{0},\ \varpi_{1},\ \varpi_{2}, \varpi_{\pi_{12}}\wedge \varpi_{\pi_{22}},\\
d\varpi_{2}&\equiv \omega_{1}\wedge \varpi_{\pi_{12}}+\omega_{2}\wedge \varpi_{\pi_{22}} \hspace{1.8cm}  
\mod\ \varpi_{0},\ \varpi_{1},\ \varpi_{2}, \varpi_{\pi_{12}}\wedge \varpi_{\pi_{22}},
\end{align*}
\begin{align*}
d\varpi_{0}\equiv \omega_{1}\wedge \varpi_{1}+\omega_{2}\wedge \varpi_{2},\ 
 \mod\ & \varpi_{0},\ \varpi_{1}\wedge \varpi_{2},\ 
\varpi_{1}\wedge \varpi_{\pi_{12}},\ \varpi_{1}\wedge \varpi_{\pi_{22}},\\
& \varpi_{2}\wedge \varpi_{\pi_{12}},\ \varpi_{2}\wedge \varpi_{\pi_{22}}
,\ \varpi_{\pi_{12}}\wedge \varpi_{\pi_{22}}.
\end{align*}
We take the dual frame:
$\left\{X_{0},\ X_{1},\ X_{2},\ X_{\pi_{12}},\ X_{\pi_{22}},\ 
X_{\omega_{1}},\ X_{\omega_{2}},\ X_{p_{12}^{1}},\ X_{p_{22}^{1}} \right\}$.  
Then, by using the same argument to the hyperbolic case, we have the 
bracket relations of $\mathfrak m_{0}$.\par
Next, we show that the isomorphism $\mathfrak m_{1}(w)\cong \mathfrak m_{1}$ for 
a point on $\Sigma_{1}$. 
On $U_{\pi_{12}\pi_{22}}$ in the proof of Proposition \ref{par-derived}, 
we set 
$\varpi_{p_{11}^{6}}:=d{p_{11}^{6}}+f\pi_{22}
,\ \varpi_{p_{12}^{6}}:=d{p_{12}^{6}}+g\pi_{22},$  
and take a coframe 
$\left\{\varpi_{0}^{\prime},\ \varpi_{1},\ \varpi_{2},\ \varpi_{\omega_{1}},\ \varpi_{\omega_{2}},\ 
\pi_{12},\ \pi_{22},\ \varpi_{p_{11}^{6}},\ \varpi_{ p_{12}^{6}} \right\}$. 
and its dual frame 
$\left\{X_{0},\ X_{1},\ X_{2},\ X_{\omega_{1}},\ X_{\omega_{2}},\ 
X_{\pi_{12}},\ X_{\pi_{22}},\ X_{p_{11}^{6}},\ X_{p_{12}^{6}} \right\}.$ 
From the proof of Proposition \ref{par-derived}, the structure equations 
at a point on $\left\{p^{6}_{11}\not=0,\ p^{6}_{12}=0\right\}$ are described by 
\begin{align*}
d\varpi_{0}^{\prime}&\equiv 0 \hspace{3.9cm} \quad
\mod\ \varpi_{0}^{\prime},\ \varpi_{1},\ \varpi_{2},\ \varpi_{\omega_{1}},\ \varpi_{\omega _{2}},\\
d\varpi_{i}&\equiv 0 \hspace{0.5cm}(i=1,2)\hspace{1.8cm} \quad
\mod\ \varpi_{0}^{\prime},\ \varpi_{1},\ \varpi_{2},\ \varpi_{\omega_{1}},\ \varpi_{\omega _{2}},\\
d\varpi_{\omega_{1}}&\equiv \pi_{12}\wedge \varpi_{p_{11}^{6}}
       +\pi_{22}\wedge \varpi_{p_{12}^{6}} \quad  
            \mod\ \varpi_{0},\ \varpi_{1},\ \varpi_{2},\ \varpi_{\omega_{1}},\ \varpi_{\omega_{2}},\\
d\varpi_{\omega_{2}}&\equiv \pi_{12}\wedge \varpi_{p_{12}^{6}} \hspace{2.7cm}
\mod\ \varpi_{0}^{\prime},\ \varpi_{1},\ \varpi_{2},\ \varpi_{\omega_{1}},\ \varpi_{\omega_{2}}.
\end{align*}
\begin{align*}
d\varpi_{0}^{\prime}&\equiv 0 \hspace{4.3cm} \quad \mod\ \varpi_{0}^{\prime},\ \varpi_{1},\ \varpi_{2},\\
d\varpi_{1}&\equiv \varpi_{\omega_{2}}\wedge \pi_{12} \hspace{3.3cm}  
\mod\ \varpi_{0},\ \varpi_{1},\ \varpi_{2}, \varpi_{\omega_{1}}\wedge \varpi_{\omega_{2}},\\
d\varpi_{2}&\equiv \varpi_{\omega_{1}}\wedge \pi_{12}+\varpi_{\omega_{2}}\wedge \pi_{22} \hspace{1.1cm}  
\mod\ \varpi_{0},\ \varpi_{1},\ \varpi_{2}, \varpi_{\omega_{1}}\wedge \varpi_{\omega_{2}},
\end{align*}
$d\varpi_{0}^{\prime}\equiv \pi_{12}\wedge \varpi_{1}\ 
\mod\ \varpi_{0}^{\prime},\ \varpi_{1}\wedge \varpi_{2},\ 
\varpi_{1}\wedge \varpi_{\omega_{1}},\ \varpi_{1}\wedge \varpi_{\omega_{2}},\ 
\varpi_{2}\wedge \varpi_{\omega_{1}},\ \varpi_{2}\wedge \varpi_{\omega_{2}},\ 
\varpi_{\omega_{1}}\wedge \varpi_{\omega_{2}}.$\\ 
Then, by using the same argument to the hyperbolic case, we have the 
bracket relations of $\mathfrak m_{1}$.\par
Finally, we prove the statement for $\mathfrak m_{2}$. 
We use the coordinate on $U_{\pi_{12}\pi_{22}}$ which is same to the case of  
$\Sigma_{1}$. 
From the proof of Proposition \ref{par-derived}, the structure equations 
at a point on $\left\{p^{6}_{11}=p^{6}_{12}=0\right\}$ are expressed as  
\begin{align*}
d\varpi_{i}&\equiv 0 \hspace{0.5cm}(i=0,1,2)\hspace{1.4cm} \quad
\mod\ \varpi_{0},\ \varpi_{1},\ \varpi_{2},\ \varpi_{\omega_{1}},\ \varpi_{\omega _{2}},\\
d\varpi_{\omega_{1}}&\equiv \pi_{12}\wedge \varpi_{p_{11}^{6}}
          +\pi_{22}\wedge \varpi_{p_{12}^{6}} \quad  
            \mod\ \varpi_{0},\ \varpi_{1},\ \varpi_{2},\ \varpi_{\omega_{1}},\ \varpi_{\omega_{2}},\\
d\varpi_{\omega_{2}}&\equiv \pi_{12}\wedge \varpi_{p_{12}^{6}} \hspace{2.7cm} 
\mod\ \varpi_{0},\ \varpi_{1},\ \varpi_{2},\ \varpi_{\omega_{1}},\ \varpi_{\omega_{2}}.
\end{align*} 
\begin{align*}
d\varpi_{0}&\equiv 0 \hspace{4.3cm} \quad \mod\ \varpi_{0},\ \varpi_{1},\ \varpi_{2},\\
d\varpi_{1}&\equiv \varpi_{\omega_{2}}\wedge {\pi_{12}} \hspace{3.2cm}  
\mod\ \varpi_{0},\ \varpi_{1},\ \varpi_{2}, \varpi_{\omega_{1}}\wedge \varpi_{\omega_{2}},\\
d\varpi_{2}&\equiv \varpi_{\omega_{1}}\wedge {\pi_{12}}+\varpi_{\omega_{2}}\wedge {\pi_{22}} \hspace{1.1cm}  
\mod\ \varpi_{0},\ \varpi_{1},\ \varpi_{2}, \varpi_{\omega_{1}}\wedge \varpi_{\omega_{2}},
\end{align*}
\begin{align*}
d\varpi_{0}\equiv 0\ 
\mod\ & \varpi_{0},\ \varpi_{1}\wedge \varpi_{2},\ 
\varpi_{1}\wedge \varpi_{\omega_{1}},\ \varpi_{1}\wedge \varpi_{\omega_{2}},\\ 
& \varpi_{2}\wedge \varpi_{\omega_{1}},\ \varpi_{2}\wedge \varpi_{\omega_{2}},\ 
\varpi_{\omega_{1}}\wedge \varpi_{\omega_{2}}.
\end{align*}
Thus we have the statement for $\mathfrak m_{2}(w)$ from the same argument. 
\end{proof}
\noindent{\bf Structures of rank 2 prolongations for elliptic equations.} 
Let $(R,D)$ be a locally elliptic equation and 
$(\Sigma(R), \hat D)$ the rank 2 prolongation.  
We use the geometric decomposition (\ref{decomposition1}) of $\Sigma(R)$ 
which is similar to the hyperbolic case. 
From Lemma \ref{ell-lem}, locally, we have 
$\Sigma_{0}|_{p^{-1}(U)}=P_{\omega_{1}\omega_{2}},\quad 
\Sigma_{2}|_{p^{-1}(U)}=P_{\pi_{11}\pi_{12}}\backslash P_{\omega_{1}\omega_{2}}$,  
where $p$ is the projection of the fibration $\Sigma(R) \to R$.
The set $\Sigma_0$ is an open set in $\Sigma(R)$, and is a $\mathbb{R}^2$--bundle over $R$.  
The set $\Sigma_2$ is a codimension 2 submanifold of $\Sigma(R)$ and is a section of $\Sigma(R) \to R$. 
\begin{proposition}\label{ell-derived}
The differential system $\hat D$ on $\Sigma(R)$ is regular, but 
is not weakly regular. More precisely, we obtain that 
$\hat D\subset \partial \hat D\subset \partial^{2}\hat D \subset 
\partial^{3}\hat D=T\Sigma(R).$ 
Moreover, we have $\partial^{2}\hat D=\partial^{(2)}\hat D$,   
$\partial^{(3)}\hat D=T\Sigma(R)\ \ {\rm on}\ \Sigma_{0}$, and 
$\partial^{(3)}\hat D=\partial^{(2)}\hat D\ \ {\rm on}\ \Sigma_{2}$. 
\end{proposition}
\begin{proof}
On each component $\Sigma_{i}$ in the decomposition,  
we calculate the structure equation of $\hat D$. 
First, we consider it on $\Sigma_{0}$.  
The canonical system $\hat D$ on $U_{\omega_{1}\omega_{2}}$ is given by 
$\hat D=\left\{\varpi_{0}=\varpi_{1}=\varpi_{2}
=\varpi_{\pi_{11}}=\varpi_{\pi_{12}}=0\right\},$  
where 
$\varpi_{\pi_{11}}:=\pi_{11}-{p_{11}^{1}}\omega_{1}-{p_{12}^{1}}\omega_{2},\quad  
\varpi_{\pi_{12}}:=\pi_{12}-{p_{12}^{1}}\omega_{1}+{p_{11}^{1}}\omega_{2}.$ 
The structure equation of $\hat D$ on $\Sigma_0$ is given by 
\begin{align*}
d\varpi_{i}&\equiv 0 \hspace{0.5cm}(i=0,1,2)\hspace{0.5cm} \quad
\bmod\ \varpi_{0},\ \varpi_{1},\ \varpi_{2},\ \varpi_{\pi_{11}},\ \varpi_{\pi_{12}},\\
d\varpi_{\pi_{11}}&\equiv \omega_{1}\wedge ({dp_{11}^{1}}+f\omega_{2})+
\omega_{2}\wedge ({dp_{12}^{1}}+g\omega_{2})\quad    
\bmod\ \varpi_{0},\ \varpi_{1},\ \varpi_{2},\ \varpi_{\pi_{11}},\ \varpi_{\pi_{12}},\\
d\varpi_{\pi_{12}}&\equiv \omega_{1}\wedge ({dp_{12}^{1}}+g\omega_{2})  
                    -\omega_{2}\wedge ({dp_{11}^{1}}+f\omega_{2}) \quad 
\bmod\ \varpi_{0},\ \varpi_{1},\ \varpi_{2},\ \varpi_{\pi_{11}},\ \varpi_{\pi_{12}}.
\end{align*}
where $f$ and $g$ are appropriate functions. Hence we have 
$\partial \hat D=\left\{\varpi_{0}=\varpi_{1}
=\varpi_{2}=0\right\}={p_{*}^{-1}}(D).$ 
The structure equation of $\partial \hat D$ is written as 
\begin{align*}
d\varpi_{0}&\equiv 0 \hspace{5.4cm} \quad \bmod\ \varpi_{0},\ \varpi_{1},\ \varpi_{2},\\
d\varpi_{1}&\equiv \omega_{1}\wedge \varpi_{\pi_{11}}+\omega_{2}\wedge \varpi_{\pi_{12}} \hspace{2cm}  
\bmod\ \varpi_{0},\ \varpi_{1},\ \varpi_{2}, \varpi_{\pi_{11}}\wedge \varpi_{\pi_{12}},\\
d\varpi_{2}&\equiv \omega_{1}\wedge \varpi_{\pi_{12}}-\omega_{2}\wedge \varpi_{\pi_{11}} \hspace{2cm}  
\bmod\ \varpi_{0},\ \varpi_{1},\ \varpi_{2}, \varpi_{\pi_{11}}\wedge \varpi_{\pi_{12}},
\end{align*}
Hence we have $\partial ^{2}\hat D=\partial ^{(2)}\hat D=\left\{\varpi_{0}=0\right\}.$  
The structure equation of $\partial^{2} \hat D$ is expressed as 
\begin{align*}
d\varpi_{0}\equiv \omega_{1}\wedge \varpi_{1}+\omega_{2}\wedge \varpi_{2},\ 
 \bmod\ & \varpi_{0},\ \varpi_{1}\wedge \varpi_{2},\  
\varpi_{1}\wedge \varpi_{\pi_{11}},\ \varpi_{1}\wedge \varpi_{\pi_{12}},\\  
& \varpi_{2}\wedge \varpi_{\pi_{11}},\ \varpi_{2}\wedge \varpi_{\pi_{12}},\ 
\varpi_{\pi_{11}}\wedge \varpi_{\pi_{12}}. 
\end{align*}
Hence, we have $\partial ^{(3)}\hat D=T\Sigma(R).$ 
Next we consider on $\Sigma_2$. 
The canonical system $\hat D$ on $U_{\pi_{11}\pi_{12}}$ is given by 
$\hat D=\left\{\varpi_{0}=\varpi_{1}=\varpi_{2}=
\varpi_{\omega_{1}}=\varpi_{\omega_{2}}=0\right\},$ 
where 
$\varpi_{\omega_{1}}:=\omega_{1}-{p_{11}^{6}}\pi_{11}-{p_{12}^{6}}\pi_{12},\quad  
\varpi_{\omega_{2}}:=\omega_{2}-{p_{12}^{6}}\pi_{11}+{p_{11}^{6}}\pi_{12}.$ 
For $w\in U_{\pi_{11}\pi_{12}}$, 
$w\in \Sigma_{2}$ if and only if $p^{6}_{11}(w)=p^{6}_{12}(w)=0$. 
Therefore, we calculate the structure equation at $w$ in the  codimension 2 
submanifold $\left\{p^{6}_{11}=p^{6}_{12}=0\right\}\subset \Sigma(R)$. 
The structure equation at a point on $\left\{p^{6}_{11}=p^{6}_{12}=0\right\}$ is 
described by 
\begin{align*}
d\varpi_{i}&\equiv 0 \hspace{0.5cm}(i=0,1,2)\hspace{0.5cm} \quad
\mod\ \varpi_{0},\ \varpi_{1},\ \varpi_{2},\ \varpi_{\omega_{1}},\ \varpi_{\omega _{2}},\\
d\varpi_{\omega_{1}}&\equiv \pi_{11}\wedge ({dp_{11}^{6}}+f\pi_{12})+\pi_{12}\wedge 
          ({dp_{12}^{6}}+g\pi_{12}) \quad  
            \mod\ \varpi_{0},\ \varpi_{1},\ \varpi_{2},\ \varpi_{\omega_{1}},\ \varpi_{\omega_{2}},\\
d\varpi_{\omega_{2}}&\equiv  \pi_{11}\wedge ({dp_{12}^{6}}+g\pi_{12})-\pi_{12}\wedge 
          ({dp_{11}^{6}}+f\pi_{12}) \quad 
\mod\ \varpi_{0},\ \varpi_{1},\ \varpi_{2},\ \varpi_{\omega_{1}},\ \varpi_{\omega_{2}}.
\end{align*}
where $f$ and $g$ are appropriate functions. Hence we have 
$\partial \hat D=\left\{\varpi_{0}=\varpi_{1}=
\varpi_{2}=0\right\}={p_{*}^{-1}}(D).$ 
The structure equation of $\partial \hat D$ is 
\begin{align*}
d\varpi_{0}&\equiv 0 \hspace{5.9cm} \quad \mod\ \varpi_{0},\ \varpi_{1},\ \varpi_{2},\\
d\varpi_{1}&\equiv \varpi_{\omega_{1}}\wedge \pi_{11}+\varpi_{\omega_{2}}\wedge \pi_{12} \hspace{2.6cm}  
\mod\ \varpi_{0},\ \varpi_{1},\ \varpi_{2}, \varpi_{\omega_{1}}\wedge \varpi_{\omega_{2}},\\
d\varpi_{2}&\equiv \varpi_{\omega_{1}}\wedge \pi_{12}-\varpi_{\omega_{2}}\wedge \pi_{11} \hspace{2.6cm}  
\mod\ \varpi_{0},\ \varpi_{1},\ \varpi_{2}, \varpi_{\omega_{1}}\wedge \varpi_{\omega_{2}},
\end{align*}
Hence we have $\partial ^{2}\hat D=\partial ^{(2)}\hat D=\left\{\varpi_{0}=0\right\}.$  
The structure equation of $\partial^{2} \hat D$ is given by 
\begin{align*}
d\varpi_{0}\equiv 0\ 
\mod\ & \varpi_{0},\ \varpi_{1}\wedge \varpi_{2},\ 
\varpi_{1}\wedge \varpi_{\omega_{1}},\ \varpi_{1}\wedge \varpi_{\omega_{2}},\\ 
& \varpi_{2}\wedge \varpi_{\omega_{1}},\ \varpi_{2}\wedge \varpi_{\omega_{2}},\ 
\varpi_{\omega_{1}}\wedge \varpi_{\omega_{2}}.
\end{align*}
Thus, we have $\partial ^{(3)}\hat D=\partial ^{(2)}\hat D.$ 
\end{proof} 
From the above proposition, $(\Sigma(R), \hat D)$ is locally 
weakly regular around $w \in \Sigma_{0}.$ So 
we can define the symbol algebra at $w$ in the sense of Tanaka. On the other hand, for a point 
$w$ on $\Sigma_{2}$, $(\Sigma(R), \hat D)$ is not weakly regular around $w$. 
However, by taking the filtration on $\Sigma(R)$ which is same to the 
hyperbolic case, we can define the symbol algebra at $w$. 
Each structure of symbol algebras is given in the following.  
\begin{proposition}\label{ell-symbol}
For $w\in \Sigma_{0}$, the symbol algebra $\mathfrak m_{0}(w)$ 
is isomorphic to $\mathfrak m_{0}$, where  
$\mathfrak m_{0}=\mathfrak g_{-4}\oplus\mathfrak g_{-3}\oplus\mathfrak g_{-2}\oplus\mathfrak g_{-1},$  
whose bracket relations are given by 
$$[X_{p_{11}^{1}},\ X_{\omega_{1}}]=[X_{p_{12}^{1}},\ X_{\omega_{2}}]=X_{\pi_{11}},\quad
[X_{p_{12}^{1}},\ X_{\omega_{1}}]=[X_{\omega_{2}},\ X_{p_{11}^{1}}]=X_{\pi_{12}},$$
$$[X_{\pi_{11}},\ X_{\omega_{1}}]=[X_{\pi_{12}},\ X_{\omega_{2}}]=X_{1},\quad
[X_{\pi_{12}},\ X_{\omega_{1}}]=[X_{\omega_{2}},\ X_{\pi_{11}}]=X_{2},$$
\quad \quad \quad \ $[X_{1},\ X_{\omega_{1}}]=[X_{2},\ X_{\omega_{2}}]=X_{0}$,\  
and the other brackets are trivial.\\ 
Here $\left\{X_{0},\ X_{1},\ X_{2},\ X_{p_{11}^{1}},\ 
X_{p_{12}^{1}},\ X_{\omega_{1}},\ X_{\omega_{2}},\ 
X_{\pi_{11}},\ X_{\pi_{12}}\right\}$ is a basis of $\mathfrak m_{0}$ and 
\begin{align*}
\mathfrak g_{-1}=\left\{X_{\omega_{1}},\ X_{\omega_{2}},\ X_{p_{11}^{1}},\ X_{p_{12}^{1}}\right\},\ 
\mathfrak g_{-2}=\left\{X_{\pi_{11}},\ X_{\pi_{12}}\right\},\ 
\mathfrak g_{-3}=\left\{X_{1},\ X_{2}\right\},\ 
\mathfrak g_{-4}=\left\{X_{0}\right\}.
\end{align*} 
For $w\in \Sigma_{2}$, the symbol algebra $\mathfrak m_{2}(w)$ 
is isomorphic to $\mathfrak m_{2}$, where  
$\mathfrak m_{2}=\mathfrak g_{-4}\oplus\mathfrak g_{-3}\oplus\mathfrak g_{-2}\oplus\mathfrak g_{-1},$ 
whose bracket relations are given by
$$[X_{p_{11}^{6}},\ X_{\pi_{11}}]=[X_{p_{12}^{6}},\ X_{\pi_{12}}]=X_{\omega_{1}},\quad
[X_{p_{12}^{6}},\ X_{\pi_{11}}]=[X_{\pi_{12}},\ X_{p_{11}^{6}}]=X_{\omega_{2}},$$
$$[X_{\pi_{11}},\ X_{\omega_{1}}]=[X_{\pi_{12}},\ X_{\omega_{2}}]=X_{1},\quad
[X_{\pi_{12}},\ X_{\omega_{1}}]=[X_{\omega_{2}},\ X_{\pi_{11}}]=X_{2},$$
and the other brackets are trivial.\\ Here  
$\left\{X_{0},\ X_{1},\ X_{2},\ X_{p_{11}^{6}},\ 
X_{p_{12}^{6}},\ X_{\omega_{1}},\ X_{\omega_{2}},\ 
X_{\pi_{11}},\ X_{\pi_{12}}\right\}$ is a basis of $\mathfrak m_{2}$ and 
\begin{align*}
\mathfrak g_{-1}=\left\{X_{\pi_{11}},\ X_{\pi_{12}},\ X_{p_{11}^{6}},\ X_{p_{12}^{6}}\right\},\ 
\mathfrak g_{-2}=\left\{X_{\omega_{1}},\ X_{\omega_{2}}\right\},\ 
\mathfrak g_{-3}=\left\{X_{1},\ X_{2}\right\},\ 
\mathfrak g_{-4}=\left\{X_{0}\right\}.
\end{align*}
\end{proposition} 
\begin{proof}
We first show that $\mathfrak m_{0}(m)\cong \mathfrak m_{0}$. 
On $U_{\omega_{1}\omega_{2}}$ in the proof of Proposition \ref{ell-derived}, 
if we set $\varpi_{p_{11}^{1}}:=d{p_{11}^{1}}+f\omega_{2}
,\ \varpi_{p_{12}^{1}}:=d{p_{12}^{1}}+g\omega_{2}$ and 
take a coframe:   
\[\left\{\varpi_{0},\ \varpi_{1},\ \varpi_{2},\ \varpi_{\pi_{11}},\ \varpi_{\pi_{12}},\ 
\omega_{1},\ \omega_{2},\ \varpi_{p_{11}^{1}},\ \varpi_{p_{12}^{1}} \right\},\]
then the structure equations are written as 
\begin{align*}
d\varpi_{i}&\equiv 0 \hspace{0.5cm}(i=0,1,2)\hspace{1.15cm} \quad
\mod\ \varpi_{0},\ \varpi_{1},\ \varpi_{2},\ \varpi_{\pi_{11}},\ \varpi_{\pi_{12}},\\
d\varpi_{\pi_{11}}&\equiv \omega_{1}\wedge \varpi_{p_{11}^{1}}+
\omega_{2}\wedge \varpi_{p_{12}^{1}}\quad    
\mod\ \varpi_{0},\ \varpi_{1},\ \varpi_{2},\ \varpi_{\pi_{11}},\ \varpi_{\pi_{12}},\\
d\varpi_{\pi_{12}}&\equiv \omega_{1}\wedge \varpi_{p_{12}^{1}}  
                    -\omega_{2}\wedge \varpi_{p_{11}^{1}} \quad 
\mod\ \varpi_{0},\ \varpi_{1},\ \varpi_{2},\ \varpi_{\pi_{11}},\ \varpi_{\pi_{12}}.
\end{align*}
\begin{align*}
d\varpi_{0}&\equiv 0 \hspace{5.4cm} \quad \mod\ \varpi_{0},\ \varpi_{1},\ \varpi_{2},\\
d\varpi_{1}&\equiv \omega_{1}\wedge \varpi_{\pi_{11}}+\omega_{2}\wedge \varpi_{\pi_{12}} \hspace{2cm}  
\mod\ \varpi_{0},\ \varpi_{1},\ \varpi_{2}, \varpi_{\pi_{11}}\wedge \varpi_{\pi_{12}},\\
d\varpi_{2}&\equiv \omega_{1}\wedge \varpi_{\pi_{12}}-\omega_{2}\wedge \varpi_{\pi_{11}} \hspace{2cm}  
\mod\ \varpi_{0},\ \varpi_{1},\ \varpi_{2}, \varpi_{\pi_{11}}\wedge \varpi_{\pi_{12}},
\end{align*}
\begin{align*}
d\varpi_{0}\equiv \omega_{1}\wedge \varpi_{1}+\omega_{2}\wedge \varpi_{2},\ 
 \mod\ & \varpi_{0},\ \varpi_{1}\wedge \varpi_{2},\  
\varpi_{1}\wedge \varpi_{\pi_{11}},\ \varpi_{1}\wedge \varpi_{\pi_{12}},\\  
& \varpi_{2}\wedge \varpi_{\pi_{11}},\ \varpi_{2}\wedge \varpi_{\pi_{12}},\ 
\varpi_{\pi_{11}}\wedge \varpi_{\pi_{12}}. 
\end{align*}
We take the dual frame 
$\left\{X_{0},\ X_{1},\ X_{2},\ X_{\pi_{11}},\ X_{\pi_{12}},\ 
X_{\omega_{1}},\ X_{\omega_{2}},\ X_{p_{11}^{1}},\ X_{p_{12}^{1}} \right\}$. 
Then, by the same argument to the hyperbolic case, we have the 
bracket relations of $\mathfrak m_{0}$.\par
Next, we prove the statement for the algebra $\mathfrak m_{2}$. 
On $U_{\pi_{11}\pi_{12}}$ in the proof of Proposition \ref{ell-derived}, 
we set $\varpi_{p_{11}^{6}}:=d{p_{11}^{6}}+f\pi_{12}
,\ \varpi_{p_{12}^{6}}:=d{p_{12}^{6}}+g\pi_{12}$ and 
take a coframe:\\   
$\left\{\varpi_{0},\ \varpi_{1},\ \varpi_{2},\ \varpi_{\omega_{1}},\ \varpi_{\omega_{2}},\ 
\pi_{11},\ \pi_{12},\ \varpi_{p_{11}^{6}},\ \varpi_{p_{12}^{6}} \right\}$, 
then the structure equations at a point on $\Sigma_2$ are given by 
\begin{align*}
d\varpi_{i}&\equiv 0 \hspace{0.5cm}(i=0,1,2)\hspace{0.5cm} \quad
\mod\ \varpi_{0},\ \varpi_{1},\ \varpi_{2},\ \varpi_{\omega_{1}},\ \varpi_{\omega _{2}},\\
d\varpi_{\omega_{1}}&\equiv \pi_{11}\wedge ({dp_{11}^{6}}+f\pi_{12})+\pi_{12}\wedge 
          ({dp_{12}^{6}}+g\pi_{12}) \quad  
            \mod\ \varpi_{0},\ \varpi_{1},\ \varpi_{2},\ \varpi_{\omega_{1}},\ \varpi_{\omega_{2}},\\
d\varpi_{\omega_{2}}&\equiv  \pi_{11}\wedge ({dp_{12}^{6}}+g\pi_{12})-\pi_{12}\wedge 
          ({dp_{11}^{6}}+f\pi_{12}) \quad 
\mod\ \varpi_{0},\ \varpi_{1},\ \varpi_{2},\ \varpi_{\omega_{1}},\ \varpi_{\omega_{2}}.
\end{align*}
\begin{align*}
d\varpi_{0}&\equiv 0 \hspace{5.9cm} \quad \mod\ \varpi_{0},\ \varpi_{1},\ \varpi_{2},\\
d\varpi_{1}&\equiv \varpi_{\omega_{1}}\wedge \pi_{11}+\varpi_{\omega_{2}}\wedge \pi_{12} \hspace{2.6cm}  
\mod\ \varpi_{0},\ \varpi_{1},\ \varpi_{2}, \varpi_{\omega_{1}}\wedge \varpi_{\omega_{2}},\\
d\varpi_{2}&\equiv \varpi_{\omega_{1}}\wedge \pi_{12}-\varpi_{\omega_{2}}\wedge \pi_{11} \hspace{2.6cm}  
\mod\ \varpi_{0},\ \varpi_{1},\ \varpi_{2}, \varpi_{\omega_{1}}\wedge \varpi_{\omega_{2}},
\end{align*}
\begin{align*}
d\varpi_{0}\equiv 0\ 
\mod\ & \varpi_{0},\ \varpi_{1}\wedge \varpi_{2},\ 
\varpi_{1}\wedge \varpi_{\omega_{1}},\ \varpi_{1}\wedge \varpi_{\omega_{2}},\\ 
& \varpi_{2}\wedge \varpi_{\omega_{1}},\ \varpi_{2}\wedge \varpi_{\omega_{2}},\ 
\varpi_{\omega_{1}}\wedge \varpi_{\omega_{2}}.
\end{align*}
Let 
$\left\{X_{0},\ X_{1},\ X_{2},\ X_{\omega_{1}},\ X_{\omega_{2}},\ 
X_{\pi_{11}},\ X_{\pi_{12}},\ X_{p_{11}^{6}},\ X_{p_{12}^{6}} \right\}$  
be the dual frame. Then, by using the same argument to the hyperbolic case, we have the 
bracket relations of $\mathfrak m_{2}$. 
\end{proof}
\section{Construction of singular solutions and the theory of submanifold of the rank 2 prolongation of the Second Jet space}
In section 2 and 3, we studied various properties of the rank 2 prolongations  
$(\Sigma(R), \hat D)$ of single equations $(R,D)$. 
Under these prolongations, we mention the strategy of the construction 
of the geometric singular solutions for each class of equations $(R,D)$. 
Moreover, we construct singular solutions for model equations 
belonging to each class. 
For this purpose, we first consider the rank 2 prolongation $\Sigma(J^2)$ of 
the second jet space $J^2(\mathbb R^2,\mathbb R)$. 
For the 2-jet space $J^2(\mathbb R^2,\mathbb R)$, we denote
the rank 2 prolongation of $J^2(\mathbb R^2,\mathbb R)$ by 
$(\Sigma(J^2), \hat{C}^2)$.
This space $\Sigma(J^2)$ is a submanifold of the Grassmann 
bundle $J(C^2, 2)$.
The geometry of $(\Sigma(J^2), \hat C^2)$ in $J(C^2, 2)$ is studied in \cite{S}. From now on, we refer the reader to \cite{S} for 
the obtained results.    
For an open set $V \subset J^2(\mathbb R^2,\mathbb R)$, 
${\Pi^{2}_{1}}^{-1}(V)$ is covered by $6$ open sets:  
$${\Pi^{2}_{1}}^{-1}(V)=V_{xy}\cup V_{xt}\cup V_{yr}\cup V_{rs}\cup 
V_{rt}\cup V_{st},$$
where $\Pi^{2}_{1}:\Sigma(J^2) \to J^2$ is the projection and each open set is 
given by 
\begin{align*}
V_{xy}:=&\left\{w \in  {\Pi^{2}_{1}}^{-1}(V)\ |\ dx\wedge dy|_{w}\not=0\right\},\  
V_{xt}:=\left\{w \in  {\Pi^{2}_{1}}^{-1}(V)\ |\ dx\wedge dt|_{w}\not=0\right\},\\ 
V_{yr}:=&\left\{w \in  {\Pi^{2}_{1}}^{-1}(V)\ |\ dy\wedge dr|_{w}\not=0\right\},\ 
V_{rs}:=\left\{w \in  {\Pi^{2}_{1}}^{-1}(V)\ |\ dr\wedge ds|_{w}\not=0\right\},\\
V_{rt}:=&\left\{w \in  {\Pi^{2}_{1}}^{-1}(V)\ |\ dr\wedge dt|_{w}\not=0\right\},\  
V_{st}:=\left\{w \in  {\Pi^{2}_{1}}^{-1}(V)\ |\ ds\wedge dt|_{w}\not=0\right\}.  
\end{align*}
The prolongation $\Sigma(J^2)$ has the similar geometric decomposition: 
$\Sigma(J^2)=\Sigma_{0}\cup \Sigma_{1}\cup \Sigma_{2},$ 
where $\Sigma_{i}=\left\{w\in \Sigma(J^2)\ |\ {\rm dim}(w\cap {\rm fiber})=i\right\}$ $(i=0,1,2)$, 
and ``fiber" means that the fiber of $T(J^2)\supset C^2\to T(J^1)$. 
Then, locally, 
\begin{align*}
\Sigma_{0}|_{{\Pi^{2}_{1}}^{-1}(V)}&=V_{xy}|_{{\Pi^{2}_{1}}^{-1}(V)},\quad 
\Sigma_{1}|_{{\Pi^{2}_{1}}^{-1}(V)}=\left\{(V_{xt}\cup V_{yr})\backslash V_{xy}\right\}|_{{\Pi^{2}_{1}}^{-1}(V)},\\
\Sigma_{2}|_{{\Pi^{2}_{1}}^{-1}(V)}&=\left\{(V_{rs}\cup V_{rt}\cup V_{st})\backslash (V_{xy}\cup V_{xt}\cup V_{yr})\right\}|_{{\Pi^{2}_{1}}^{-1}(V)},
\end{align*}
The set $\Sigma_{0}=J^3$ is an open set in $\Sigma(J^2)$ and is a 
$\mathbb R^4$-bundle over $J^2$.
The set $\Sigma_1$ is a codimension 1 submanifold in $\Sigma(J^2)$. 
The set $\Sigma_2$ is a  codimension 2 submanifold in $\Sigma(J^2)$ and is a 
$\mathbb P^2$-bundle over $J^2$. 
In the following, we give the description of the canonical system 
$(\Sigma(J^2),\hat C^2)$ on each coordinate. 
\begin{enumerate} 
\item[(A)] $V_{xy}\cong J^3,\ (x,y,z,p,q,r,s,t,p_{111}, p_{112},p_{122},p_{222})$:\\
$\hat C^2=\left\{\varpi_{0}=\varpi_{1}=\varpi_{2}
=\varpi_{r}=\varpi_{s}=\varpi_{t}=0\right\},$   
where 
$\varpi_{r}=dr-p_{111}dx-p_{112}dy,\  \varpi_{s}=ds-p_{112}dx-p_{122}dy,\ 
\varpi_{t}=dt-p_{122}dx-p_{222}dy$. 
\item[(B)] $V_{xt},\ (x,y,z,p,q,r,s,t,a, B,c,e)$:\\
$\hat C^2=\left\{\varpi_{0}=\varpi_{1}=\varpi_{2}=
\varpi_{y}=\varpi_{r}=\varpi_{s}=0\right\},$ 
where 
$\varpi_{y}=dy-adx-Bdt,\  \varpi_{r}=dr-cdx-(a^2+eB)dt,\ 
\varpi_{s}=ds-edx-adt$. 
\item[(C)] $V_{yr},\ (x,y,z,p,q,r,s,t,a, B,c,e)$:\\
$\hat C^2=\left\{\varpi_{0}=\varpi_{1}=\varpi_{2}=
\varpi_{x}=\varpi_{s}=\varpi_{t}=0\right\},$ 
where 
$\varpi_{x}=dx-ady-Bdr,\  \varpi_{s}=ds-cdy+adr,\ 
\varpi_{t}=dt-edy-(a^2+Bc)dr$. 
\item[(D)] $V_{rs},\ (x,y,z,p,q,r,s,t,B, D,E,F)$:\\
$\hat C^2=\left\{\varpi_{0}=\varpi_{1}=\varpi_{2}=
\varpi_{x}=\varpi_{y}=\varpi_{t}=0\right\},$  
where 
$\varpi_{x}=dx-(DE-BF)dr-Bds,\  \varpi_{y}=dy-Bdr-Dds,\ 
\varpi_{t}=dt-Edr-Fds$. 
\item[(E)] $V_{rt},\ (x,y,z,p,q,r,s,t,A, D,E,F)$:\\
$\hat C^2=\left\{\varpi_{0}=\varpi_{1}=\varpi_{2}=
\varpi_{x}=\varpi_{y}=\varpi_{s}=0\right\},$ 
where 
$\varpi_{x}=dx-Adr+(DE-CF)dt,\  \varpi_{y}=dy+(AF-(DE-CF)E)dr-Ddt$,\\  
$\varpi_{s}=ds-Edr-Fdt$. 
\item[(F)] $V_{st},\ (x,y,z,p,q,r,s,t,A, B,E,F)$:\\
$\hat C^2=\left\{\varpi_{0}=\varpi_{1}=\varpi_{2}=
\varpi_{x}=\varpi_{y}=\varpi_{r}=0\right\},$  
where 
$\varpi_{x}=dx-Ads-Bdt,\  \varpi_{y}=dy-Bds+(BE-AF)dt,\ 
\varpi_{r}=dr-Eds-Fdt$. 
\end{enumerate}
The reason we introduced $\Sigma(J^2)$ is that 
$\Sigma(R)$ is regarded as the subset in $\Sigma(J^2)$. 
More precisely, we need to construct the equivariant 
embedding $\iota:\Sigma(R) \hookrightarrow \Sigma(J^2)$ which give the following commutative diagram:
\begin{align}\label{equivariant1}
&\Sigma(R) \hookrightarrow \Sigma(J^2) \nonumber \\
& \quad  \downarrow \hspace{1.4cm} \downarrow \\ 
& \quad R \hookrightarrow J^2(\mathbb R^2, \mathbb R). \nonumber
\end{align}
Here, the correspondences except for $\iota$ are already given.  
This diagram is an extension of the following commutative diagram. 
\begin{align}\label{equivariant2}
&R^{(1)} \hookrightarrow J^3(\mathbb R^2,\mathbb R) \nonumber \\
& \quad  \downarrow \hspace{1.4cm} \downarrow \\ 
& \quad R \hookrightarrow J^2(\mathbb R^2, \mathbb R). \nonumber
\end{align} 
where $R^{(1)}$ is the prolongation of $(R,D)$ with independence condition. 
In general, for given second order PDE $R=\left\{F=0\right\}$ with independent 
variables $x,y$, this prolongation $R^{(1)}$  corresponds to a 
third order PDE system which is obtained by partial derivation of $F=0$ 
for the two variables $x,y$. 
Hence, $R^{(1)}$ can be regarded naturally as a submanifold in $J^3$ which is 
also the prolongation of $J^2$ with the independence condition.\par  
Let us return to the diagram (\ref{equivariant1}). 
If we can construct the equivariant embedding $\iota:\Sigma(R) \hookrightarrow \Sigma(J^2)$, 
then we can obtain singular solutions $L$ 
by the following strategy: 
\begin{center}
Find an integral manifold $L$ of $(\Sigma(R), \hat D)\subset (\Sigma(J^2), \hat C^2)$ 
passing through the $\Sigma_{1}\cup \Sigma_{2}$. 
\end{center}
Indeed, in the rest of this section, we construct 
the embedding and singular solutions for model 
equations belonging to the each class. \par
\medskip\noindent{\bf Singular solutions of a hyperbolic equation.} 
We consider the wave equation $R=\left\{s=0 \right\}$ as a model equation. 
The differential system $D=\left\{\ \varpi_{0}=\varpi_{1}=\varpi_{2}=0 \right\}$ 
is given by  
$\varpi_{0}=dz-pdx-qdy,\  
\varpi_{1}=dp-rdx,\ 
\varpi_{2}=dq-tdy$.  
The structure equation of $D$ is written as   
\begin{align*}
d\varpi_{0}= -dp\wedge dx-dq\wedge dy,\quad   
d\varpi_{1}= -dr\wedge dx,\quad   
d\varpi_{2}= -dt\wedge dy. 
\end{align*} 
For an open set $U$ in $R$, 
we have the covering $p^{-1}(U)=P_{xy}\cup P_{xt}\cup P_{yr}\cup P_{rt}$ 
of the fibration $p:\Sigma(R)\to R$ 
followed by Theorem \ref{hyp-pro}, 
where 
\begin{align*}
U_{xy}:&=\left\{v\ \in \pi^{-1}(U)\ |\ dx|_{v}\wedge dy|_{v}\not= 0\right\},\  
U_{xt}:=\left\{v\ \in \pi^{-1}(U)\ |\ dx|_{v}\wedge dt|_{v}\not= 0\right\},\\  
U_{yr}:&=\left\{v\ \in \pi^{-1}(U)\ |\ dy|_{v}\wedge dr|_{v}\not= 0\right\},\ 
U_{rt}:=\left\{v\ \in \pi^{-1}(U)\ |\ dr|_{v}\wedge dt|_{v}\not= 0\right\},\\ 
P_{xy}:&=p^{-1}(U) \cap U_{xy},\ 
P_{xt}:=p^{-1}(U) \cap U_{xt},\ 
P_{yr}:=p^{-1}(U) \cap U_{yr},\  
P_{rt}:=p^{-1}(U) \cap U_{rt}. 
\end{align*}
The geometric decomposition 
$\Sigma(R)=\Sigma_{0}\cup\Sigma_{1}\cup\Sigma_{2}$ is given by 
$\Sigma_{0}|_{p^{-1}(U)}=P_{xy},\ 
\Sigma_{1}|_{p^{-1}(U)}=(P_{xt}\cup P_{yr}) 
\backslash P_{xy},\ 
\Sigma_{2}|_{p^{-1}(U)}=P_{rt}\backslash 
(P_{xy}\cup P_{xt}\cup P_{yr})$. 
Now, by using this decomposition, we consider embeddings 
from $\Sigma(R)$ into $\Sigma(J^2)$.\\ 
$({\rm i})$ On the open set $V_{xy}=J^3\subset \Sigma(J^2)$.\par  
On $V_{xy}$, we consider the submanifold 
$\overline \Sigma_{xy}=\left\{s=p_{112}=p_{122}=0\right\}.$ 
On $\overline \Sigma_{xy}$, we have the induced differential system   
$D_{\overline \Sigma_{xy}}=\left\{\varpi_{0}=\varpi_{1}=
\varpi_{2}=\varpi_{r}=\varpi_{t}=0\right\},$ 
where  
$\varpi_{r}=dr-p_{111}dx,\ \varpi_{t}=dt-p_{222}dx.$  
Clearly, this system $(\overline \Sigma_{xy}, D_{\overline \Sigma_{xy}})$ is isomorphic to 
$(P_{xy}, \hat D)\subset (\Sigma(R), \hat D)$. Indeed, this system is equal to 
the third order PDE which is obtained by 
partial derivation of the original equation $s=0$ for the independent 
variables $x,y$. The projection to $R$ of these integral manifolds are regular solutions of the wave equation $s=0$. \\ 
$({\rm ii})$ On the open set $V_{xt}\subset \Sigma(J^2)$.\par  
We consider singular solutions of corank 1 which are the projections of 
integral manifolds of $\Sigma(J^2)$ passing through $\Sigma_1$.
On $V_{xt}$, we consider the submanifold 
$\overline \Sigma_{xt}=\left\{s=a=e=0\right\}.$ 
On $\overline \Sigma_{xt}$, we have the differential system 
$D_{\overline \Sigma_{xt}}=\left\{\varpi_{0}=\varpi_{1}=\varpi_{2}
=\varpi_{y}=\varpi_{r}=0\right\},$ 
where $\varpi_{y}=dy-Bdt,\ \varpi_{r}=dr-cdx.$   
Note that $w\in \Sigma_1 \iff B(w)=0$.
Clearly, this system $(\overline \Sigma_{xt}, D_{\overline \Sigma_{xt}})$ is isomorphic to 
$(P_{xt}, \hat D)\subset (\Sigma(R), \hat D)$. We construct integral 
manifolds of this system in the following.  
Let 
$\iota:S \hookrightarrow \overline \Sigma_{xt}\subset \Sigma(J^2)$
be a graph defined by 
$$(x,\ y(x,t),\ z(x,t),\ p(x,t),\ q(x,t),\ r(x,t),\ t,\ B(x,t),\ c(x,t)) 
\ \textrm{around}\ (x_0,t_0).$$
If $S$ is an integral submanifold of $(\overline \Sigma_{xt}, 
D_{\overline \Sigma_{xt}})$, 
then the following conditions are 
satisfied: 
\begin{align}
&\iota^{*}\varpi_{0}=
\iota^{*}(dz-pdx-qdy)=(z_x-p-qy_x)dx+(z_t-qy_t)dt=0,\label{wave1} \\
&\iota^{*}\varpi_{1}=
\iota^{*}(dp-rdx)=(p_x-r)dx+p_{t}dt=0,\label{wave2}\\
&\iota^{*}\varpi_{2}=
\iota^{*}(dq-tdy)=(q_x-ty_x)dx+(q_t-ty_t)dt=0,\label{wave3}\\
&\iota^{*}\varpi_{y}=
\iota^{*}(dy-Bdt)= y_{x}dx+(y_t-B)dt=0,\label{wave4}\\
&\iota^{*}\varpi_{r}=
\iota^{*}(dr-cdx)=(r_{x}-c)dx+r_{t}dt=0.\label{wave5}
\end{align}
We have $y(x,t)=y(t),\ B(x,t)=y^\prime(t)$ from $(\ref{wave4})$, and note that 
the condition passing through $\Sigma_{1}$ is $B(t_0)=y^{\prime}(t_0)=0$. 
From $(\ref{wave3})$, we have $q=\int ty_tdt=ty-Y$ where $Y:=\int ydt$.    
From $(\ref{wave1})$, $z=\int (ty-Y)y_tdt+z_{0}(x)
=\frac{ty^2}{2}+\frac{1}{2}\int y^2dt -Yy+z_{0}(x)$ where 
$z_0(x)$ is a function on $S$ depending only $x$, and 
$p=z_{x}=z_{0}^{\prime}(x)$. For $(\ref{wave2})$, the function $p$ satisfies 
$p_t=0$ and we have $r=z_{0}^{\prime\prime}(x)$. 
For $(\ref{wave5})$, the function $r$ satisfies 
$r_t=0$ and we have $c=z_{0}^{\prime\prime\prime}(x)$.  
Therefore, we obtain the solution of $s=0$ around $(x_0,t_0)$ given by 
\begin{eqnarray*}\label{wavesingular}
(x,\ y(x,t),\ z(x,t),\ p(x,t),\ q(x,t),\ r(x,t),\ t,\ B(x,t),\ c(x,t)) 
\hspace{3cm} \nonumber \\
 =(x,\ y(t),\ \frac{ty^2}{2}+\frac{1}{2}\int y^2dt -y\int ydt+z_{0}(x),\ z_{0}^{\prime}(x),\ ty-\int ydt,\ z_{0}^{\prime\prime}(x),\ t,y^\prime ,\ z_{0}^{\prime\prime\prime}(x)). \nonumber 
\end{eqnarray*}
for arbitrary functions $y(t)$ and $z_0(x)$.
These integral surfaces with the condition $y^{\prime}(t_0)=0$ are geometric singular solutions of corank 1. \\ 
$({\rm iii})$ On the open set $V_{yr}\subset \Sigma(J^2)$.\par 
We omit this case since $V_{yr}$ is isomorphic to $V_{xt}$ by the symmetry 
for $x$ and $y$. \\
$({\rm iv})$ On the open set $V_{rt}\subset \Sigma(J^2)$.\par 
We will consider singular solutions of corank 2 which are the projections of 
integral manifolds of $\Sigma(J^2)$ passing through $\Sigma_2$.
On $V_{rt}$, we consider the submanifold 
$\overline \Sigma_{rt}=\left\{s=E=F=0\right\}.$ 
On $\overline \Sigma_{rt}$, we have the induced differential system:  
\[D_{\overline \Sigma_{rt}}=\left\{\varpi_{0}=\varpi_{1}=
\varpi_{2}=\varpi_{x}=\varpi_{y}=0\right\},\] 
where $\varpi_{x}=dx-Adr,\ \varpi_{y}=dy-Ddt.$   
Note that $w\in \Sigma_2 \iff A(w)=D(w)=0$.
This system $(\overline \Sigma_{rt}, D_{\overline \Sigma_{rt}})$ 
is isomorphic to 
$(P_{rt}, \hat D)\subset (\Sigma(R), \hat D)$. We construct integral 
manifolds of this system in the following.  
Let 
$\iota:S \hookrightarrow \overline \Sigma_{rt}\subset \Sigma(J^2)$
be a graph defined by 
$$(x(r,t),\ y(r,t),\ z(r,t),\ p(r,t),\ q(r,t),\ r,\ t,\ A(r,t),\ D(r,t)) 
\ \textrm{around}\ (r_0,t_0).$$
If $S$ is an integral submanifold of $(\overline \Sigma_{rt}, 
D_{\overline \Sigma_{rt}})$, 
then the following conditions are 
satisfied: 
\begin{align}
&\iota^{*}\varpi_{0}=
\iota^{*}(dz-pdx-qdy)=(z_r-px_r-qy_r)dr+(z_t-px_t-qy_t)dt=0,\label{wave6} \\
&\iota^{*}\varpi_{1}=
\iota^{*}(dp-rdx)=(p_r-rx_r)dr+(p_{t}-rx_t)dt=0,\label{wave7}\\
&\iota^{*}\varpi_{2}=
\iota^{*}(dq-tdy)=(q_r-ty_r)dr+(q_t-ty_t)dt=0,\label{wave8}\\
&\iota^{*}\varpi_{y}=
\iota^{*}(dx-Adr)= (x_r-A)dr+x_tdt=0,\label{wave9}\\
&\iota^{*}\varpi_{r}=
\iota^{*}(dy-Ddt)=y_{r}dr+(y_{t}-D)dt=0.\label{wave10}
\end{align}
From $(\ref{wave10})$, we have $y(r,t)=y(t),\ D(x,t)=y^\prime(t)$.
From $(\ref{wave9})$, we have $x(r,t)=x(r),\ A(x,t)=x^\prime(r)$.
From $(\ref{wave8})$, $q=\int ty^\prime dt =ty-Y$ where $Y:=\int ydt.$
From $(\ref{wave7})$, $p=\int rx^\prime dr =rx-X$ where $X:=\int xdr.$
From $(\ref{wave6})$, 
$
z=\frac{1}{2}\left(rx^2+ty^2+\int x^2 dr+\int y^2dt  \right)
-\left(x\int x dr+y\int ydt \right).
$
Hence, we get the solution of $s=0$ around $(x_0,t_0)$ on $U_{rt}$ given by 
\begin{eqnarray}
(x(r,t),\ y(r,t),\ z(r,t),\ p(r,t),\ q(r,t),\ r,\ t,\ A(r,t),\ D(r,t))
\hspace{4cm} \nonumber \\
 =\left( x(r),\ y(t),\ \frac{1}{2}\left(rx^2+ty^2+\int x^2 dr+\int y^2dt  \right)-\left(x\int x dr+y\int ydt \right)\right. \hspace{1cm} \nonumber \\
\hspace{7cm} \left.,\ rx-\int xdr,\ ty-\int ydt,\ r,\ t,\ x^\prime(r),\ y^\prime(t)
\right) . 
\nonumber 
\end{eqnarray}
for arbitrary functions $x(r)$ and $y(t)$.
These integral surfaces with the condition 
$x^\prime (r_0)=y^\prime (t_0)=0$ are geometric singular solutions of 
corank 2. \par
\medskip\noindent{\bf Singular solutions of a parabolic equation.} 
We consider the equation $R=\left\{r=0\right\}$. 
The differential system $D=\left\{\ \varpi_{0}=\varpi_{1}=\varpi_{2}=0\right\}$ is given by  
$\varpi_{0}=dz-pdx-qdy,\ 
\varpi_{1}=dp-sdy,\ 
\varpi_{2}=dq-sdx-tdy$.   
The structure equation of $D$ is written as   
\begin{align*}
d\varpi_{0}= -dp\wedge dx-dq\wedge dy,\quad 
d\varpi_{1}= -ds\wedge dy,\quad
d\varpi_{2}= -ds\wedge dx-dt\wedge dy.  
\end{align*}
Let $U$ be an open set in $R$. We have the covering 
$p^{-1}(U)=P_{xy}\cup P_{xt}\cup P_{st}$ 
of the fibration $p:\Sigma(R)\to R$ from Lemma \ref{par-lem}, 
where  
$U_{xy}:=\left\{v \in \pi^{-1}(U)\ |\ dx|_{v}\wedge dy|_{v}\not= 0\right\}$,
$U_{xt}:=\left\{v \in \pi^{-1}(U)\ |\ dx|_{v}\wedge dt|_{v}\not= 0\right\},\   
U_{st}:=\left\{v \in \pi^{-1}(U)\ |\ ds|_{v}\wedge dt|_{v}\not= 0\right\},\    
P_{xy}:=p^{-1}(U) \cap U_{xy},\ 
P_{xt}:=p^{-1}(U) \cap U_{xt},\  
P_{st}:=p^{-1}(U) \cap U_{st}$.  
The geometric decomposition $\Sigma(R)=\Sigma_{0}\cup\Sigma_{1}\cup\Sigma_{2}$ 
is given by 
$\Sigma_{0}|_{p^{-1}(U)}=P_{xy},\quad 
\Sigma_{1}|_{p^{-1}(U)}=P_{xt} 
\backslash P_{xy},\quad 
\Sigma_{2}|_{p^{-1}(U)}=P_{st}\backslash 
(P_{xy}\cup P_{xt})$.  
This prolongation $\Sigma(R)$ is realized as a submanifold of 
$\Sigma(J^2)$ as follows:\\ 
$({\rm i})$ On the open set $J^3=V_{xy}\subset \Sigma(J^2)$.\par 
On $V_{xy}$ in $\Sigma(J^2)$, we consider the submanifold given by   
$\overline \Sigma_{xy}=\left\{r=p_{111}=p_{112}=0\right\}.$ 
We have the induced differential system 
$D_{\overline \Sigma_{xy}}=\left\{\varpi_{0}=\varpi_{1}=\varpi_{2}=
\varpi_{s}=\varpi_{t}=0\right\}$
on $\overline \Sigma_{xy}$,    
where $\varpi_{s}=ds-p_{122}dy,\ \varpi_{t}=dt-p_{122}dx-p_{222}dy$.  
This system $(\overline \Sigma_{xy}, D_{\overline \Sigma_{xy}})$ is 
isomorphic to 
$(P_{xy}, \hat D)\subset (\Sigma(R), \hat D)$.  
Indeed, this system is equal to the third order PDE which is obtained by 
partial derivation of the original equation $r=0$ for the independent 
variables $x,y$. The projection to $R$ of these integral manifolds are regular solutions of the equation $r=0$. \\ 
$({\rm ii})$ On the open set $V_{st}\subset \Sigma(J^2)$.\par  
We will consider singular solutions of corank 1 and 2 
which are obtained by the projections of 
integral manifolds of $\Sigma(J^2)$ passing through smooth points 
in $\Sigma(R)$.  
Recall that $\Sigma_1\backslash \{\textrm{singular\ points} \} \subset 
\Sigma(R)$ is covered by $P_{st}$. Hence, we work on $V_{st}$ and consider 
the submanifold given by  
$\overline \Sigma_{st}=\left\{r=E=F=0\right\}.$ 
We have the induced differential system    
$D_{\overline \Sigma_{st}}=\left\{\varpi_{0}=\varpi_{1}=\varpi_{2}
=\varpi_{x}=\varpi_{y}=0\right\}$ on ${\overline \Sigma_{st}}$, 
where $\varpi_{x}=dx-Ads-Bdt,\ \varpi_{y}=dy-Bds.$ 
Note that 
\begin{eqnarray*}
 w\in \Sigma_1\backslash \{\textrm{singular\ points} \} 
& \iff & A(w)\not= 0 , B(w)=0  \\
 w\in \Sigma_2 \hspace{3.3cm} & \iff & A(w)=B(w)=0.
\end{eqnarray*}
This system $(\overline \Sigma_{st}, D_{\overline \Sigma_{st}})$ is isomorphic to 
$(P_{st}, \hat D)\subset (\Sigma(R), \hat D)$. 
We construct integral manifolds of this system. 
Let  
$\iota:S \hookrightarrow \overline \Sigma_{st}\subset \Sigma(J^2)$ be a graph 
defined by 
$$(x(s,t),\ y(s,t),\ z(s,t),\ p(s,t),\ q(s,t),\ s,\ t,\ A(s,t),\ B(x,t))
\ \textrm{around}\ (s_0,t_0).$$ 
If $S$ is an integral manifold of $D_{\overline \Sigma_{st}}$, then the following 
conditions are satisfied:   
\begin{align}
\iota^{*}\varpi_{0}:=&(z_s-px_s-qy_s)ds+(z_t-px_t-qy_t)dt=0,\label{heat1} \\
\iota^{*}\varpi_{1}:=&(p_s-sy_s)ds+(p_{t}-sy_{t})dt=0,\label{heat2}\\
\iota^{*}\varpi_{2}:=&(q_s-sx_s-ty_s)ds+(q_t-sx_t-ty_t)dt=0,\label{heat3}\\
\iota^{*}\varpi_{x}:=&(x_{s}-A)ds+(x_t-B)dt=0,\label{heat4}\\
\iota^{*}\varpi_{y}:=&(y_{s}-B)ds+y_{t}dt=0.\label{heat5}
\end{align}
We have $y(s,t)=y(s),\ B(s,t)=y^\prime (s)$ from $(\ref{heat5})$.  
From $(\ref{heat4})$, $x=ty^\prime(s)+x_{0}(s)$, where $x_0(s)$ is a function on $S$ depending only $s$, and $A=x_s=ty^{\prime\prime}(s)+x_{0}^\prime (s)$. 
From $(\ref{heat2})$, $p=\int sy_sds=sy-Y$ where $Y:=\int yds$.    
From $(\ref{heat3})$,  we also have $q=tsy^\prime +sx_0-\int x_0 ds$. 
Similarly, from $(\ref{heat1})$, $z=t(sy-Y)y^\prime +syx_0 +\int (yx_0)ds 
-x_0\int yds-y\int x_0ds$. 
Hence we have solutions of $r=0$ given by 
\begin{eqnarray*}
(x(s,t),\ y(s,t),\ z(s,t),\ p(s,t),\ q(s,t),\ s,\ t,\ A(s,t),\ B(s,t)) 
\hspace{3cm} \nonumber \\
 =(ty^\prime(s)+x_{0}(s),\ y(s),\ t(sy-\int yds)y^\prime +syx_0 +\int (yx_0)ds 
-x_0\int yds-y\int x_0ds, \nonumber \\ 
\hspace{4cm} sy-\int yds,\ tsy^\prime +sx_0-\int x_0 ds,\ s,\ t,\ ty^{\prime\prime}+x_{0}^\prime ,\ y^\prime ). \nonumber 
\end{eqnarray*}
for arbitrary functions $y(s)$ and $x_0(s)$.
These integral surfaces which satisfy the condition 
$$
A(s_0,t_0)=t_0 y^{\prime\prime}(s_0)+x_{0}^\prime(s_0) \not=0\ ,\  
 B(s_0)=y^{\prime}(s_0)=0
$$
are geometric singular solutions of corank 1. On the other hand,  
these integral with the condition 
$$
A(s_0,t_0)=t_0 y^{\prime\prime}(s_0)+x_{0}^\prime(s_0) =0\ ,\  
 B(s_0)=y^{\prime}(s_0)=0
$$
are geometric singular solutions of corank 2. \par
\medskip\noindent{\bf Singular solutions of an elliptic equation.} 
We consider the Laplace equation $R=\left\{r+t=0\right\}$. 
The differential system $D=\left\{\varpi_{0}=\varpi_{1}=\varpi_{2}=0\right\}$ is given by  
$\varpi_{0}=dz-pdx-qdy,\ 
\varpi_{1}=dp-rdx-sdy,\  
\varpi_{2}=dq-sdx+rdy$.
The structure equation of $D$ is  expressed as 
\begin{align*}
d\varpi_{0}= -dp\wedge dx-dq\wedge dy,\  
d\varpi_{1}= -dr\wedge dx-ds\wedge dy,\  
d\varpi_{2}= -ds\wedge dx+dr\wedge dy. 
\end{align*}
Then, for an open set $U \subset R$, we have the covering 
$p^{-1}(U)=P_{xy}\cup P_{rs}$ of the fibration $p:\Sigma(R)\to R$, 
where  
\begin{align*}
U_{xy}:&=\left\{v \in \pi^{-1}(U)\ |\ dx|_{v}\wedge dy|_{v}\not= 0\right\},\  
U_{rs}:=\left\{v \in \pi^{-1}(U)\ |\ dr|_{v}\wedge ds|_{v}\not= 0\right\},\\ 
P_{xy}:&=p^{-1}(U) \cap U_{xy},\  
P_{rs}:=p^{-1}(U) \cap U_{rs}.
\end{align*}
The geometric decomposition $\Sigma(R)=\Sigma_{0}\cup\Sigma_{2}$ 
is given by 
$\Sigma_{0}|_{p^{-1}(U)}=P_{xy},\ 
\Sigma_{2}|_{p^{-1}(U)}=P_{rs}\backslash P_{xy}$. 
This prolongation $\Sigma(R)$ is realized as a submanifold of 
$\Sigma(J^2)$ as follows:\\ 
$({\rm i})$ $J^3=V_{xy}\subset \Sigma(J^2)$.\par 
On $V_{xy}$, we consider the submanifold given by \\ 
$\overline \Sigma_{xy}=\left\{r+t=0,\ 
p_{111}=-p_{122},\ p_{112}=-p_{222}\right\}.$
We have the induced differential system    
$D_{\overline \Sigma_{xy}}=\left\{\varpi_{0}=\varpi_{1}=
\varpi_{2}=\varpi_{r}=\varpi_{s}=0\right\}$ on $\overline \Sigma_{xy}$, 
where $\varpi_{r}=dr-p_{111}dx-p_{112}dy, \ \varpi_{s}=ds-p_{112}dx+p_{111}dy.$   
This system $(\overline \Sigma_{xy}, 
D_{\overline \Sigma_{xy}})$ is isomorphic to 
$(P_{xy}, \hat D)\subset (\Sigma(R), \hat D)$.  
Indeed, this system is equal to the third order PDE which is obtained by 
partial derivation of the original equation $r+t=0$ for the independent 
variables $x,y$. The projection to $R$ of integral manifolds are 
regular solutions of the wave equation $r+t=0$. \\ 
$({\rm ii})$ On $V_{rs}\subset \Sigma(J^2).$\par  
We will consider singular solutions of corank 2 
which are the projections of 
integral manifolds of $\Sigma(J^2)$ passing through $\Sigma_2$.  
On $V_{rs}$, we consider the submanifold given by  
$\overline \Sigma_{rs}=\left\{r+t=0, E=-1, F=0\right\}.$ 
We have the induced differential system    
$D_{\overline \Sigma_{rs}}=\left\{\varpi_{0}=\varpi_{1}=
\varpi_{2}=\varpi_{x}=\varpi_{y}=0\right\}$ on $\overline \Sigma_{rs}$, 
where $\varpi_{x}=dx+Ddr-Bds,\ \varpi_{y}=dy-Bdr-Dds.$    
Recall that $w\in \Sigma_2 \iff B(w)=D(w)=0$.
This system $(\overline \Sigma_{rs}, D_{\overline \Sigma_{rs}})$ is isomorphic to 
$(P_{rs}, \hat D)\subset (\Sigma(R), \hat D)$. 
We construct integral manifolds of this system. 
Let  
$\iota:S \hookrightarrow \overline \Sigma_{rs}\subset \Sigma(J^2)$ be a graph 
defined by 
$$(x(r,s),y(r,s),z(r,s),p(r,s),q(r,s),r,s,B(r,s),D(r,s))
\ \textrm{around}\ (r_0,s_0).$$ 
If $S$ is an integral manifold of $D_{\overline \Sigma_{rs}}$, then the following 
conditions are satisfied:   
\begin{align}
\iota^{*}\varpi_{0}:=&(z_r-px_r-qy_r)dr+(z_s-px_s-qy_s)ds=0,\label{laplace1} \\
\iota^{*}\varpi_{1}:=&(p_r-rx_r-sy_r)dr+(p_s-rx_s-sy_s)ds=0,\label{laplace2}\\
\iota^{*}\varpi_{2}:=&(q_r-sx_r+ry_r)dr+(q_s-sx_s+ry_s)ds=0,\label{laplace3}\\
\iota^{*}\varpi_{x}:=&(x_r+D)dr+(x_s-B)ds=0,\label{laplace4}\\
\iota^{*}\varpi_{y}:=&(y_{r}-B)dr+(y_s-D)ds=0.\label{laplace5}
\end{align}
From $(\ref{laplace4})$ and $(\ref{laplace5})$, a complex function 
$f(z):=y(r,s)+ix(r,s) \ (z:=r+is)$ must be a holomorphic function.
From $(\ref{laplace2})$ and $(\ref{laplace3})$, $p(r,s), q(r,s)$ are 
considered as solutions of a differential equation
\begin{eqnarray}
& & q_s=sx_s-ry_s\ ,\ p_r=rx_r+sy_r\label{laplace6}\\ 
& & q_r=sx_r-ry_r\ ,\ p_s=rx_s+sy_s\label{laplace7} .
\end{eqnarray}
for given functions $x_r=-y_s, y_r=x_s$.
Then, we also get Cauchy-Riemann equation $q_r=-p_s, \ q_s=p_r$ from 
Cauchy-Riemann equation for $y(r,s), x(r,s)$. Hence a complex function 
$g(z):=p(r,s)+iq(r,s) \ (z:=r+is)$ is also holomorphic.
From $(\ref{laplace1})$, $z(r,s)$ is considered as  a solution of a differential equation
\begin{equation}
z_r = px_r+qy_r\ ,\ z_s=px_s+qy_s \label{laplace8}\\
\end{equation}
for given functions $p, q, x_r=-y_s, y_r=x_s$.

Conversely, for a given holomorphic function $f(z)=y(r,s)+ix(r,s) \ (z:=r+is)$
we consider the differential equation $(\ref{laplace6}), (\ref{laplace7})$
 for $p, q$ where $x, y$ are given functions.
Then, the differential equation is Frobenius since $y(r,s)$ and $x(r,s)$ satisfy Cauchy-Riemann equation. Therefore, the existence of the solution of $(\ref{laplace6}), (\ref{laplace7})$ is guaranteed and $g(z):=p(r,s)+iq(r,s) \ (z:=r+is)$ is holomorphic. Next, we consider
the differential equation $(\ref{laplace8})$ for $z$ where $x, y, p, q$ are given. Then, this differential equation is Frobenius since $f(z)$ and $g(z)$ are 
holomorphic functions and have solutions.
Finally, let $f(z)=y(r,s)+ix(r,s)\ (z:=r+is)$ be a holomorphic function and $p(r,s), q(r,s), z(r,s)$ be the functions obtained by 
the above construction. Then, 
$$
(x(r,s),\ y(r,s),\ z(r,s),\ p(r,s),\ q(r,s),\ r,\ s,\ y_r(s,t),\ y_s(s,t))
$$
is a integral surface. These integral surfaces which satisfy the condition 
$y_r(s_0,t_0)=y_s(s_0,t_0)=0$ are geometric singular solutions of corank 2. 
\section{Tower constructions of special rank 4 distributions.}
In section 2 and 3, we studied geometric structures of rank 2 prolongations 
for each class of equations. 
In this section, we define special rank 4 distributions which 
are generalization of distributions 
induced by PDEs and construct tower structures 
of these distributions by successive rank 2 prolongations. 
\begin{definition}
Let $R$ be a $k+6$ dimensional manifold ($k\geq 0$), 
and $D$ be a differential system of rank $4$ on $R$. 
Then, 

(i)\ $(R,D)$ is {\it hyperbolic type} at $w\in R$ if there exists a 
local coframe $\left\{\varpi_{i}, \theta_{j}, \omega_{j}, \pi_{j}\right\}$ 
$(i=1,...,k,$  $j=1,2)$ around $w\in R$ such that  
$D=\{\varpi_{i}=\theta_{j}=0 \}$ 
around $w\in R$ and the following structure equation holds at $w$:
\begin{align}
d\varpi_{i}&\equiv 0 
\hspace{2cm} \mod \ \varpi_{i}, \theta_{j} \nonumber \\
d\theta_{1}&\equiv \omega_{1}\wedge \pi_{1} \quad \quad \ \mod\ \varpi_{i}, \theta_{j},\\
d\theta_{2}&\equiv \omega_{2}\wedge \pi_{2} \quad \quad \ \mod\ \varpi_{i}, \theta_{j}. \nonumber
\end{align}
(ii)\ $(R,D)$ is {\it parabolic type} at $w\in R$ if there exists a 
local coframe $\left\{\varpi_{i}, \theta_{j}, \omega_{j}, \pi_{j}\right\}$ 
$(i=1,..,k,\ j=1,2)$ around $w\in R$ such that 
$D=\{\varpi_{i}=\theta_{j}=0 \}$ 
around $w\in R$ and the following structure equation holds at $w$:
\begin{align}
d\varpi_{i}&\equiv 0 
\hspace{3.3cm} \mod \ \varpi_{i}, \theta_{j} \nonumber \\
d\theta_{1}&\equiv \quad \quad \quad \quad \ \omega_{2}\wedge \pi_{1} \quad \mod\ \varpi_{i}, \theta_{j},\\
d\theta_{2}&\equiv \omega_{1}\wedge \pi_{1}+\omega_{2}\wedge \pi_{2} \quad \mod\ \varpi_{i}, \theta_{j}. \nonumber
\end{align}
(iii)\ $(R,D)$ is {\it elliptic type} at $w\in R$ if there exists a 
local coframe $\left\{\varpi_{i}, \theta_{j}, \omega_{j}, \pi_{j}\right\}$ 
$(i=1,..,k,\ j=1,2)$ around $w\in R$ such that 
$D=\{\varpi_{i}=\theta_{j}=0 \}$ 
around $w\in R$ and the following structure equation holds at $w$:
\begin{align}
d\varpi_{i}&\equiv 0 
\hspace{3.3cm} \mod \ \varpi_{i}, \theta_{j} \nonumber \\
d\theta_{1}&\equiv \omega_{1}\wedge \pi_{1}+\omega_{2}\wedge \pi_{2} \quad \mod\ \varpi_{i}, \theta_{j},\\
d\theta_{2}&\equiv \omega_{1}\wedge \pi_{2}-\omega_{2}\wedge \pi_{1} \quad \mod\ \varpi_{i}, \theta_{j}. \nonumber
\end{align}
\end{definition}
\begin{proposition}\label{rank4} 
Let $(R,D)$ be a hyperbolic type, parabolic type or elliptic type. Then 
 the first derived system $\partial D$ of $D$ is a 
subbundle of rank $6$ and the Cauchy characteristic system 
$Ch(D)$ of $D$ is trivial, that is $Ch(D)=\left\{0\right\}$. 
\end{proposition}
\begin{proof}
This statement is obtained by the very definitions. 
\end{proof} 
\begin{remark}
In fact, the converse of the above proposition also holds. Namely, 
let $D$ be a differential system of rank $4$ on a 
$k+6$ dimensional manifold $R$ with rank $\partial D=6$, $Ch(D)=\{ 0\}$. Then, for any $w\in R$, $(R,D)$ is 
a hyperbolic type, parabolic type or elliptic type at $w$ (\cite{S2}).
\end{remark}
\begin{proposition}\label{tower} \ 
\begin{enumerate}
\item[$(${\rm i}$)$] 
If $(R,D)$ is locally hyperbolic, then the rank $2$ prolongation 
$(\Sigma(R),\hat D)$ of 
$(R,D)$ is also hyperbolic at any point in $\Sigma(R)$. Moreover, $\Sigma(R)$ is a $T^2$-bundle over $R$.
\item[$(${\rm ii}$)$] If $(R,D)$ is locally parabolic, then 
$(\Sigma(R)\backslash \left\{singular\ points \right\},\hat D)$ is also parabolic at any point in $\Sigma(R)\backslash \left\{singular\ points \right\}$. 
Moreover, $\Sigma(R)\backslash \left\{singular\ points \right\}$ is a $S^1 \times \mathbb{R}$-bundle over $R$. 
\item[$(${\rm iii}$)$]
If $(R,D)$ is locally elliptic, then the rank $2$ prolongation $(\Sigma(R),\hat D)$ of 
$(R,D)$ is also elliptic at any point in $\Sigma(R)$. Moreover, $\Sigma(R)$ is a $S^2$-bundle over $R$.
\end{enumerate}  
\end{proposition}
\begin{proof}
These statements are obtained by the same arguments of the proof of Theorem \ref{hyp-pro}, Proposition \ref{hyp-symbol1}, \ref{hyp-symbol2} for the hyperbolic case, Theorem \ref{par-topology}, Proposition \ref{par-symbol} 
for the parabolic case and Theorem \ref{ell-topology}, Proposition \ref{ell-symbol} for the elliptic case. 
\end{proof} 

For the locally hyperbolic, locally parabolic or locally elliptic type distribution $(R,D)$, we can define $k$-th rank 2 prolongation $(\Sigma^{k}(R),\hat D^{k})$ of $(R,D)$ by the above Proposition, successively. For hyperbolic 
and elliptic type $(R,D)$, we define
$$
(\Sigma^{k}(R),\hat D^{k}):=(\Sigma(\Sigma^{k-1}(R)),\hat{\hat {D}}^{k-1})
\qquad (k=1,2,\ldots ),
$$
where $(\Sigma^{0}(R),\hat D^{0}):=(R,D)$. For parabolic type $(R,D)$, we define 
$$
(\Sigma^{k}(R),\hat D^{k}):=(\Sigma(\Sigma^{k-1}(R))\backslash \left\{singular\ points \right\},\hat{\hat {D}}^{k-1}) \qquad (k=1,2,\ldots )
$$
where $(\Sigma^{0}(R),\hat D^{0}):=(R,D)$. 

\begin{theorem}\label{towerthm}
If $(R,D)$ is locally hyperbolic, locally parabolic or locally elliptic then the $k$-th rank $2$ prolongation $(\Sigma^k(R), \hat D^{k})$ of $(R,D)$ is also hyperbolic, parabolic or elliptic at any point in $\Sigma^k(R)$, respectively.
\end{theorem}
\begin{proof}
This theorem is obtained from the successive applications of Proposition \ref{tower}.
\end{proof}

\begin{remark}\label{EDS}
For the hyperbolic case, Bryant, Griffiths and Hsu proved the above theorem 
for the exterior differential systems in \cite{BGH}. By our argument, for parabolic and elliptic cases, one can show that Theorem \ref{towerthm} have the similar extension for the exterior differential system (\cite{S2}).
\end{remark}

\end{document}